\documentclass[11pt,leqno]{article}
\usepackage{amsmath,amsthm,amsfonts,amssymb}

\newcommand{\CH}{\mathcal{H}}
\newcommand{\CG}{\mathcal{G}}
\newcommand{\CalB}{\mathcal{B}}
\newcommand{\CalD}{\mathcal{D}}
\newcommand{\CalE}{\mathcal{E}}
\newcommand{\CalW}{\mathcal{W}}
\newcommand{\CalA}{\mathcal{A}}
\newcommand{\CalG}{\mathcal{G}}

\newcommand{\E}{{\mathbb E}}

\newcommand{\eps}{{\varepsilon}}

\newcommand{\C}{{\mathbb C}}
\newcommand{\R}{{\mathbb R}}
\newcommand{\RR}{{\mathbb R}}

\newcommand{\NN}{{\mathbb N}}
\newcommand{\ZZ}{{\mathbb Z}}
\newcommand{\TT}{{\mathbb T}}
\newcommand{\CC}{{\mathbb C}}

\newcommand{\EE }{{\mathbb E}}

\newcommand{\II}{{\mathbb I}}
\newcommand{\DD}{{\mathbb D}}

\newcommand{\QQ}{{\mathbb Q}}

\newcommand{\MM}{{\mathbb M}}

\newcommand\cA{{\cal  A}}
\newcommand\cB{{\cal  B}}
\newcommand\cD{{\cal  D}}
\newcommand\cU{{\cal  U}}
\newcommand\cH{{\cal  H}}

\newcommand\cI{{\cal  I}}

\newcommand\cG{{\cal  G}}

\newcommand\cL{{\cal  L}}

\newcommand\cE{{\cal  E}}
\newcommand\cF{{\cal  F}}
\newcommand\cP{{\cal  P}}

\newcommand\cO{{\cal O}}

\newcommand\cM{{\mathcal M}}

\newcommand{\at}{{\alpha_t }}

\newcommand{\ar}{{\alpha_r }}

\newcommand{\bR}{\mathbb{R}}

\newcommand{\bC}{\mathbb{C}}
\newcommand{\bZ}{\mathbb{Z}}

\def\eps{\epsilon }

\newcommand\adots{\mathinner{\mkern2mu\raise1pt\hbox{.}
\mkern3mu\raise4pt\hbox{.}\mkern1mu\raise7pt\hbox{.}}}

\newtheorem{theo}{Theorem}[section]
\newtheorem{prop}[theo]{Proposition}

\newtheorem{lem}[theo]{Lemma}
\newtheorem{defn}[theo]{Definition}
\newtheorem{ass}[theo]{Assumption}

\newtheorem{rem}[theo]{Remark}

\newtheorem{nota}[theo]{Notations}
\newtheorem{assumption}[theo]{Assumption}
\newtheorem{definition}[theo]{Definition}

\numberwithin{equation}{section}

\begin{document}

\title{\textbf{Hyperbolic boundary problems  with large oscillatory coefficients:  multiple amplification}}
\author{Mark Williams}

\author{ 
Mark {\sc Williams}\thanks{University of North Carolina, Mathematics Department, CB 3250, Phillips Hall, 
Chapel Hill, NC 27599. USA. Email: {\tt williams@email.unc.edu}.}}

\maketitle
\begin{abstract}
 
 We study weakly stable hyperbolic boundary problems with highly oscillatory coefficients that are large, $O(1)$,  compared to the small wavelength $\eps$ of oscillations.  Such problems arise, for example, in the study of classical questions concerning the stability of  Mach stems and compressible vortex sheets.  For such applications one seeks to prove energy estimates that are in an appropriate sense ``uniform" with respect to the small wavelength $\eps$, but the large oscillatory coefficients  are a formidable obstacle to obtaining such estimates.   In this paper we analyze  a simplified form of the linearized problems that are relevant to the above stability questions, and obtain results that are both positive and negative.  On the one hand we identify favorable structural conditions under which it is possible to prove uniform estimates, and then do so by a new approach.  We also construct examples showing that  large oscillatory coefficients can give rise to an instantaneous \emph{multiple} amplification of the amplitude of solutions relative to data; for example, boundary data of a given amplitude $O(1)$ can \emph{immediately} give rise to a solution of amplitude $O(\frac{1}{\eps^K})$, where  $K>1$.\footnote{Examples of first-order amplification, where $K=1$, are well-known \cite{MA,CG}.}   We use the examples of multiple amplification  to confirm the optimality of our uniform estimates
 when the favorable structural conditions hold.   When those conditions do not hold,   we explain how  multiple amplification of infinite order may rule out useful estimates.

\end{abstract}

\tableofcontents

\section{Introduction}\label{intro}

\qquad We study linear hyperbolic boundary problems on $\Omega:=\RR_t\times \RR_{x_1}\times \{x_2:x_2\geq 0\}$ of the form
\begin{align}\label{i1}
 \begin{split}
 &L(\partial)u+\cD\left(\frac{\phi_{in}}{\eps}\right)u:=\partial_t u+B_1\partial_{x_1}u+B_2\partial_{x_2}u+\cD\left(\frac{\phi_{in}}{\eps}\right)u=F\left(t,x,\frac{\phi_0}{\eps}\right)\text{ in }x_2>0\\
 &Bu=G\left(t,x_1,\frac{\phi_0}{\eps}\right)\text{ on }x_2=0\\
 &u=0 \text{ in }t<0,
 \end{split}
 \end{align}
where the $B_j$  are  $N\times N$ constant matrices, $B_2$ is invertible, and the highly oscillatory matrix coefficient $\cD\left(\frac{\phi_{in}}{\eps}\right)$ is large, $O(1)$, compared to the small wavelength $\eps\in (0,\eps_0]$.\footnote{It is sometimes necessary to replace $(F,G)$ in \eqref{i1} by functions $(F^\eps,G^\eps)$ of the same arguments.}  We take the problem to be weakly stable in the sense that $(L(\partial),B)$ fails to satisfy the uniform Lopatinski condition (Definition \ref{ustable}) in a specific way (Assumption \ref{assumption3}).\footnote{Such problems are referred to as ``weakly real" (WR) problems in \cite{BRSZ}.}  The boundary matrix $B$ is a constant $p\times N$ matrix of appropriate rank $p$,  and the functions $F(t,x,\theta)$, $G(t,x_1,\theta)$ and $\cD(\theta_{in})$ are respectively periodic of period $2\pi$ in $\theta$ and $\theta_{in}$.     The boundary phase is taken to be $\phi_0(t,x_1)=\beta_l\cdot (t,x_1)$, where $\beta_l=(\sigma_l,\eta_l)\in\RR^2\setminus 0$ is one of the directions where the uniform Lopatinski condition fails.   The interior phase 
  \begin{align}\label{i2}
  \phi_{in}(t,x)=\phi_0(t,x_1)+\omega_N(\beta_l) x_2
  \end{align}
  is one of the incoming phases  (Definition \ref{def2}); its restriction to $x_2=0$ is $\phi_0$. 
  For convenience we take $F$ and $G$ to be zero in $t<0$.   

The problem \eqref{i1} is a simplified form of linearized problems that arise, for example, in the study of stability of Mach stems and vortex sheets.   In trying to rigorously justify the classical mechanisms for Mach stem formation and vortex sheet roll-up proposed in \cite{MR} and \cite{AM}, one is led to study weakly stable linearized problems with large oscillatory coefficients of the form
\begin{align}\label{i3}
\begin{split}
&\partial_t u+\sum^2_{i=1}B_i\left(\eps v(t,x,\frac{\phi_{in}}{\eps})\right)\partial_{x_i}u+D\left(v(t,x,\frac{\phi_{in}}{\eps})\right)u=F\left(t,x,\frac{\phi_0}{\eps}\right)\text{ in }x_2>0\\
&B\left(\eps v(t,x,\frac{\phi_{0}}{\eps})\right)u=G\left(t,x_1,\frac{\phi_0}{\eps}\right) \text{ on }x_2=0\\
&v=0\text{ in }t<0,
\end{split}
\end{align}
where $\eps\in (0,\eps_0]$ and  $v,F,G$ are periodic in their third arguments.  The mechanisms in question are exhibited by approximate WKB  (or geometric optics) solutions to the original quasilinear equations.  One approach to showing that such approximate solutions are close to true exact solutions (i.e., ``justifying" the geometric optics solutions) depends essentially on first proving energy estimates for \eqref{i3} that are in an appropriate sense  \emph{uniform}   with respect to small $\eps$ (Remark \ref{unif}).   Such estimates do not exist at present.  Methods that have been used successfully on other hyperbolic boundary problems fail to yield uniform estimates for  \eqref{i3}.\footnote{Here we have in mind: (a) problems where the uniform Lopatinski condition holds \cite{K,CGW1};  (b) weakly stable problems like \eqref{i3} but with non-oscillatory coefficients \cite{C1,C,CS};  and  (c) weakly  stable problems like \eqref{i3} where the oscillatory function $v$  is replaced by $\eps v$ in the arguments of $B_i$, $D$, and $B$ \cite{CGW3}.}  In fact, those methods fail even when applied to the simplified problem \eqref{i1}; however, the approach of \cite{CGW3} works  if $\cD$ is replaced by $\eps \cD$ in \eqref{i1}.   The main obstacle, roughly, is that the estimates exhibit a loss of one derivative from solution $u$ to data $(F,G)$,  and the available methods produce error terms too large when $\eps$ is small to be absorbed by the left sides of the estimates.    It is not yet clear whether the obstacle can be overcome with better methods or is truly insurmountable. 

We note that even when $\cD=0$ in \eqref{i1}, one observes both a loss of derivatives in the energy estimates and an associated phenomenon of (first-order) \emph{amplification};  data $(F,G)$ of size $O(1)$ generally gives rise to a solution of size $O(\frac{1}{\eps})$ \cite{CG}.


It is partly in order to assess the feasibility of proving uniform estimates for problems like \eqref{i3},  estimates sufficiently strong to justify geometric optics solutions of the original nonlinear equations,  that we study here the problem of  proving uniform estimates for \eqref{i1}.\footnote{Section \ref{disc} gives a preliminary assessment.}   Also, we regard this question as a natural question in the linear hyperbolic theory worth studying in its own right.  

In \cite{W3} we studied this question for the problem \eqref{i1} in the special case where $\cD(\theta_{in})=e^{i\theta_{in}}I_{N\times N}$.    This choice leads to the simplest one-sided cascades (defined below).  In this paper we consider general (sufficiently regular) oscillatory $N\times N$ coefficients 
\begin{align}\label{i4}
\cD(\theta_{in})=(d_{i,j}(\theta_{in}))_{i,j=1,\dots,N},
\end{align}
which produce one-sided cascades when $\cD(\theta_{in})$ has  only positive Fourier spectrum, and two-sided cascades when  $\cD(\theta_{in})$  has both positive and negative Fourier spectrum.   
Here, in addition to proving optimal estimates for certain problems with two-sided cascades, we substantially refine and extend the methods introduced in \cite{W3} to obtain estimates that are optimal in the one-sided case.   Consequently,  the  ``amplification exponent" $\EE$ that appears in the estimates of Theorem \ref{tt26} is generally much smaller than the corresponding exponent in the estimates of \cite{W3}.  The optimality of $\EE$ is verified in section \ref{multiple} by the construction (and justification) of approximate geometric optics solutions whose amplitudes  exhibit exactly the maximum order of  multiple amplification relative to the data  permitted by the estimate (Remark \ref{opt}).

\subsection{Iteration estimate}
\qquad  Here we describe the first main step in the analysis, which is to prove an \emph{iteration estimate} for an associated singular system.  The same estimate is used for the cases of one and two-sided cascades.  The only restriction on $\cD(\theta_{in})$ for the moment is that its Fourier spectrum is contained in $\ZZ\setminus 0$.\footnote{See Remark \ref{ts1}(4) for the case where $\cD$ has nonzero mean.     Also, part (5) of that remark considers what happens when $\phi_{out}$ is used instead of $\phi_{in}$ in \eqref{i1}.}

 Let us first rewrite \eqref{i1} with slightly modified $F$ as
\begin{align}\label{i5}
\begin{split}
&D_{x_2} u+A_0 D_tu+A_1D_{x_1}u-iB_2^{-1}\cD\left(\frac{\phi_{in}}{\eps}\right)u=F(t,x,\frac{\phi_0}{\eps})\\
&Bu= G(t,x_1,\frac{\phi_0}{\eps})\text{ on }x_2=0\\
&u=0\text{ in }t<0,
\end{split}
\end{align}
where $A_0=B_2^{-1}$, $A_1=B_2^{-1}B_1$, and $D_{x_i}=\frac{1}{i}\partial_{x_i}$.  We study \eqref{i5} by   looking for a  solution  of the form $u(t,x)=U(t,x,\frac{\phi_0}{\eps})$, where $U(t,x,\theta)$ is periodic in $\theta$.   This yields in the obvious way what we refer to as the corresponding  \emph{singular system} for $U$: \footnote{Clearly, $u=u^\eps$ and $U=U^\eps$ depend on $\eps$, but we usually suppress  $\eps$-dependence in the notation for these and other functions.      Singular systems were used in \cite{JMR1} for initial value problems in free space.}
\begin{align}\label{i6}
\begin{split}
&D_{x_2}U+A_0(D_t+\frac{\sigma_l}{\eps}D_\theta)U+A_1(D_{x_1}+\frac{\eta_l}{\eps}D_\theta)U-iB_2^{-1}\cD\left(\frac{\omega_N(\beta_l)}{\eps}x_2+\theta\right)U=F(t,x,\theta)\\
&BU=G(t,x_1,\theta)\text{ on }x_2=0\\
&U=0\text{ in }t<0.
\end{split}
\end{align}
The matrix $\cD(\theta_{in})$ can be written 
\begin{align}\label{i7}
\cD(\theta_{in})=\sum^N_{i,j=1}d_{i,j}(\theta_{in})M_{i,j}, \text{ where }d_{i,j}(\theta_{in})=\sum_{r\in\ZZ\setminus 0}\alpha^{i,j}_re^{ir\theta_{in}}
\end{align}
and $M_{i,j}$ is the $N\times N$ matrix with $(i,j)$ entry equal to 1 and all other entries $0$.   Using \eqref{i7} we will see (Remark \ref{reduction}) that we can reduce the study of \eqref{i6} to the study of the same problem with $\cD(\theta_{in})$ replaced by $d(\theta_{in})M$, where $M$ is any constant $N\times N$ matrix and $d(\theta_{in})$ is a scalar periodic function
\begin{align}\label{i7a}
d(\theta_{in})=\sum_{r\in\ZZ\setminus 0}\alpha_re^{ir\theta_{in}}.
\end{align}
The first equation of  \eqref{i6}  with  $\cD(\theta_{in})$ replaced by $d(\theta_{in})M$ can then be written
\begin{align}
D_{x_2}U+A_0(D_t+\frac{\sigma_l}{\eps}D_\theta)U+A_1(D_{x_1}+\frac{\eta_l}{\eps}D_\theta)U-i\left(\sum_{r\in\ZZ\setminus 0} \ar e^{i\left(r\frac{\omega_N(\beta_l)}{\eps}x_2+r\theta\right)}\right)B_2^{-1}MU=F(t,x,\theta).
\end{align}

Next we consider the Laplace-Fourier transform in $(t,x_1,\theta)$ of the singular system.  
Expand $U(t,x,\theta)=\sum_{k\in\bZ}U_k(t,x)e^{ik\theta}$, expand $F$ and $G$ similarly, set
\begin{align}\label{i8}
\zeta:=(\tau,\eta):=(\sigma-i\gamma,\eta), \text{ where }(\sigma,\eta)\in\RR^2,\; \gamma\geq 0,
\end{align}
and define  $V_k(x_2,\zeta):=\widehat{U_k}(\zeta,x_2)$, the Laplace-Fourier transform in $(t,x_1)$ of $U_k(t,x)$.  If we 
define 
\begin{align}
X_k:=\zeta+\frac{k\beta_l}{\eps}\text{ and }\cA(\zeta)=-(A_0\tau+A_1\eta),
\end{align}
we can write the transformed singular problem for the $V_k$ as : 
\begin{align}\label{i9}
\begin{split}
&D_{x_2}V_k-\mathcal{A}(X_k)V_k=i\sum_{r\in\ZZ\setminus 0}  \ar e^{ir\frac{\omega_N(\beta_l)}{\eps}x_2}B_2^{-1}MV_{k-r}+\widehat{F_k}(x_2,\zeta)\\
&BV_k=\widehat{G_k}(\zeta)\text{ on }x_2=0.
\end{split}
\end{align}

The \emph{iteration estimate}, proved in Proposition \ref{tt30},  is an estimate valid for $\gamma\geq \gamma_0>0$ of the form
\begin{align}\label{i10}
\|V_k\|\leq  \frac{C}{\gamma}\sum_{r\in\ZZ\setminus 0}\sum_{t\in\ZZ}\|\ar\at \DD(\eps,k,k-r)V_{k-r-t}\|+\frac{C}{\gamma^2}\left|\widehat{F_k}|X_k|\right|_{L^2(x_2,\sigma,\eta)}+\frac{C}{\gamma^{3/2}}\left|\widehat{G_k}|X_k|\right|_{L^2(\sigma,\eta)}.
\end{align}
Here $\|V_k\|$, defined in \eqref{tu30},  is a modified $L^2(x_2,\sigma,\eta)$ norm,  the constants $C$ and $\gamma_0$ are independent of $(\eps,\zeta,k)$, and the $\alpha_r$ are as in \eqref{i9} (we redefine $\alpha_0$ to be $1$ in \eqref{i10}).

It is far from clear at this point that an estimate of the form \eqref{i10} is useful.   That depends on the behavior of the \emph{global amplification factors} $\DD(\eps,k,k-r)$, Definition \ref{t29}.\footnote{Of course, this also depends on having the coefficients $\alpha_r$ decay sufficiently rapidly with respect to $r$. That is the case provided $d(\theta_{in})$ is sufficiently regular, which we always assume.}   One expects these factors, which are functions of $\zeta$, to depend on $(\eps,k,r)$ and to be \emph{large} sometimes, that is $\gtrsim \frac{1}{\eps}$, because of the failure of the uniform Lopatinski condition.   
In sections \ref{afactors} and \ref{ptools} we show that the $\DD(\eps,k,k-r)$ can be chosen so that for each $\zeta$:
\begin{align}\label{i11}
 \DD(\eps,k,k-r)(\zeta) \text{ takes one of three values: }1, C_5r^2, \text{ or } \frac{C_5r^2}{\eps\gamma},
\end{align}
where $C_5$ is a fixed positive constant.\footnote{As $\zeta$ varies, the value taken can vary.}  In particular, the $\DD(\eps,k,k-r)$ are independent of $k$!    The factors $r^2$ in \eqref{i11} are harmless; they are killed by the decay of the $\alpha_r$.   One must understand for a given choice of $(\eps,\zeta)$ how many of these factors can be large;   we say more about this below.

The norm $\|V_k\|$ satisfies 
\begin{align}\label{i12}
|(\|V_k\|)|_{\ell^2(k)}\gtrsim |e^{-\gamma t}U|_{L^2(t,x,\theta)}+\left|\frac{e^{-\gamma t}U(0)}{\sqrt{\gamma}}\right|_{L^2(t,x_1,\theta)}, 
\end{align}
for $U(t,x,\theta)$ as in \eqref{i6}. The second main step in the analysis is to use the iteration estimate to estimate $|(\|V_k\|)|_{\ell^2(k)}$ in terms of the data $(F,G)$, and this leads us to examine the cascades produced by iterating \eqref{i10}.

\subsection{One and two-sided cascades}
\qquad We first discuss the one-sided cascades that arise when the coefficients  $\alpha_r$ in \eqref{i9} vanish for $r\leq 0$, that is, when $d(\theta_{in})$ has only positive Fourier spectrum.\footnote{There is an exactly parallel theory when $d(\theta_{in})$ has  only negative Fourier spectrum.}
Let us assume for now that  $F=0$ and $G=\sum^\infty_{k=N^*}G_k(t,x_1)e^{ik\theta}$ for some $N^*\in \ZZ$; in fact, take $N^*=1$ for the moment.   Since $G_k=0$ for $k<1$, it follows from \eqref{i9} that $V_k=0$ for $k<1$, so the iteration estimate \eqref{i10} now reduces to
\begin{align}\label{i13}
\|V_k\|\leq \frac{C}{\gamma}\sum^{k-1}_{r=1}\sum_{t=0}^{k-r-1}\|\ar\alpha_t\DD(\eps,k,k-r)V_{k-r-t}\|+|\CalG_k|_{L^2(\sigma,\eta)}, \;k\geq 1,
\end{align}
where $\mathcal{G}_k=\frac{C}{\sqrt{\gamma}}\frac{\widehat{G_k}|X_k|}{\gamma}.$   For any $k\geq 1$ the indices $k-r-t$ that appear in nonzero terms on the right side of \eqref{i13}  satisfy $k-r-t<k$. 


We can  estimate the $V_k$ for $k\geq 1$ in terms of the $G_j$, $1\leq j\leq k$ by iterating the estimate \eqref{i13}.     
The iterations quickly become tedious to write down in detail, but let us do the first few.     Writing $|\CalG_j|_{L^2(\sigma,\eta)}$ simply as $|\CalG_j|$ 
we obtain:\footnote{When iterating the estimate \eqref{i13}, we repeatedly make use of the fact that functions of $\zeta$ like $\DD(\eps,p,p-r)$ commute right through the problem \eqref{i9}.}
\begin{align}\label{b001}
\begin{split}
&\|V_1\|\leq |\CalG_1|\\
&\|V_2\|\leq \frac{C}{\gamma}\|\alpha_1\DD(\eps,2,1)V_1\|+|\CalG_2|\leq \frac{C}{\gamma}|\alpha_1\DD(\eps,2,1)\CalG_1|+|\CalG_2|\\
&\|V_3\|\leq \frac{C}{\gamma}\|\alpha_1\DD(\eps,3,2)V_2\|+\frac{C}{\gamma}\|\alpha_1\DD(\eps,3,2)V_1\|+\frac{C}{\gamma}\|\alpha_2\DD(\eps,3,1)V_1\|+|\CalG_3|\leq\\
&\frac{C}{\gamma}\left[\frac{C}{\gamma}|\alpha_1\DD(\eps,3,2)\alpha_1\DD(\eps,2,1)\CalG_1|+|\alpha_1\DD(\eps,3,2)\CalG_2|\right]+\frac{C}{\gamma}|\alpha_1\DD(\eps,3,2)\CalG_1|+\frac{C}{\gamma}|\alpha_2\DD(\eps,3,1)\CalG_1|+|\CalG_3|.
\end{split}
\end{align}

It turns out that all the essential information about the  estimate of $V_k$  is contained in the cascade that arises in the obvious way by iteration of \eqref{i13}.    For example,  the cascade corresponding to the estimate of $V_3$ in \eqref{b001} is the following four-stage cascade:
\begin{align}\label{b002}
[(V_3)]\to[(V_2,V_1,V_1,\CalG_3)]\to [(V_1,\cG_2)(\cG_1)(\cG_1)\cG_3]\to [(\cG_1),\cG_2,\cG_1,\cG_1,\cG_3].
\end{align}
Here brackets $[\cdot ]$ mark the individual stages, and a $V_j$ or $\cG_j$ term is included in parentheses only when it makes its very first appearance in the cascade.

 When $k=5$ the following six-stage cascade can easily be written down \emph{without} first doing a detailed estimate of $V_5$ as in \eqref{b001}:
\begin{align}\label{b01}
\begin{split}
&[(V_5)]\to [(V_4,V_3,V_2,V_1,V_3,V_2,V_1,V_2,V_1,V_1,\cG_5)]\to\\
& [(V_3,V_2,V_1,V_2,V_1,V_1,\cG_4),(V_2,V_1,V_1,\cG_3),(V_1,\cG_2),(\cG_1),(V_2,V_1,V_1,\cG_3),\\
&\qquad (V_1,\cG_2),(\cG_1),(V_1,\cG_2),(\cG_1),(\cG_1),\cG_5]\to\\
& [(V_2,V_1,V_1,\cG_3),(V_1,\cG_2),(\cG_1),(V_1,\cG_2),(\cG_1),(\cG_1),\cG_4,(V_1,\cG_2),(\cG_1),\\
&\qquad (\cG_1),\cG_3,(\cG_1),\cG_2,(V_1,\cG_2),(\cG_1),(\cG_1),\cG_3,(\cG_1),\cG_2,\cG_1,(\cG_1),\cG_2,\cG_1,\cG_1,\cG_5]\to\\
&[(V_1,\cG_2),(\cG_1),(\cG_1),\cG_3,(\cG_1),\cG_2,\cG_1,(\cG_1),\cG_2,\cG_1,\cG_1,\cG_4,(\cG_1),\cG_2,\cG_1,\\
&\qquad \cG_1,\cG_3,\cG_1,\cG_2,(\cG_1),\cG_1,\cG_2,\cG_1,\cG_1,\cG_3,\cG_1,\cG_2,\cG_1,\cG_1,\cG_2,\cG_1,\cG_1,\cG_5]\to \\
& [(\cG_1),\cG_2,\cG_1,\cG_1,\cG_3,\cG_1,\cG_2,\cG_1,\cG_1,\cG_2,\cG_1,\cG_1,\cG_4,\cG_1,\cG_2,\cG_1,\\
&\qquad \cG_1,\cG_3,\cG_1,\cG_2,\cG_1,\cG_1,\cG_2,\cG_1,\cG_1,\cG_3,\cG_1,\cG_2,\cG_1,\cG_1,\cG_2,\cG_1,\cG_1,\cG_5]
\end{split}
\end{align}
In \eqref{b001} and \eqref{b01} the indices  that appear on the $V_j$ in the second and subsequent stages are less than the index that appears in the first stage (one-sided cascades).

In section \ref{2s} we provide a more efficient way of constructing cascades like \eqref{b01} and show how to reconstruct the estimate of $\|V_k\|$ in terms of the $|\cG_j|$ from the cascade.
As in \eqref{b001}, in the estimate of $\|V_k\|$ terms will occur on the right side in which products of up to $k-1$ global amplification factors appear:
\begin{align}\label{i14}
\DD(\eps,k,k_1)\DD(\eps,k_2,k_3)\DD(\eps,k_4,k_5)\dots\text{ where }k>k_1\geq k_2>k_3\geq k_4>k_5\dots.
\end{align}
It appears that one can obtain a useful final energy estimate of $|(\|V_k\|)|_{\ell^2(k)}$ only in problems where one can show that for any given $(\eps,\zeta)$, the number of large (that is $\frac{C_5r^2}{\eps\gamma}$) factors in products like \eqref{i14} is bounded above by a number $\EE$ that is \emph{independent} of $(\eps,\zeta,k)$.  In our first main result, Theorem \ref{tt26}, we identify classes of problems for which such an $\EE$ exists.   For these problems  simple formulas for $\EE$ \eqref{tt26d},\eqref{E} are given in terms of the numbers of incoming and outgoing modes $|\cO|$, $|\cI|$ and the number of directions $\beta_j\in \Upsilon^0$ where the uniform Lopatinski condition fails.\footnote{Here  $\Upsilon_0$ is the set of directions where the uniform Lopatinski condition fails.  In 2D problems like \eqref{i1}, $\Upsilon_0$ has the form $\{\pm \beta_1,\dots,\pm\beta_L\}$ for some unit vectors $\beta_j\in \RR^2$ (Proposition \ref{s7fj}).  The sets $\cI$, $\cO$ are the index sets for the incoming and outgoing modes (Definition \ref{def2}).}  An efficient procedure for estimating $|(\|V_k\|)|_{\ell^2(k)}$ on the basis of estimates like \eqref{b001} is given in the proof of Proposition \ref{bb6}.  This leads to the following estimate for $U(t,x,\theta)$ as in \eqref{i6}:  there exist $\eps_0>0$ and positive constants $\EE\in\NN$,  $K$,  $\gamma_0$ independent of $\eps$  such that for $0<\eps<\eps_0$ and $\gamma\geq \gamma_0$ we have\footnote{The argument with $F=0$ goes through unchanged if $G=\sum_{k\geq N^*}G_ke^{ik\theta}$ instead of  $G=\sum_{k\geq 1}G_ke^{ik\theta}$.     The case $G=0$, $F=\sum_{k\geq N^*}F_ke^{ik\theta}$ can be treated in the same way. The constants $K,\gamma_0, \eps_0$ in \eqref{est} are independent of $N^*$, so one can take a limit as $N^*\to  -\infty$ to handle general data $(F,G).$}
\begin{align}\label{est}
\begin{split}
&|e^{-\gamma t}U|_{L^2(t,x,\theta)}+\left|\frac{e^{-\gamma t}U(0)}{\sqrt{\gamma}}\right|_{L^2(t,x_1,\theta)}\leq \\
&\qquad \frac{K}{(\eps\gamma)^\mathbb{E}}\left[\frac{1}{\gamma^2}\left(\sum_{k\in\mathbb{Z}}\left||X_k|\widehat{F_k}\right|_{L^2(x_2,\sigma,\eta)}^2\right)^{1/2}+\frac{1}{\gamma^{3/2}}\left(\sum_{k\in\mathbb{Z}}\left|\widehat{G_k}|X_k|\right|^2_{L^2(\sigma,\eta)}\right)^{1/2}\right].
\end{split}
\end{align}
We call  the process of passing from the iteration estimate to an estimate like \eqref{est}  \emph{the cascade estimates}. 

\begin{rem}\label{unif}
(1) We refer to an estimate like \eqref{est} as a \emph{uniform} estimate because the constants $\EE$, $K$, and $\gamma_0$ are independent of $\eps$.  In section \ref{multiple} we use such estimates to justify high order geometric optics expansions for some weakly stable systems like \eqref{i1}.  If, for example,  the estimate held only for $\gamma_0\gtrsim \frac{1}{\eps}$, it would be useless for this purpose (Remark \ref{opt}, part (3)).   Such a condition on $\gamma_0$ is needed for estimates obtained by applying the approach of \cite{C1,C} or \cite{CGW3} to problems with large oscillatory coefficients like \eqref{i1}.

(2) When the uniform Lopatinski condition (Definition \ref{ustable}) holds, the estimate \eqref{est} holds with $\EE=0$ and without the large $|X_k|$ factors on the right.

\end{rem}

\subsubsection{Two-sided cascades.}  \label{2sc} \qquad Two-sided cascades appear whenever $d(\theta_{in})$ as in \eqref{i7a} has both positive and negative Fourier spectrum.  
First we consider the simplest choice of $d(\theta_{in})$ that gives rise to a two-sided cascade: 
\begin{align}
d(\theta_{in})=e^{i\theta_{in}}+e^{-i\theta_{in}}.
\end{align}
With $F=0$ for now the transformed singular problem \eqref{i9} reduces to 
\begin{align}\label{c2}
\begin{split}
&D_{x_2}V_k-\mathcal{A}(X_k)V_k=i e^{i\frac{\omega_N(\beta_l)}{\eps}x_2}B_2^{-1}MV_{k-1}+i e^{-i\frac{\omega_N(\beta_l)}{\eps}x_2}B_2^{-1}MV_{k+1}\\
&BV_k=\widehat{G_k}(\zeta)\text{ on }x_2=0.
\end{split}
\end{align}
and the iteration estimate \eqref{i10} reduces to
\begin{align}\label{c4}
\|V_k\|\leq \frac{C}{\gamma}\sum_{r\in\{1,-1\}}\sum_{t\in\{0,1,-1\}}\|\DD(\eps,k,k-r)V_{k-r-t}\|+|\CalG_k|_{L^2(\sigma,\eta)}, \;k\in\ZZ,
\end{align}

Suppose we take $G(t,x_1,\theta)=G_0(t,x_1)$ and attempt to estimate $\|V_0\|$ in terms of $|\cG_0|$ by iterations of \eqref{c4} parallel to \eqref{b002}, \eqref{b01}.
We now obtain  an \emph{infinite} cascade in which terms $V_j$ for \emph{all} $j\in\ZZ$ (a two-sided cascade) eventually appear on the right:\footnote{In the step $[(V_0)]\to [(V_{-2},V_{-1},V_0,V_0,V_1,V_2,\cG_0)]$ the terms $V_{-1}$,$V_1$  correspond to the terms in \eqref{c4}$_{k=0}$ where $r=+1,-1$ respectively, and $t=0$.}
\begin{align}\label{c5}
\begin{split}
&[(V_0)]\to [(V_{-2},V_{-1},V_0,V_0,V_1,V_2,\cG_0)]\to [(V_{-4},V_{-3},V_{-2},V_{-2}, V_{-1},V_0), (V_{-3},V_{-2},V_{-1},V_{-1},V_{0},V_1), \\
&(V_{-2},V_{-1},V_{0},V_{0}, V_{1},V_2,\cG_0),(V_{-2},V_{-1},V_{0},V_{0}, V_{1},V_2,\cG_0),(V_{-1},V_{0},V_{1},V_{1}, V_{2},V_3),(V_{0},V_{1},V_{2},V_{2}, V_{3},V_4),\cG_0]\\
&\to ...
\end{split}
\end{align}
 From \eqref{c5} we see that the cascade will never terminate in a stage where only entries $\cG_0$ appear. 
Moreover, it is not hard to check that if one writes out the iterated estimates corresponding to \eqref{c5}, the following situation will obtain: 
for any particular  $(\eps,\zeta,p)$  and for \emph{any} $\EE\in\NN$,   after enough iterations terms will eventually appear on the right hand side of the estimate in which products of $\EE$ copies of $\DD(\eps,p,p-1)(\zeta)$ occur.\footnote{The same applies to $\DD(\eps,p,p+1)$ of course.}   If $\DD(\eps,p,p-1)(\zeta)=\frac{C_5}{\eps\gamma}$ ($r^2=1$ now), then factors $\left(\frac{1}{\eps\gamma}\right)^\EE$ will occur for all $\EE\in\NN$.   Thus, it appears that one can obtain a sensible estimate only in problems where  all factors $\DD(\eps,p,p\pm 1)$, $p\in\ZZ$ are $\leq C_5$ for all $\zeta$ (or almost every $\zeta$). 
 In that case 
 one can estimate $|(\|V_k\|)|_{\ell^2(k)}$ by summing (the square of) \eqref{c4} over $k$ and choosing $\gamma$ large enough to absorb the $\DD$ terms on the left.  This argument does \emph{not} involve iteration of \eqref{c4}, and thus \emph{avoids}  two-sided cascades.

 Our second main result, Theorem \ref{tvv29}, identifies a class of problems like \eqref{i1} with oscillatory coefficients having both positive and negative spectrum and for which the global amplification factors  satisfy 
 \begin{align}\label{i15}
 \DD(\eps,k,k-r)\leq C_5r^2 \text{ for all }(\eps,k,r,\zeta).
 \end{align}
 The property \eqref{i15} allows us to prove an estimate like \eqref{est} with $\EE=0$. The proof shows that \eqref{i15} holds whenever 
 \begin{align}\label{i16}
 \Upsilon_0=\{\beta_l,-\beta_l\} \text{ and }|\cI|=1.
 \end{align}
Theorem \ref{multiamp}, discussed below, provides an example where \eqref{i15} fails when $\Upsilon_0=\{\beta_l,-\beta_l\}$, but $|\cI|=2$.     
 
 In Remark \ref{ts1} we discuss how Theorems \ref{tt26} and \ref{tvv29}  can be applied to problems derived from the linearized compressible 2D Euler equations.

\subsection{Multiple amplification and optimality of the estimates}
\qquad The literature on geometric optics for weakly stable hyperbolic boundary problems contains several examples of \emph{first-order} amplification: roughly, oscillatory data $(F,G)$ of a given amplitude and wavelength $\eps$  yields a solution whose amplitude is larger by a factor of $\frac{1}{\eps}$ (for example, \cite{MA, MR, AM, CG, CGW3, CW6}.  The estimate \eqref{est} in cases where $\EE\geq 1$ suggests that multiple amplification by a factor $\frac{1}{\eps^{\EE+1}}$ may occur.\footnote{Recall that the factors $|X_k|$ on the right in \eqref{est} already produce one order of amplification.}

In section \ref{multiple} we study a weakly stable $3\times 3$ problem \eqref{d1} for which $\Upsilon_0=\{\beta_l,-\beta_l\}$.  There are three phases 
\begin{align}
\phi_i(t,x)=\beta_l\cdot (t,x_1)+\omega_i(\beta_l)x_2, \;i=1,2,3,
\end{align}
with $\phi_2$, $\phi_3$ incoming and $\phi_1$ outgoing, and it is assumed that they exhibit a single \emph{resonance} (defined in \eqref{tt7}):
\begin{align}\label{i16a}
\phi_2+\phi_3=2\phi_1.
\end{align}
The oscillatory coefficient in the problem is given by $d(\theta_{in})=e^{i\theta_3}$.   In our third main result, Theorem \ref{multiamp}, we construct and rigorously justify high order approximate solutions which exhibit instantaneous double amplification, that is, amplification by a factor of $\frac{1}{\eps^2}$ evident at any time $t>0$.   
The amplified solution consists of a wave in the boundary that travels along a characteristic of the Lopatinski determinant and which ``radiates" doubly amplified waves into the interior which travel along characteristics corresponding to the two incoming phases.\footnote{The boundary amplitude equation governing propagation of the wave in the boundary is a transport equation with a nonlocal zero-order term \eqref{f00} involving a Fourier multiplier $m(D_{\theta_0})$.}

In Remark \ref{opt} we explain that successively adding terms 
$e^{i2\theta_3}$, $e^{i4\theta_3}$, $e^{i8\theta_3}$, $\dots$  to $d(\theta_{in})$ yields problems that should successively exhibit instantaneous 3rd order, 4th order, 5th order, $\dots$ amplification.
Each such problem would exhibit the maximum order of amplification permitted by the estimate \eqref{est} when $\EE$ is computed by the formula \eqref{E} of Theorem \ref{tt26}, and would thus further demonstrate the optimality of that estimate.

Proposition \ref{tt8} considers problems with resonances  in the ``bad cases" for which the quantity $\Omega_{i,j}$ (defined below in \eqref{i22a})  is rational and lies in $(0,\infty)$ or $(-\infty,-1)$.\footnote{For the resonance \eqref{i16a} we have $\Omega_{1,2}=1$.}   This proposition, taken together with our examples of multiple amplification, indicates that for such problems there is no hope of proving an estimate like \eqref{est} with finite $\EE$ when the Fourier spectrum of $d(\theta_{in})$ is an arbitrary infinite subset of $\ZZ$.   That remains true even when the spectrum of $d(\theta_{in})$ is restricted to be  purely positive.

Our example of double amplification confirms that are situations in which some of the global amplifications $\DD(\eps,k,k-r)(\zeta)$ are indeed large, that is, equal to $\frac{C_5r^2}{\eps\gamma}$, on $\zeta-$sets of positive measure.    We do not see how to be sure of this without such examples.  Section \ref{disc} contains more discussion related to this point.


Finally,  observe that any estimate like \eqref{est} for which $\EE=0$, such as  the estimate \eqref{tv31} in Theorem \ref{tvv29},  is optimal.   Simple examples show that  the amplifying factors $|X_k|$ on the right of \eqref{est} are unavoidable in weakly stable  problems like \eqref{i1} even when $\cD=0$ \cite{CG}.

\begin{rem}\label{indep}
1) The results of section \ref{1s}   show that the factors $\DD(\eps,k,k-r)$ are determined just by our assumptions on $(L(\partial),B)$; they are \emph{independent} of the choice of the oscillatory factor $\cD(\theta_{in})$.   

2)  If $\cD$ is replaced by $\eps\cD$ in \eqref{i1}, then $V_{k-r-t}$ is replaced by $\eps V_{k-r-t}$ on the right in the iteration estimate
\eqref{i10}, and this leads (by a simple argument \emph{not} involving iteration of \eqref{i10}) to a final  estimate like \eqref{est} where $\EE=0$; multiple amplification does not occur.\footnote{Such an estimate in this case already follows by  different methods from the results of \cite{CGW3}.}     The properties of the spectrum of $\cD$, the assumptions involving the quantity $\Omega_{i,j}$ made in Theorems \ref{tt26} and \ref{tvv29}, the presence or absence of resonances, and the choice of $\phi_{in}$ or $\phi_{out}$ in the argument of $\eps\cD$ are all irrelevant to the estimate in this case.   Multiple instantaneous amplification is a phenomenon associated with large oscillatory coefficients.

\end{rem}

\subsection{Remarks on the proofs. }\label{remarks}\qquad We conclude with some informal remarks intended to give more insight into the proofs. \\

\textbf{Iteration and cascade estimates. }Consider first the formulas \eqref{a5}, \eqref{a6} for solutions
$w^+_k(x_2,\zeta)$, $w^-_k(x_2,\zeta)$ to the system obtained by diagonalizing the singular transformed problem \eqref{i9} near one of the bad directions $\beta\in\Upsilon^0$. 
For a given $(\eps,k)$ these formulas are valid for $\zeta$ such that $X_k:=\zeta+ k\frac{\beta_l}{\eps}$ lies in a small enough conic neighborhood $\Gamma_\delta(\beta)$ of $\beta$.\footnote{Here $\beta\in \Upsilon^0$ may or may not equal $\beta_l$, where  $\beta_l\in\Upsilon^0$ is the particular bad direction that appears in \eqref{i2}; $\beta_l$ is fixed throughout the paper.   This same $\beta_l$ appears in the definition $X_k:=\zeta+k\frac{\beta_l}{\eps}$.  When $\zeta$ is large compared to $k\frac{\beta_l}{\eps}$, $X_k$ can of course lie far from the $\beta_l$ direction.}
The functions  of $\zeta$ denoted by $\omega_j(\eps,k;\beta)$ that appear in \eqref{a3b} are related to the eigenvalues $\omega_j(X_k)$ of $\cA(X_k)$ ($\cA$ as in \eqref{i9}) by the slightly abusive equation
\begin{align}\label{i17}
\omega_j(\eps,k;\beta)(\zeta)=\omega_j(X_k) \text{ for }X_k\in\Gamma_\delta(\beta).
\end{align} 
The ``most dangerous" terms in the expressions for $w^\pm_k$ are the terms in the sum appearing in the second line of the expression for $w^-_k(x_2,\zeta)$,  \eqref{a6}. The matrix  $[Br_-(\eps,k)(\zeta)]^{-1}$ that appears there is large for $X_k$ near $\beta$ and satisfies\footnote{The function $\Delta(\zeta):=\det Br_-(\zeta)$ (Definition \ref{s7b}) is the Lopatinski determinant.  The uniform Lopatinski condition fails at $\beta$ precisely when $\Delta(\beta)=0$.}
\begin{align}\label{i18}
|[Br_-(\eps,k)(\zeta)]^{-1}|\lesssim |\Delta(X_k)|^{-1}\lesssim \frac{|X_k|}{\gamma} \text{ for }\gamma>0.
\end{align}
With the normalizations chosen in the next section (which imply \eqref{a7}, e.g.), one sees from  \eqref{a6} that to control the $r-$th term in the sum one must control quantities like
\begin{align}\label{i19}
\frac{1}{\Delta(X_k)}\int^\infty_0e^{i\omega_i(X_k)(-s)+ir\frac{\omega_N(\beta_l)}{\eps}s}\ar w^-_{k-r,j}(s,\zeta)ds,  \;i\in\cO,\;\;  j,N\in\cI
\end{align}
where $w^-_{k-r,j}$ is the $j-$th component of $w^-_{k-r}$.   

Consider the case when $X_k$ lies in a small conic neighborhood of $\beta_l$,  $\Gamma_\delta(\beta_l)$.  
We can try to control the factor $|\Delta(X_k)|^{-1}$ in \eqref{i19} by doing an integration by parts in the oscillatory integral.  An $s-$derivative will fall on $w^-_{k-r,j}$, and we  suppose for these remarks that $X_{k-r}\in\Gamma_\delta(\beta_l)$ in order to use the equation for $w^-_{k-r,j}$ given by \eqref{a3a}$_{k-r}$.  Since $\partial_s w^-_{k-r,j}$ can grow like $|X_k|$, there is no sure gain if we do the integration by parts directly on \eqref{i19}.  Instead, we set
\begin{align}\label{i20a}
w^{*,-}_{k-r,j}(x_2,\zeta):=e^{-i\omega_j(X_{k-r})x_2}w^-_{k-r,j},\text{ noting that }\partial_{x_2}w^{*,-}_{k-r,j}=e^{-i\omega_j(X_{k-r})x_2}h^-_{k-r,j},
\end{align}
where $h^-_{k-r,j}$ is the $j$-component of the right side of \eqref{a3a}$_{k-r}$.    We then rewrite \eqref{i19} as 
\begin{align}\label{i20}
\begin{split}
&\frac{1}{\Delta(X_k)}\int^\infty_0e^{i\omega_i(X_k)(-s)+ir\frac{\omega_N(\beta_l)}{\eps}s}e^{i\omega_j(X_{k-r})s}\ar w^{*,-}_{k-r,j}(s,\zeta)ds=\\
&\frac{1}{\Delta(X_k)}\int^\infty_0e^{-iE_{i,j}(\eps,k,k-r)s}\ar w^{*,-}_{k-r,j}(s,\zeta)ds, \text{ where }E_{i,j}(\eps,k,k-r)=\omega_i(X_k)-r\frac{\omega_N(\beta_l)}{\eps}-\omega_j(X_{k-r}).
\end{split}
\end{align}
We show in Proposition \ref{t18} and its refinements, Propositions \ref{t18aa} and \ref{goodcase},  that it ``frequently" happens that 
\begin{align}\label{i21}
|E_{i,j}(\eps,k,k-r)(\zeta)|\geq  C_3\frac{|X_k|}{|r|} \text{ or } |E_{i,j}(\eps,k,k-r)|\geq C_3 |X_{k-r}|, \text{ for a }C_3 \text{ independent of }(\eps,\zeta,k,r).
\end{align}
It is clear from \eqref{i18} that when the first possibility holds, an integration by parts can be used to control the factor $|\Delta(X_k)|^{-1}$ in \eqref{i20}.  It is less clear but true that when the second possibility holds, one can also control $|\Delta(X_k)|^{-1}$; see step \text{7} of the proof of Proposition \ref{tt30}.  If for some $(i,j)\in\cO\times\cI$ and $(\eps,\zeta,k,r)$ the alternative \eqref{i21} fails to hold, then we must define 
the global amplification factor $\DD(\eps,k,k-r)(\zeta)$ to be $\frac{C_5r^2}{\eps\gamma}$ for that particular choice of $(\eps,\zeta,k,r)$.\footnote{In the iteration estimate \eqref{i10}, for a given $r\in\ZZ\setminus 0$ most of the terms in the inner sum $\sum_{t\in\ZZ}$ come from terms like $e^{-i\omega_j(X_{k-r})x_2}h^-_{k-r,j}$ as in \eqref{i20a}, which result from the integration by parts.}    

One must also consider cases where $X_k\in\Gamma_\delta(\beta)$ for  $\beta \in\Upsilon^0\setminus \{\pm\beta_l\}$.  In these cases the quantity $\frac{\Delta(X_{k-r})}{\Delta(X_k)}$ is useful for controlling the large factor $|\Delta(X_k)|^{-1}$.  In Proposition \ref{keyt} we show that it ``frequently" happens that 
\begin{align}\label{i22}
\left|\frac{\Delta(X_{k-r})}{\Delta(X_k)}\right|\leq C_1|r|, \text{ for a }C_1 \text{ independent of }(\eps,\zeta,k,r).
\end{align}
If for some $(\eps,\zeta,k,r)$ the estimate \eqref{i22} fails to hold, then we must define 
the global amplification factor $\DD(\eps,k,k-r)(\zeta)=\frac{C_5r^2}{\eps\gamma}$ for that particular choice of $(\eps,\zeta,k,r)$.     We restrict attention to two space dimensions partly in order to control quotients like the one in \eqref{i22} (Remark \ref{2d}).\footnote{See also Remark \ref{factor} for more about the restriction to two space dimensions.} 

 In order to carry out the cascade estimates, we must understand  how ``often" \eqref{i21} and \eqref{i22} can fail. This work is done in the Propositions listed in the above two paragraphs, as well as in Proposition \ref{large}.   For example, Proposition \ref{t18aa} shows that when
\begin{align}\label{i22a}
\Omega_{i,j}:= \frac{\omega_i(\beta_l)-\omega_N(\beta_l)}{\omega_j(\beta_l)-\omega_i(\beta_l)}\in (-1,0),
\end{align}
then \eqref{i21} holds for ``most" $(k,r)$.    Proposition \ref{large} counts for any given $(\eps,\zeta)$ how many of the  amplification factors in products like \eqref{i14} can be large.  It turns out that under the assumptions of Theorems \ref{tt26} and \ref{tvv29}, no more than $\EE$ can be large, where $\EE$ is specified in those theorems.

\textbf{Construction of approximate solutions. }In section \ref{multiple} we construct approximate solutions exhibiting double amplification to the problem \eqref{d1}.    The approximate solutions have   the form
\begin{align}\label{i23}
u^\eps_a(t,x)=\sum^J_{k=-1}\eps^k U_k(t,x,\frac{\Phi}{\eps}),\;  \;\; U_k(t,x,\theta)=\sum_{\alpha\in\ZZ^3} U_{k,\alpha}(t,x)e^{i\alpha\theta},
\end{align}
where $\Phi(t,x)=(\phi_1,\phi_2,\phi_3)$ is a triple of resonant phases, and the profiles $U_k(t,x,\theta)$ are $2\pi-$periodic with respect to $\theta=(\theta_1,\theta_2,\theta_3)$.  
An extra difficulty in the construction of approximate solutions exhibiting multiple (as opposed to single) amplification is the higher degree of coupling among the profile equations.    For example, to determine the trace of the leading order profile $U_{-1}$, we must now solve two coupled boundary amplitude equations, \eqref{d17}(a),(b), instead of just one boundary amplitude equation as in the case of single amplification.
The  equations \eqref{d17}(a),(b)  involve in turn  higher unknown amplitudes (respectively, $U_0$, $U_1$).

To cope with this high degree of coupling, we introduce the following device.   We  make a list in section \ref{table} of the possible modes that we \emph{expect} to appear in the various profiles; that is, for each $U_k$ in \eqref{i23}, we specify the $\alpha$ for which $U_{k,\alpha}$ might possibly be nonzero.\footnote{In section \ref{table} we don't actually list the various $\alpha$, but the list we give is equivalent to a specification of the $\alpha$.} 
The list is at first just a reasonable guess that takes into account the boundary data in \eqref{d1}, the single resonance \eqref{d2}, the profile equations, and what we already know about the exact solution.  We then make two assumptions on which we base the construction of the profiles: (a) profiles $U_k$ exist which satisfy all the profile equations; (b)  the only    possible nonzero modes of these profiles are those which appear in our list.    It is not clear at first that these  assumptions are \emph{consistent}.  However, by making these assumptions we are able to construct profiles that  satisfy the profile equations and whose nonzero modes manifestly lie in our list.  Thus, the construction itself verifies the consistency and correctness of the two assumptions.   A key advantage of this approach is that it allows us to decouple the equations by solving for \emph{individual} modes of profiles in the appropriate order, starting with the low modes (that is, modes for which $|\alpha|$ is small.)     We are not concerned about uniqueness of the profiles, because  we know that the exact solution is unique, and we show in Theorem \ref{multiamp} that the approximate solution is $O(\eps^\infty)$ close in $L^\infty$ to the exact solution.


\begin{nota}
1) If $f(\tau,\eta,\eps,k,\gamma)$ is a function of $(\tau,\eta,\eps,k,\gamma)\in \CalD$ for some domain $\CalD$,  the statement $ f\sim 1$ means that there exist positive constants $A_1$, $A_2$ independent of $(\tau,\eta,\eps,k,\gamma)\in \CalD$ such that 
$$A_1\leq |f(\tau,\eta,\eps,k,\gamma)|\leq A_2 \text{ on }\CalD.$$

2) If $g$ is another such function, the statement $f\lesssim g$  means that there exists a positive constant $C$ independent of $(\tau,\eta,\eps,k,\gamma)\in \CalD$ such that 
$$ |f| \leq C|g| \text{ on }\CalD.$$

3) The constants $C$, $C_1$, $C_2$, $K$, $M$, etc. that appear frequently in the estimates below are always independent of the important parameters $(\tau,\eta,\eps,k,\gamma)$.  

4) If $S$ is a finite set, we denote the cardinality of $S$ by $|S|$. 

5) For $n\in\NN$ we denote the $n\times n$ identity matrix by $I_n$.

6) We sometimes denote the norm on $L^2(\Omega\times \TT)$ by $|U|_{L^2(t,x,\theta)}$ (and use similar notation for other spaces), when the domain of the variables $(t,x,\theta)$ is clear from the context.
\end{nota}

\textbf{Acknowledgment.} We thank Jean-Francois Coulombel for stimulating discussions over a number of years related to the topic of this paper.

\section{Assumptions and main results}\label{assumptions}

 \qquad We consider the problem 
 \begin{align}\label{s1}
 \begin{split}
 &\partial_t u+B_1\partial_{x_1}u+B_2\partial_{x_2}u+\cD\left(\frac{\phi_{in}}{\eps}\right)u=F(t,x,\frac{\phi_0}{\eps})\text{ in }x_2>0\\
 &Bu=G(t,x_1,\frac{\phi_0}{\eps})\text{ on }x_2=0\\
 &u=0 \text{ in }t<0.
 \end{split}
 \end{align}
  Here the $B_j$  are constant $N\times N$ matrices, $B_2$ is invertible, and the boundary phase is taken to be $\phi_0(t,x_1)=\beta_l\cdot (t,x_1)$, where $\beta_l=(\sigma_l,\eta_l)\in\RR^2\setminus 0$ is one of the directions where the uniform Lopatinski condition (Definition \ref{ustable})  fails.   The matrix $\cD(\theta_{in})$ and functions $F(t,x,\theta)$, $G(t,x_1,\theta)$  are respectively $2\pi-$periodic in $\theta_{in}$ and $\theta$.
  Also, 
  \begin{align}\label{s1a}
  \phi_{in}(t,x)=\phi_0(t,x_1)+\omega_N(\beta_l) x_2
  \end{align}
  is one of the incoming phases (see Definition \ref{def2}).  
    For convenience we take $F$ and $G$ to be zero in $t<0$. 

\begin{assumption}[Strict hyperbolicity]\label{assumption1}
The $B_j$ are real matrices, and there exist real valued functions $\lambda_j(\eta,\xi)$, $j=1,\dots,N$ that are analytic on $\RR^2\setminus 0$ and homogeneous of degree one such that 
\begin{align}
\det(\sigma I+ B_1\eta+B_2\xi)=\prod_{k=1}^N (\sigma+\lambda_k(\eta,\xi))\text{ for all }(\eta,\xi)\in \RR^2\setminus 0.
\end{align}
Moreover, we have
\begin{align}
\lambda_1(\eta,\xi)<\lambda_2(\eta,\xi)<\dots <\lambda_N(\eta,\xi) \text{ for all }(\eta,\xi)\in \RR^2\setminus 0.
\end{align}

\end{assumption}

We rewrite \eqref{s1} as 
\begin{align}\label{s2}
\begin{split}
&D_{x_2} u+A_0 D_tu+A_1D_{x_1}u-iB_2^{-1}\cD\left(\frac{\phi_{in}}{\eps}\right)u=F(t,x,\frac{\phi_0}{\eps})\\
&Bu= G(t,x_1,\frac{\phi_0}{\eps})\text{ on }x_2=0\\
&u=0\text{ in }t<0,
\end{split}
\end{align}
where $A_0=B_2^{-1}$, $A_1=B_2^{-1}B_1$, and $F$ has been modified in an unimportant  way.  
Let us  introduce the matrix 
\begin{align}\label{s3b}
\mathcal{A}(\tau,\eta)=-(A_0\tau+A_1\eta),\;\;(\tau,\eta)=(\sigma-i\gamma,\eta)\in\CC\times \RR, 
\end{align}
and define the following sets of frequencies: 

\begin{align*}
& \Xi := \Big\{ \zeta:=(\sigma-i\gamma,\eta) \in \C \times \R \setminus (0,0) : \gamma \ge 0 \Big\} \, ,
& \Sigma := \Big\{ \zeta \in \Xi : \sigma^2 +\gamma^2 +\eta^2 =1 \Big\} \, ,\\
& \Xi_0 := \Big\{ (\sigma,\eta) \in \R \times \R \setminus (0,0) \Big\} = \Xi \cap \{ \gamma = 0 \} \, ,
& \Sigma_0 := \Sigma \cap \Xi_0 \, .
\end{align*}

The hyperbolic region and the glancing set are defined as follows.
 
\begin{definition}
\label{defhyp}
a) The hyperbolic region ${\mathcal H}$ is the set of all $(\sigma,\eta) \in \Xi_0$ such that the matrix
       ${\mathcal A}(\sigma,\eta)$ is diagonalizable with real eigenvalues.

b)  Let ${\bf G}$ denote the set of all $(\sigma,\eta,\xi) \in \R \times \R^2$ such that $(\eta,\xi) \neq 0$ and there exists
       an integer $k \in \{1,\dots,N\}$ satisfying:
\begin{equation*}
\sigma + \lambda_k(\eta,\xi) = \dfrac{\partial \lambda_k}{\partial \xi} (\eta,\xi) = 0 \, .
\end{equation*}
If $\pi ({\bf G})$ denotes the projection of ${\bf G}$ on the  first $2$  coordinates (in other words $\pi (\sigma,\eta,\xi)
=(\sigma,\eta)$ for all $(\sigma,\eta,\xi)$), the glancing set ${\mathcal G}$ is ${\mathcal G} :=
\pi ({\bf G}) \subset \Xi_0$.

\end{definition}

\begin{assumption}\label{assumption2}
The matrix $B_2$ is invertible and has $p$ positive eigenvalues, where $1\leq p\leq N-1$.  The boundary matrix $B$ is $p\times N$, real,  and of rank $p$. 
\end{assumption}

\begin{prop}\cite{K}
\label{stable}
Let Assumptions \ref{assumption1} and \ref{assumption2} be satisfied. Then for all $\zeta \in \Xi \setminus
\Xi_0$, the matrix $i{\mathcal A}(\zeta)$ has no purely imaginary eigenvalue and its stable subspace $\E^s
(\zeta)$ has dimension $p$. Furthermore, $\E^s$ defines an analytic vector bundle over $\Xi \setminus \Xi_0$
that can be extended as a continuous vector bundle over $\Xi$.
\end{prop}

\noindent For all $(\sigma,\eta) \in \Xi_0$, we let $\E^s(\sigma,\eta)$ denote the continuous extension of $\E^s(\zeta)$ to the
point $(\sigma,\eta)$.  The analysis in \cite{Met} shows that away from the glancing set ${\mathcal G} \subset \Xi_0$,
$\E^s(\zeta)$ depends analytically on $\zeta$, and the hyperbolic region ${\mathcal H}$ does not contain any
glancing point.

\begin{defn}\cite{K}
\label{ustable}
As before let $p$ be the number of positive eigenvalues of $B_2$, and let 
$$L(\partial)=\partial_t+B_1\partial_{x_1}+B_2\partial_{x_2}.$$ 
The problem $(L(\partial),B)$ is said to
be \emph{uniformly stable} or to satisfy the \emph{uniform Lopatinski condition} (ULC) if
\begin{equation}\label{us}
B \, : \, \E^s (\zeta) \longrightarrow \bC^p
\end{equation}
is an isomorphism for all $\zeta\in \Sigma$.
Similarly, we say $(L(\partial),B)$ satisfies the $ULC$ on $\Xi$,  (respectively, on a closed conic subset $\Gamma\subset \Xi$), if the map \eqref{us} is an isomorphism on $\Sigma$ (respectively, on  the  subset of $\Sigma$ corresponding to $\Gamma$). 

\end{defn}

We now fix a choice of $\beta=(\underline \sigma,\underline \eta)\in \mathcal{H}$.   As a consequence of strict hyperbolicity there is a closed conic neighborhood $\Gamma^+_\delta(\beta)$ of $\beta$ in $\Xi$ with opening angle $\delta>0$, $$ \Gamma_\delta^+(\beta)=\left\{\zeta\in\Xi:\left|\frac{\zeta}{|\zeta|}-\frac{\beta}{|\beta|}\right|\leq \delta\right\},$$ such that  $\CalA(\zeta)$ has $N$ distinct eigenvalues $\omega_j(\zeta)$ and corresponding eigenvectors $R_j(\zeta)$ satisfying
\begin{align}\label{s4}
\CalA(\zeta)R_j(\zeta)=\omega_j(\zeta)R_j(\zeta), \; j=1,\dots,N \text{ on }\Gamma^+_\delta(\beta).
\end{align}
 The functions $\omega_j$, $R_j$ map $\Gamma_\delta^+(\beta)$ into $\CC$, $\CC^N$ respectively, are homogeneous of degree one, and are analytic in $\tau$, $C^\infty$ (in fact, real-analytic) in $\eta$.    We also define normalized vectors 
 \begin{align}\label{s4a}
 r_j(\zeta):=R_j(\zeta)/|R_j(\zeta)|,
 \end{align}
 which are merely $C^\infty$ in $(\sigma,\gamma,\eta)$. 
 
 To each root ${\omega}_j(\beta)=\underline{\omega}_j$ there corresponds a unique integer $k_j \in \{ 1,\dots,N\}$ such that
$\underline{\sigma} + \lambda_{k_j} (\underline{\eta},\underline{\omega}_j)=0$. We can then define the following
real 
phases $\phi_j$ and their associated group velocities:
\begin{equation}
\label{phases}
\forall \, j =1,\dots,N \, ,\quad \phi_j (x):= \phi_0(t,y)+\underline{\omega}_j \, x_2 \, ,\quad
{\bf v}_j := \nabla \lambda_{k_j} (\underline{\eta},\underline{\omega}_j) \, .
\end{equation}
Let us observe that each group velocity ${\bf v}_j$ is either incoming or outgoing with respect to the space
domain $\R^2_+$: the last coordinate of ${\bf v}_j$ is nonzero. This property holds because $\beta$
does not belong to the glancing set ${\mathcal G}$.  For any $\beta\in\cH$ there are exactly $N-p$ outgoing phases and (after relabelling if necessary) we denote the corresponding set of 
indices by $\mathcal{O}=\{1,\dots,N-p\}$.  The set of incoming indices is $\mathcal{I}=\{N-p+1,\dots,N\}$. 

We can therefore adopt the following classification:

\begin{definition}
\label{def2}
The phase $\phi_j$ is incoming if the group velocity ${\bf v}_j$ is incoming (that is, when $\partial_{\xi_2}
\lambda_{k_j} (\underline{\eta},\underline{\omega}_j)>0$), and it is outgoing if the group velocity ${\bf v}_j$
is outgoing ($\partial_{\xi_2} \lambda_{k_j} (\underline{\eta},\underline{\omega}_j) <0$).  If the phase $\phi_j$ is incoming (resp., outgoing), we shall also refer to the corresponding frequency $\underline{\omega}_j$ as incoming (resp., outgoing).

\end{definition}

   The $\omega_j$ are real-valued for $\gamma=0$ and can be  divided into two groups according as $j\in \mathcal O$ or $\mathcal I$.     There exists  a constant $c>0$ such that for $\zeta\in\Gamma^+_\delta(\beta)$
 \begin{align}\label{s5}
 \begin{split}
 &\mathrm{Im}\;\omega_j(\zeta)\leq -c\gamma \text{ for }j\in\mathcal{O}\\
 &\mathrm{Im}\;\omega_j(\zeta)\geq c\gamma \text{ for }j\in \mathcal{I}.
\end{split}
\end{align}

We define the $N\times (N-p)$ matrix $r_+(\zeta)$ and the $N\times p$ matrix $r_-(\zeta)$ by 
\begin{align}\label{s6}
r_+=\begin{pmatrix}r_1&r_2&\dots&r_{N-p}\end{pmatrix}, \qquad r_-=\begin{pmatrix}r_{N-p+1}&\dots&r_{N}\end{pmatrix} \text{ on }\Gamma^+_\delta(\beta).
\end{align}
Similarly, define $R_\pm(\zeta)$ using the unnormalized eigenvectors $R_j(\zeta)$. 

Next we  introduce the $N\times N$ matrix
\begin{align}\label{s6a}
S(\zeta)=\begin{pmatrix}r_+(\zeta)&r_-(\zeta)\end{pmatrix}\text{ on }\Gamma^+_\delta(\beta).
\end{align}
Having fixed the column vectors $r_j(\zeta)$, we define an $(N-p)\times N$ matrix $\ell_+(\zeta)$ and a $p\times N$ matrix $\ell_-(\zeta)$  such that 
\begin{align}\label{s7} 
S^{-1}(\zeta)=\begin{pmatrix}\ell_+(\zeta)\\\ell_-(\zeta)\end{pmatrix} \text{ on }\Gamma^+_\delta(\beta).
\end{align}
The rows of $S^{-1}(\zeta)$ are given by row vectors $\ell_j(\zeta)$, $j=1,\dots,N$, and these satisfy $\ell_j(\zeta)\sim 1$.

The following well-known proposition, proved in \cite{CG}, gives a useful decomposition of $\EE^s(\zeta)$.
\begin{prop}\label{s7a}
For $\zeta\in\Gamma^+_\delta(\beta)$ the stable subspace $\EE^s(\zeta)$ admits the decomposition 
\begin{align}
\EE^s(\zeta)=\bigoplus_{j\in \mathcal{I}}\mathrm{span}\;r_j(\zeta),
\end{align}
and the vectors $r_j(\beta)$ can be (and are) taken to be real vectors. 
\end{prop}

Next we define the  Lopatinski determinant and recall some facts from \cite{W3}.
\begin{definition}\label{s7b}
 For $\zeta\in \Gamma^+_\delta(\beta)$ define the  analytic Lopatinski determinant
\begin{align}
\Delta_a(\zeta)=\det BR_-(\zeta).
\end{align}
 and the normalized  Lopatinski determinant
\begin{align}
\Delta(\zeta)=\det Br_-(\zeta).
\end{align}
\end{definition}
Observe that $\Delta(\zeta)$ is $C^\infty$  in 
$\zeta$ and positively homogeneous of degree $0$ on $\Gamma^+_\delta(\beta)$.   Moreover,   the map \eqref{us} fails to be an isomorphism at $\zeta\in\Gamma^+_\delta(\beta)$ if and only if $\Delta(\zeta)=0$.

We can now formulate our weak stability assumption on the problem $(L(\partial),B)$.

\begin{assumption}
\label{assumption3}
\begin{itemize}
 \item For all $\zeta \in \Xi \setminus \Xi_0$, $\text{\rm Ker} \;B \cap \E^s (\zeta) = \{ 0\}$.

 \item The set $\Upsilon_0 := \{ \zeta \in \Sigma_0 : \text{\rm Ker}\; B \cap \E^s (\zeta) \neq \{ 0\} \}$ is
       nonempty and included in the hyperbolic region ${\mathcal H}$. 

 \item For all $\beta\in \Upsilon_0$ there exists a neighborhood $\Gamma^+_\delta(\beta)$ as above on which functions $\omega_j$, $R_j$, $\Delta_a$ are defined and we have
 \begin{align}\label{firstorder}
\Delta_a (\beta)=0\text{ and } \partial_\tau\Delta_a (\beta)\neq 0.
 \end{align}
 
\end{itemize}
\end{assumption}

\begin{rem}\label{factor}
1). Using \eqref{firstorder} and the implicit function theorem, and after reducing $\delta>0$ if necessary, we may write
\begin{align}\label{factor1}
\Delta_a(\tau,\eta)=(\tau-g(\eta))H_+(\tau,\eta) \text{ on }\Gamma^+_\delta(\beta),
\end{align}
where $g$ and $H_+$ inherit the obvious regularity from $\Delta_a$, $g$ is homogeneous of degree one, and $H_+$ is homogeneous of degree $p-1$ and nonvanishing on $\Gamma^+_\delta(\beta)$. 
Since $d=2$ the function $g$ is in fact linear
\begin{align}\label{s7c}
g(\eta)=c_+\eta \text{ for some }c_+=c_+(\beta)\in \RR.
\end{align}

2).  Since $\Delta(\zeta)=\frac{ \Delta_a(\zeta)}{\prod^N_{j=N-p+1} |R_j(\zeta)|}$, we obtain from \eqref{factor1}, \eqref{s7c}:
\begin{align}\label{s7d}
\Delta(\tau,\eta)=(\tau-c_+\eta))h_+(\tau,\eta)\text{ on }\Gamma^+_\delta(\beta),
\end{align}
where  $h_+(\tau,\eta)$ is homogenous of degree $-1$ and nonvanishing on $\Gamma^+_\delta(\beta)$.

3.)Using Assumption \ref{assumption3}, the compactness of $\Upsilon_0$, and the analyticity of $\Delta_a$ one can deduce as in \cite{BS}, section 8.3, that $\Upsilon_0$ is a finite set
in the case $d=2$ that we are considering:
\begin{align}\label{s7e}
\Upsilon_0=\{\beta_j, \;j=1,\dots,M_0\}.
\end{align}
Proposition \ref{s7fj} below shows  that $M_0$ is even.

\end{rem}

We can now state our main result for problems with oscillatory coefficients having only positive spectrum, the case of one-sided cascades.   There is, of course, an exactly parallel result for coefficients with only negative spectrum.   On a first reading the reader might wish to focus on parts (b) and (c) of  the following theorem.

Consider the singular problem \eqref{i6} for $U(t,x,\theta)$ or the equivalent transformed problem for $V_k(x_2,\zeta):=\widehat{U_k}(\zeta,x_2)$:
\begin{align}\label{tt26a}
\begin{split}
&(a) D_{x_2}V_k-\mathcal{A}(X_k)V_k=i\sum_{r\in \ZZ\setminus 0}   e^{ir\frac{\omega_N(\beta_l)}{\eps}x_2}B_2^{-1}\widehat{\cD}(r)V_{k-r}+\widehat{F_k}(x_2,\zeta)\\
&(b) BV_k=\widehat{G_k}(\zeta)\text{ on }x_2=0.
\end{split}
\end{align}

\begin{theo}[Energy estimate:  $\cD$ has positive spectrum only]\label{tt26}
\emph{}

(a). Consider solutions $U(t,x,\theta)$ of the singular system \eqref{i6} with forcing terms $F=\sum_{k\in\mathbb{Z}}F_k(t,x)e^{ik\theta}$, $G=\sum_{k\in\mathbb{Z}}G_k(t,x_1)e^{ik\theta}$ in $H^1(\Omega\times\TT)$, $H^1(\RR^2\times \TT)$ respectively, under Assumptions \ref{assumption1}, \ref{assumption2}, \ref{assumption3}.    Assume  the $N\times N$ matrices $\widehat{\cD}(r)$ in \eqref{tt26a} satisfy 
\begin{align}\label{b00}
 \widehat{\cD}(r)=0\text{ for }r\leq 0,\;\;|\widehat{\cD}(r)|\lesssim |r|^{-(M+2)} \text{ for some }M\geq 2.
\end{align}

For $i\in\cO$, $j\in\cI\setminus \{N\}$ let \footnote{Here the function $\omega_j$ is defined as in \eqref{s4}, but where $\beta=\beta_l$.}
\begin{align}\label{b000}
\Omega_{i,j}:=\frac{\omega_i(\beta_l)-\omega_N(\beta_l)}{\omega_j(\beta_l)-\omega_i(\beta_l)}.
\end{align}
Suppose there exist positive constants $\eps_0$, $\delta_0$  such that for $0<\eps\leq \eps_0$, $0<\delta\leq \delta_0$, $\zeta\in \Xi$ and any strictly increasing sequence of integers $(k_p)$,  there exist numbers $\MM_{i,j}\geq 0$, $\lambda_{i,j}>0$ independent of $(\zeta,\eps,\delta)$ and $(k_p)$,  such that  the set 
\begin{align}\label{tt26b}
\cM_{i,j}(\zeta,\eps,\delta):=\{p\in\ZZ:X_{k_p}\in\Gamma_{\frac{\delta}{r_p}}(\beta_l), X_{k_{p-1}}\in \Gamma_{\frac{\delta}{r_p}}(\beta_l), |t(p)-r_p\Omega_{i,j}|<\lambda_{i,j}\}
\end{align}
has finite cardinality $|\cM_{i,j}(\zeta,\eps,\delta)|\leq\MM_{i,j}$.\footnote{Here we use the notation $\Gamma_\delta(\beta)=\Gamma_\delta^+(\beta)\cup\Gamma_\delta^-(\beta)$, where $\Gamma^-_\delta(\beta):=\{(\sigma-i\gamma,\eta):(-\sigma-i\gamma,-\eta)\in\Gamma^+_\delta(\beta)\}$;   see \eqref{ss7}.     Moreover, $r_p:=k_p-k_{p-1}$ and $t(p)$ is given by $\tilde X_{k_{p-1}}=t(p)\frac{\beta_l}{\eps}$, where $\tilde X_{k_{p-1}}$ is the orthogonal projection of $X_{k_{p-1}}$ on $\beta_l$.}

Define the natural number $\mathbb{E}$ by
\begin{align}\label{tt27}
\mathbb{E}=\left(\frac{|\Upsilon_0|}{2}-1\right)+\sum_{i\in\cO,j\in\cI\setminus \{N\}}\MM_{i,j},
\end{align}
where we set $\sum\MM_{i,j}=0$ in case $\cI=\{N\}$.
Then there exist positive constants $\gamma_0$, $K$ such that for $0<\eps<\eps_0$ and $\gamma\geq \gamma_0$ we have 
\begin{align}\label{tt28}
|U^\gamma|_{L^2(t,x,\theta)}+\left|\frac{U^\gamma(0)}{\sqrt{\gamma}}\right|_{L^2(t,x_1,\theta)}\leq \frac{K}{(\eps\gamma)^\mathbb{E}}\left[\frac{1}{\gamma^2}\left(\sum_{k\in\mathbb{Z}}\left||X_k|\widehat{F_k}\right|_{L^2(x_2,\sigma,\eta)}^2\right)^{1/2}+\frac{1}{\gamma^{3/2}}\left(\sum_{k\in\mathbb{Z}}\left|\widehat{G_k}|X_k|\right|^2_{L^2(\sigma,\eta)}\right)^{1/2}\right].
\end{align}
Equivalently,
\begin{align}\label{tt28a}
|U^\gamma|_{L^2(t,x,\theta)}+\left|\frac{U^\gamma(0)}{\sqrt{\gamma}}\right|_{L^2(t,x_1,\theta)}\leq \frac{K}{(\eps\gamma)^\mathbb{E}}\left[\frac{1}{\gamma^2}|\Lambda_D F^\gamma|_{L^2(t,x,\theta)}+\frac{1}{\gamma^{3/2}}|\Lambda_D G^\gamma|_{L^2(t,x_1,\theta)}\right],
\end{align}
where $U^\gamma:=e^{-\gamma t}U$, $U^\gamma(0)$ is the trace on $x_2=0$,  and $\Lambda_D$ is the singular operator associated to $|X_k|$.\footnote{We define $\Lambda_DG^\gamma:= \sum_k\int e^{i(\sigma t+\eta x_1+k\theta)}|X_k|\;\widehat{G_k}(\tau-i\gamma,\eta)d\sigma d\eta.$}

(b) Under Assumptions \ref{assumption1}, \ref{assumption2}, \ref{assumption3} suppose 
\begin{align}
\begin{split}
&\Omega_{i,j}\in (-1,0) \text{ for } i\in\cO, j\in (\cI\setminus \{N\}).
\end{split}
\end{align}
Then the hypotheses of part (a) are satisfied with 
\begin{align}\label{tt26c}
\begin{split}
&  \lambda_{i,j}=\frac{1}{3}\mathrm{min}\:\{|\Omega_{i,j}-(-1)|, |\Omega_{i,j}-0|\}>0,\\
& \MM_{i,j}=1,  i\in \cO, j\in (\cI\setminus \{N\}), 
\end{split}
\end{align}
and the estimate \eqref{tt28} holds where $\EE$ as in \eqref{tt27} has the value
\begin{align}\label{tt26d}
\EE=|\cO|(|\cI|-1)+\left(\frac{|\Upsilon_0|}{2}-1\right).
\end{align}

(c) Under Assumptions \ref{assumption1}, \ref{assumption2}, \ref{assumption3} suppose now that the oscillatory coefficient in \eqref{i6} has \emph{finite} positive spectrum, that is,  suppose there exists $P\in\NN$ such that $\widehat{\cD}(r)=0$ for all but $P$ distinct choices of $r\in\NN$.   We now make \emph{no assumption} on the numbers $\Omega_{i,j}$.  
Then the estimate \eqref{tt28} holds with 
\begin{align}\label{E}
\EE=P|\cO|(|\cI|-1)+\left(\frac{|\Upsilon_0|}{2}-1\right).
\end{align}

\end{theo}


Next we state our main result for problems with oscillatory coefficients having both positive and negative spectrum, the case of two-sided cascades.     

 \begin{theo}[Energy estimate: $\cD$ has positive and negative spectrum]\label{tvv29}
  
   Consider solutions $U(t,x,\theta)$ of the singular system \eqref{i6} with forcing terms $F=\sum_{k\in\mathbb{Z}}F_k(t,x)e^{ik\theta}$, $G=\sum_{k\in\mathbb{Z}}G_k(t,x_1)e^{ik\theta}$ in $H^1(\Omega\times\TT)$, $H^1(\RR^2\times \TT)$ respectively.   Suppose  the coefficients $\widehat{\cD}(r)$ in \eqref{tt26a} satisfy 
   $\widehat{\cD}(0)=0$ and 
   $|\widehat{\cD}(r)|\lesssim |r|^{-(M+1)}$ for some $M\geq 2$ .

As before we make the structural  Assumptions \ref{assumption1}, \ref{assumption2}, \ref{assumption3}, but now we \emph{add} the assumption that
\begin{align}\label{tv30} 
{\Upsilon_0}=\{\beta_l,-\beta_l\}\text{ and } \cI=\{N\} \;(\text{thus}, \cO=\{1,\dots,N-1\}).
\end{align}
Then there exist positive constants $\eps_0$, $\gamma_0$, $K$   such that for $0<\eps\leq \eps_0$ and $\gamma\geq \gamma_0$ we have
\begin{align}\label{tv31}
|U^\gamma|_{L^2(t,x,\theta)}+\left|\frac{U^\gamma(0)}{\sqrt{\gamma}}\right|_{L^2(t,x_1,\theta)}\leq K\left[\frac{1}{\gamma^2}\left(\sum_{k\in\mathbb{Z}}\left||X_k|\widehat{F_k}\right|_{L^2}^2\right)^{1/2}+\frac{1}{\gamma^{3/2}}\left(\sum_{k\in\mathbb{Z}}\left|\widehat{G_k}|X_k|\right|^2_{L^2(\sigma,\eta)}\right)^{1/2}\right].
\end{align}

  \end{theo}

\begin{rem}\label{ts1}
1) 
Suppose that   $\Omega_{i,j}$ defined by \eqref{b000} lies in $(-\infty,-1)$.      Then by interchanging $\omega_N$ and $\omega_j$ (recall $j,N\in \mathcal{I}$) we obtain a quotient that lies in $(-1,0)$.

2).   In Appendix B of \cite{CGW3} we consider the linearized compressible Euler equations in two space dimensions (a $3\times 3$ system) obtained by linearizing  at a given
 specific volume $v>0$ and a subsonic incoming velocity $(0,u)$, $0<u<c$.    We choose a frequency $\beta_l$ in the hyperbolic region which yields distinct eigenvalues $\omega_i(\beta_l)$, $i=1,2,3$ such that $\omega_2$, $\omega_3$ are incoming and $\omega_1$ is outgoing, so $|\cI|=2$.   Taking the  boundary matrix $B=\begin{pmatrix}0&v&0\\u&0&v\end{pmatrix}$, we check that the weak stability assumptions  \ref{assumption1}, \ref{assumption2}, \ref{assumption3} are satisfied with $\Upsilon^0=\{\beta_l,-\beta_l\}$.  Using the formulas for $\omega_j$ given there, one can verify
 \begin{align}
 \frac{\omega_1-\omega_2}{\omega_3-\omega_1}\in (-1,0).
 \end{align}
Thus, if we add an oscillatory zero-order term $\cD(\phi_2/\eps)$ to the linearized operator, we obtain a problem like \eqref{s1} to which part (b) of Theorem \ref{tt26} applies.

3) Similarly,  section 5.3 of \cite{CG} considers the linearized compressible Euler equations obtained by linearizing at $v>0$ and a subsonic outgoing velocity $(0,u)$, $-c<u<0$.  The choices made there of $\beta_l$ and a $1\times 3$ boundary matrix $B$ allow Theorem \ref{tvv29} to be applied; in this case  $|\cI|=1$, $\Upsilon^0=\{\beta_l,-\beta_l\}$.

4) In Theorem \ref{tvv29} or in any case of Theorem \ref{tt26} where $\EE=0$,  we can allow the oscillatory coefficient $\cD$ to have nonzero mean, $\widehat{\cD}(0)\neq 0$. In that case  equation  \eqref{tt26a} is modified to be 
\begin{align}\label{tz1}
D_{x_2}V_k-\mathcal{A}(X_k)V_k-iB_2^{-1}\widehat{\cD}(0)V_{k}=i\sum_{r\in \ZZ\setminus 0}  \ar e^{ir\frac{\omega_N(\beta_l)}{\eps}x_2}B_2^{-1}\widehat{\cD}(r)V_{k-r}+\widehat{F_k}(x_2,\zeta).
\end{align}
One can then use an argument in the proof of \cite{CGW3}, Proposition 2.4 to simultaneously diagonalize the second and third terms on the left of \eqref{tz1}.  This produces an error term 
$r_{-1}(X_k)V_k$, where $r_{-1}$ is homogeneous of degree $-1$, which can be treated as part of the interior forcing and absorbed in the final estimate by taking $\gamma$ large.  The simultaneous diagonalization process replaces $\xi_\pm(X_k)$ in \eqref{a3a}, \eqref{a3b} by $\xi_\pm(X_k)+r_{0_\pm}(X_k)$, where $r_{0_\pm}$ is homogeneous of degree zero.   
This change  necessitates some straightforward, minor changes in the proof of the iteration estimate.  
When $\EE\neq 0$ this argument does not work; the error term $r_{-1}(X_k)V_k$ cannot be treated as forcing and absorbed.

5) For the problem \eqref{i1} with an oscillatory coefficient  $\cD(\frac{\phi_{out}}{\eps})$ instead of $\cD(\frac{\phi_{in}}{\eps})$, one can prove an analogue of Theorem \ref{tt26}, but with \emph{larger} amplification exponents $\EE$.\footnote{For example, Proposition \ref{goodcase} no longer holds in this case.}  Such a result is given in \cite{W3}.  We do not see how to prove an analogue of Theorem \ref{tvv29} in this case, since the application of Proposition \ref{2sided} requires that all $\DD(\eps,k,k-r)$ be ``small".  
Another reason for our focus on the $\cD(\frac{\phi_{in}}{\eps})$ case is its greater relevance to the Mach stem problem; this is explained in \cite{CW6}. 





\end{rem}

 \qquad Consider now the $3\times 3$, strictly hyperbolic WR problem
 \begin{align}\label{d1}
 \begin{split}
 &\partial_t u+B_1\partial_{x_1}u+B_2\partial_{x_2}u+e^{i\frac{\phi_{3}}{\eps}}Mu=0\text{ in }x_2>0\\
 &Bu=\eps G(t,x_1,\frac{\phi_0}{\eps}):=\eps g_1(t,x_1) e^{i\frac{\phi_0}{\eps}}\text{ on }x_2=0\\
 &u=0 \text{ in }t<0.
 \end{split}
 \end{align}
  Here the $B_j$ and $M$ are constant $3\times 3$ matrices, the $B_j$ are real, $B_2$ is invertible, and $\phi_0(t,x_1)=\beta_l\cdot (t,x_1)$, where $\beta_l=(\sigma_l,\eta_l)\in \Upsilon_0=\{\beta_l,-\beta_l\}$.   The $2\times 3$ matrix $B$ is  real and of rank $2$.   
%

The system has characteristic phases $\phi_m(t,x)=\beta_l\cdot (t,x_1)+\omega_m(\beta_l)x_2$, $m=1,2,3$, where $\phi_2$, $\phi_3$ are incoming and $\phi_1$ outgoing.
Let us assume that  the \emph{only} resonance is 
\begin{align}\label{d2}
\phi_2+\phi_3=2\phi_1\Leftrightarrow \omega_2(\beta_l)+\omega_3(\beta_l)=2\omega_1(\beta_l). 
\end{align}
In section \ref{multiple} we construct high order approximate solutions to \eqref{d1} of the form
\begin{align}\label{d3}
u^\eps_a(t,x)=\sum_{k=-1}^J\eps^k U_k\left(t,x,\frac{\Phi}{\eps}\right),
\end{align}
where $\Phi=(\phi_1,\phi_2,\phi_3)$ and the profiles $U_k(t,x,\theta)$ are $2\pi$-periodic with respect to $\theta=(\theta_1,\theta_2,\theta_3)$.\footnote{In stating Theorem \ref{multiamp} we use $\theta$ as a placeholder for $\frac{\Phi}{\eps}$ and $\theta_0$ as a placeholder for $\frac{\phi_0}{\eps}$.} The construction requires the  following  small divisor assumption:\footnote{By results of \cite{JMR2}, the small divisor assumption \ref{sd} is generically satisfied.} 

\begin{ass}\label{d2z}
There exist constants $C>0$ and $a\in\RR$ such that for all $(k,l)\in\NN\times\NN$ with $k\neq l$ we have
\begin{align}
|\det L(kd\phi_2+ld\phi_3)|\geq C|(k,l)|^{a}.
\end{align}
\end{ass}

\begin{theo}[Instantaneous double amplification]\label{multiamp}
For any $T>0$  let $\Omega_T=(-\infty,T]\times \{(x_1,x_2):x_2\geq 0\}$, and consider the problem \eqref{d1} with resonance \eqref{d2}  on $\Omega_T$ under assumptions \ref{assumption1}, \ref{assumption2}, \ref{assumption3}, \ref{d2z}, where $g_1(t,x_1)\in H^\infty((-\infty,T]\times \RR)$ and vanishes in $t<0$.

a) The problem has a unique exact solution $u^\eps\in H^\infty(\Omega_T)$.  Moreover,   $u^\eps=U^\eps(t,x,\theta_0)|_{\theta_0=\frac{\phi_0}{\eps}}$, where $U^\eps(t,x,\theta_0)$ is the  solution given by Theorem \ref{tt26}(c) to the  singular problem \eqref{i6}  corresponding to \eqref{d1}.

b) For any $Q\in \NN$,  the problem has a high order approximate solution $u^\eps_a(t,x)\in H^\infty(\Omega_T)$ of the form \eqref{d3} with $J=J(Q)$,  which satisfies
\begin{align}\label{Q}
|u^\eps(t,x)-u^\eps_a(t,x)|_{L^\infty(\Omega_T)}=O(\eps^Q).
\end{align}

c) The leading profile $U_{-1}(t,x,\theta)$ is generally nonzero (and independent of $\eps$) for arbitrarily small $t>0$.\footnote{The proof of Theorem \ref{multiamp} will clarify the sense in which ``generally" here means ``for most choices of $M$ and $g_1$" in \eqref{d1}.}   Thus, the exact solution $u^\eps$ exhibits instantaneous double amplification, that is, amplification by a factor of $\frac{1}{\eps^2}$ relative to the boundary data.

\end{theo}

\section{Iteration estimate}\label{1s}


\qquad 
In this section we prove the iteration estimate (Proposition \ref{tt30}) after presenting some definitions and tools needed for its statement and proof.     
Our first task is to choose suitable extensions to $\Xi$ of the functions $R_j(\zeta)$, $\omega_j(\zeta)$, $r_j(\zeta)$, $\Delta_a(\zeta)$, and $\Delta(\zeta)$ defined in section \ref{assumptions}.  

\subsection{Extensions to $\Gamma_\delta(\beta)$ and then to $\Xi$.}\label{extension}

\qquad Now fix $\beta\in\Upsilon_0$ and a conic neighborhood $\Gamma^+_\delta(\beta)\ni\beta$ as before.
We first extend the vectors $R_j$ from $\Gamma^+_\delta(\beta)$ to the conic set $\Gamma_\delta(\beta)=\Gamma^+_\delta(\beta)\cup \Gamma^-_\delta(\beta)\subset \Xi$, where 
\begin{align}\label{ss7}
\Gamma^-_\delta(\beta):=\{(\sigma-i\gamma,\eta):(-\sigma-i\gamma,-\eta)\in\Gamma^+_\delta(\beta)\}, 
\end{align}
by setting 
\begin{align}\label{s7ff}
R_j(\sigma-i\gamma,\eta)=\overline{R}_j(-\sigma-i\gamma,-\eta) \text{ for }(\sigma-i\gamma,\eta)\in\Gamma^-_\delta(\beta).
\end{align}
The extended $R_j$ are analytic in $\tau$, $C^\infty$ in $\eta$,   and homogeneous of degree $1$ in $\Gamma_\delta(\beta)$.      Next we take   $C^\infty$ extensions of these vectors to the rest of $\Xi$ with the property that $S(\zeta)$, defined as in \eqref{s6a} but using the extended $R_j$ to define
$r_j=R_j/|R_j|$, is invertible with a uniformly bounded inverse $S^{-1}(\zeta)$ on $\Xi$.   As before we denote the rows of $S^{-1}(\zeta)$ by $\ell_j(\zeta)$, $j=1,\dots,N$.  

We extend the eigenvalues $\omega_j(\zeta)$ from $\Gamma^+_\delta(\beta)$ to $\Gamma_\delta(\beta)$ in essentially the same way:
\begin{align}\label{s7fh}
\omega_j(\sigma-i\gamma,\eta)=-\overline{\omega}_j(-\sigma-i\gamma,-\eta) \text{ for }(\sigma-i\gamma,\eta)\in\Gamma^-_\delta(\beta).
\end{align}
We then further extend the $\omega_j$ to the rest of $\Xi$ as $C^\infty$ functions homogeneous of degree $1$.

As with $R_j$, each $\omega_j$ is analytic in $\tau$ (and smooth in $\eta$) on $\Gamma_\delta(\beta)$.  Moreover, since the matrices $A_j$ in \eqref{s3b} are real, we have
\begin{align}\label{s7fk}
\mathcal{A}(\zeta)R_j(\zeta)=\omega_j(\zeta)R_j(\zeta)\text{ on }\Gamma_\delta(\beta).
\end{align}
It follows directly from \eqref{s7fh} that we now have on $\Gamma_\delta(\beta)$:
 \begin{align}\label{s7fi}
 \begin{split}
 &\mathrm{Im}\;\omega_j(\zeta)\leq -c\gamma \text{ for }j\in\mathcal{O}\\
 &\mathrm{Im}\;\omega_j(\zeta)\geq c\gamma \text{ for }j\in \mathcal{I},
\end{split}
\end{align}
for $c$ as in \eqref{s5}.    Thus, we see that the extended $R_j(\zeta)\in \mathbb{E}^s(\zeta)$ for $j\in \mathcal{I}$  on $\Gamma_\delta(\beta)$.   Noting also that since the $R_j(\beta)$ are real 
we have $\det BR_-(-\beta)=0$, we obtain:\footnote{The above extension of the $\omega_j$ and $R_j$ from $\Gamma^+_\delta(\beta)$ to $\Gamma_\delta(\beta)$ satisfying \eqref{s7fk} and \eqref{s7fi} works near any $\beta\in\mathcal{H}$, not just $\beta\in\Upsilon_0$.}

\begin{prop}\label{s7fj}
Under Assumptions \ref{assumption1}, \ref{assumption2}, and \ref{assumption3}, 
if $\beta\in\Upsilon_0$ then $-\beta\in \Upsilon_0$. Thus, $|\Upsilon_0|=M_0$ is even, and there exist vectors $\beta_1, \dots, \beta_{\frac{M_0}{2}}$ in $\Sigma_0$ such that 
\begin{align}
\Upsilon_0=\{\pm\beta_1,\dots,\pm\beta_{\frac{M_0}{2}}\}.
\end{align}
\end{prop}



The extensions of the $R_j$ similarly give us extensions of $\Delta_a$ and $\Delta$ to $\Xi$, with $\Delta_a$ (resp., $\Delta$) having the same regularity as the $R_j$ (resp., as the $r_j$).   The extension $\Delta(\zeta)$ clearly satisfies (recall \eqref{s7d})
\begin{align}\label{s7fg}
|\Delta(\zeta)|\sim \frac{|\tau-c_+\eta|}{|\zeta|}\text{ on }\Gamma_\delta(\beta),
\end{align}
and  we choose the extensions of the $R_j$ so that 
\begin{align}\label{s7f}
|\Delta(\zeta)|\sim 1 \text{ on }\Xi\setminus \Gamma_\delta(\beta).
\end{align}

\begin{rem}\label{s7g}

1) The constant $c_+$ in \eqref{s7c}, \eqref{s7fg} depends on the choice of $\beta\in\Upsilon_0$, so we will \emph{sometimes} write $c_+(\beta)$  to indicate this.

2)  The fully extended functions  $\omega_j$, $R_j(\zeta)$, $r_j(\zeta)$, $\Delta(\zeta)$, $S(\zeta)$ are defined on $\Xi$, but were first defined on $\Gamma_\delta(\beta)$, so they all depend on the choice of $\beta \in\Upsilon_0$.  
Thus,  we will {sometimes} write $\omega_j(\zeta;\beta)$, 
$R(\zeta;\beta)$, $\Delta(\zeta;\beta)$, etc., when it is important to recall this.    Observe that generally $\omega_j(\zeta;\beta_k)\neq \omega_j(\zeta;\beta_m)$ if $k\neq m$.

3) It follows from our construction that there exists a constant $C$ independent of $\zeta\in\Xi$ and $j\in \{1,\dots,M_0\}$ such that 
\begin{align}\label{s7h}
|S(\zeta;\beta_j)|\leq C\text{ and }|S^{-1}(\zeta;\beta_j)|\leq C\text{ for all }\zeta\in \Xi, \;j\in\{1,\dots,M_0\}.
\end{align}

3)It will be convenient to set $\Upsilon_0^+=\{\beta_1,\dots,\beta_{\frac{M_0}{2}}\}$, where these $\beta_j=(\sigma_j,\eta_j)$ are the (possibly relabeled) elements of $\Upsilon_0$ such that
$\sigma_j\geq 0$.  In the ambiguous case where $\sigma_j=0$ we take $\beta_j=(0,1)$.  It is no restriction to (and we do now) suppose that $\beta_l$ (as in \eqref{s1a}) is in $\Upsilon^+_0$.  Moreover,  we shall henceforth always take the $\beta$ that appears in expressions like $R(\zeta;\beta)$, $\Delta(\zeta;\beta)$, etc., to be an element of $\Upsilon^+_0$.

\end{rem}




The following proposition is a simple consequence of \eqref{s7fg}.

\begin{prop}\label{s9}


 The $p\times p$ matrix $Br_-(\zeta)$ satisfies
\begin{align}\label{s10a}
\left|[Br_-(\zeta)]^{-1}\right|\lesssim |\Delta(\zeta)|^{-1}\sim \frac{|\zeta|}{|\tau-c_+(\beta)\eta|}\text{ on }\Gamma_\delta(\beta).
\end{align}

\end{prop}

     \begin{nota}\label{Xk}
     
     1.  For $\zeta=(\tau,\eta)\in\Xi$ we set $X_k:=\zeta+\frac{k\beta_l}{\eps}$ for $\beta_l\in \Upsilon^+_0$ as in \eqref{i9}.\footnote{Unlike $\beta$, the frequency $\beta_l$ appearing in \eqref{i9} and \eqref{s1a}  is fixed once and for all.} 
     
     2.  Given any function $f(\zeta)$ defined for $\zeta\in \Xi$, we denote by $f(\eps,k)$ (with slight abuse) the function of $\zeta$ given by  $f(X_k)$ for that particular choice of $(\eps,k)$.   This defines the functions
     $r_\pm(\eps,k;\beta)$, $\ell_\pm(\eps,k;\beta)$, $S(\eps,k;\beta)$, $\Delta(\eps,k;\beta)$, etc. that we use below.   These functions change, of course, as $\beta$ varies.   Observe, for example, that \footnote{This notation is intended to emphasize that we are viewing the function in \eqref{s10aa} as a function of $\zeta$ for fixed $(\eps,k)$.}
     \begin{align}\label{s10aa}
     \omega_j(\eps,k;\beta)(\zeta)=\omega_j(X_k;\beta).
     \end{align}
     
   3.  Let $\chi_b(\zeta;\beta)$ be the characteristic function of $\Gamma_\delta(\beta)$.  Thus, $\zeta\in\mathrm{supp}\;\chi_b(\eps,k;\beta)$ if and only if $X_k\in\Gamma_\delta(\beta)$.

  \end{nota}

\subsection{Tools for the iteration estimate}\label{tools1s}

\quad In the proof of the iteration estimate we will repeatedly use the following simple proposition, where $|f|_{L^2(x_2,\sigma,\eta)}$ denotes the $L^2$ norm over $\bR^+_{x_2}\times \bR_\sigma\times \bR_\eta$.   The proof, given in \cite{W3}, is almost immediate.
\begin{prop}\label{g1z}
For $\gamma>0$, $\tau=\sigma-i\gamma$ we  have on $x_2\geq 0$:
\begin{align}\label{g1}
\begin{split}
&(a) \left|\int^{x_2}_0e^{-\gamma (x_2-s)}f(s,\tau,\eta)ds\right|_{L^2(x_2,\sigma,\eta)}\leq \frac{1}{\gamma}|f|_{L^2(x_2,\sigma,\eta)}\\
&(b) \left|\int^{\infty}_{x_2}e^{\gamma (x_2-s)}f(s,\tau,\eta)ds\right|_{L^2(x_2,\sigma,\eta)}\leq \frac{1}{\gamma}|f|_{L^2(x_2,\sigma,\eta)}\\
&(c)   \left|\int^{\infty}_0e^{-\gamma s}f(s,\tau,\eta)ds\right|_{L^2(\sigma,\eta)}\leq \frac{1}{\sqrt{2\gamma}}|f|_{L^2(x_2,\sigma,\eta)}\\
&(d) \left|e^{-\gamma x_2}g(\tau,\eta)\right|_{L^2(x_2,\sigma,\eta)} = \frac{1}{\sqrt{2\gamma}}|g|_{L^2(\sigma,\eta)}.
\end{split}
\end{align}

\end{prop}

\qquad We turn now to the problem of estimating $V_k(x_2,\zeta)$ as in \eqref{i9}.   For $X_k$ outside a conic neighborhood of the set of bad directions $\Upsilon_0$ we can 
 expect to have a Kreiss-type estimate.  The main new problem is to estimate $V_k$ for $X_k$ in a neighborhood $\Gamma_\delta(\beta)$ for any given $\beta\in\Upsilon_0^+$; for such 
 $X_k$ the quantity  $|\Delta(X_k;\beta)|^{-1}$ can blow up as $\eps\to 0$.

The next lemma is an easy consequence of the estimates on $\Delta(\zeta;\beta)$ in section \ref{extension} and definitions given there.

\begin{lem}\label{t6}
Fix $\beta\in\Upsilon^+_0$.  Recall that $X_k:=\zeta+k\frac{\beta_l}{\eps}$ and that $\zeta\in \mathrm{supp}\;\chi_b(\eps,k;\beta)\Leftrightarrow X_k\in \Gamma_\delta(\beta)$. 
For $k\in \ZZ$ the following estimates hold:\footnote{Here and below we use Notation \ref{Xk}, and suppress much of the $\beta$ and $\zeta$ dependence; thus, for example,  $\Delta(\eps,k)=\Delta(\eps,k;\beta)$.   In \eqref{t6a}(b), for example, $\Delta(\eps,k)$ is evaluated at the same $\zeta$ that appears in the definition of $X_k$.}
\begin{align}\label{t6a}
\begin{split}
&(a)\;|\Delta(\eps,k)|\lesssim 1\text{ on }\Xi \\
&(b)\; \left|[Br_-(\eps,k)]^{-1}\right|\lesssim |\Delta(\eps,k)|^{-1}\sim \frac{|X_k|}{|\tau-c_+(\beta)\eta|}\text{ on }\mathrm{supp}\;\chi_b(\eps,k;\beta)\\
&(c)\; \left| [Br_-(\eps,k)]^{-1}\right|\lesssim |\Delta(\eps,k)|^{-1} \leq C(\delta)\text{ on }\Xi\setminus  \mathrm{supp}\;\chi_b(\eps,k;\beta) \\
&(d)\; |Br_{\pm}(\eps,k)|\lesssim 1 \text{ on }\Xi,\\
& (e)\; |\omega_i(\eps,k)-\omega_j(\eps,k)|\sim |X_k| \text{ for }i\neq j\text{ on }\mathrm{supp}\;\chi_b(\eps,k;\beta)\\
&(f)\; \mathrm{Im}\;\omega_j(\eps,k)\leq -c\gamma\text{ for }j\in \mathcal{O},\;\;\;\mathrm{Im}\;\omega_j(\eps,k)\geq c\gamma\text{ for }j\in \mathcal{I}, \text{ on }\mathrm{supp}\;\chi_b(\eps,k;\beta).
\end{split}
\end{align}

(g) Let $r\in\ZZ\setminus 0$.  When $X_k \notin  \Gamma_{\frac{\delta}{|r|}}(\beta)$ we have  $|\Delta(\eps,k)|^{-1}\leq C(\delta)|r|$.

\end{lem}

\begin{rem}\label{2d}
The estimate \eqref{t6a}(b) uses the \emph{linearity} of the function $g(\eta)=c_+(\beta)\eta$ appearing in \eqref{factor1}, a consequence of the assumption that we are working in two space dimensions.   That estimate is used in the proof of Proposition \ref{keyt} and in step \textbf{8} of the proof of Proposition \ref{tt30}.
\end{rem}

The functions $E_{i,j}(\eps,k,k-r)$ in the next definition arise as exponents in the oscillatory integrals of step \textbf{7} of the iteration estimate (Proposition \ref{tt30}).   They are used to control the large factors $[Br_-(\eps,k)]^{-1}$ appearing in \eqref{a6}.   Recall that for a given $(\zeta,\eps)$, we have $\omega_i(\eps,k;\beta)(\zeta)=\omega_i(\zeta;\beta)|_{\zeta=X_k}$, and, moreover, $\omega_N(\beta_l)=\omega_N(\zeta;\beta_l)|_{\zeta=\beta_l}$.

\begin{defn}
Let $k\in\ZZ$, $r\in\ZZ\setminus 0$.    
For $\zeta=(\tau,\eta)\in\Xi$ we define the function of $\zeta$:
\begin{align}\label{t17}
\begin{split}
&E_{i,j}(\eps,k,k-r):=\omega_i(\eps,k;\beta_l)-\frac{r\omega_N(\beta_l)}{\eps}-\omega_j(\eps,k-r;\beta_l), \text{ where }i\in\mathcal{O}, \;j\in\mathcal{I}\\
\end{split}
\end{align}
\end{defn}


The next two propositions, which are proved in section \ref{ptools},  address important technical issues that arise in the proof of the iteration and cascade estimates.   In particular, they are needed for defining and controlling the amplification factors $\DD(\eps,k,k-r)$. 
In these propositions 
we fix an \emph{admissible sequence} of integers.

\begin{defn}\label{admissible}
The sequence $(k_j)_{j\in\ZZ}$ is \emph{admissible} if $k_j\in\ZZ$ for all $j$ and $(k_j)$ is either strictly increasing or strictly decreasing.
We define the ``step size" $r_j=k_j-k_{j-1}\in\ZZ\setminus 0$. 

\end{defn}

\begin{prop}\label{keyt}
Let $(k_j)_{j\in\ZZ}$ be an admissible sequence.  
We continue to work with a fixed $\beta\in\Upsilon^+_0$, which may or may not equal $\beta_l$ as in the oscillatory term of  \eqref{i1}.   
   Consider any given $(\zeta,\eps)\in \Xi\times (0,\eps_0]$.
Provided $\delta$ and $\eps_0$ are small enough,
there exist positive constants $C_1,C_2$ independent of $(\zeta,\eps,j)$
such that for
all but at most one $j\in \ZZ$   such that   $X_{k_j}\in\Gamma_\delta(\beta)$,  we have:
\begin{align}\label{t11za}
\begin{split}
\left|\frac{\Delta(\eps,k_{j-1})}{\Delta(\eps,k_j)}\right|\leq C_1|r_j|.
\end{split}
\end{align}
If the exceptional case $j=m_1(\zeta,\eps;\beta)$ occurs (that is, if \eqref{t11za} fails for $j=m_1(\zeta,\eps;\beta)$), we have
\begin{align}\label{t11zb}
\begin{split}
&\left|\frac{\Delta(\eps,k_{m_1-1})}{\Delta(\eps,k_{m_1})}\right|\leq \frac{C_2|r_{m_1}|}{\eps\gamma}.
%
\end{split}
\end{align}


\end{prop}






\begin{prop}\label{t18}
Let $\beta\in\Upsilon_0^+$ and 
suppose  $i\in\mathcal{O}$, $j\in\mathcal{I}$, 
There exist positive constants $\eps_0$, $\delta_0$ and positive constants $C_3$, $C_4$ \emph{independent of} $(\zeta,\eps,p)\in \Xi\times (0,\eps_0]\times  \ZZ$ such that the following situation holds:

1) If $\beta\neq \beta_l$, then for any given $(\zeta,\eps,\delta)\in\Xi\times(0,\eps_0]\times (0,\delta_0]$, at least one of $X_{k_p}, X_{k_{p-1}}$ does not  lie in $\in\Gamma_{\frac{\delta}{|r_p|}}(\beta)$.

2) Suppose $\beta=\beta_l$, and for any given $(\zeta,\eps,\delta)\in\Xi\times(0,\eps_0]\times (0,\delta_0]$ let the set $M_{i,j}(\zeta,\eps,\delta;C_3)\subset \ZZ$ be characterized by the condition that   $p\notin M_{i,j}(\zeta,\eps,\delta;C_3)$ if and only if  \emph{either} 

a) at least one of $X_{k_p}, X_{k_{p-1}}$ does not  lie in $\in\Gamma_{\frac{\delta}{|r_p|}}(\beta_l)$, \emph{or} 

b) both points  lie in $\Gamma_{\frac{\delta}{|r_p|}}(\beta_l)$ and 
\begin{align}\label{t18w}
\quad |E_{i,j}(\eps,k_p,k_{p-1})|\geq C_3\frac{|X_{k_p}|}{|r_p|} \text{ or }|E_{i,j}(\eps,k_p,k_{p-1})|\geq C_3|X_{k_{p-1}}|.
\end{align}
Then for every $p\in M_{i,j}(\zeta,\eps,\delta;C_3)$ we have 
\begin{align}\label{t18z}
|X_{k_{p}}|\leq \frac{C_4|r_p|}{\eps}.
\end{align}

\end{prop}

Thus, $M_{i,j}(\zeta,\eps,\delta;C_3)$ is a set of ``bad" $p$ values: $p$ lies in this set if and only if both $X_{k_p}, X_{k_{p-1}}$  lie in $\in\Gamma_{\frac{\delta}{|r_p|}}(\beta_l)$ and \eqref{t18w} fails to hold.  We are able to prove energy estimates like \eqref{tt28} only in those cases where, for any given $(i,j,\delta)$ as above,  the set $M_{i,j}(\zeta,\eps,\delta;C_3)$ has a finite cardinality bounded above by numbers $\MM_{i,j}$ \emph{independent of} $(\zeta,\eps,\delta)$ and the choice of sequence $(k_p)$; see Theorem \ref{tt26}.   The exponent $\EE$ in \eqref{tt28} is determined in part by the size of these upper bounds.



\begin{rem}[Uniformity with respect to sequences $(k_j)_{j\in\ZZ}$]\label{uniform}
(1) The previous two propositions were stated with respect to a given admissible sequence  $(k_j)_{j\in\ZZ}$.     The statements involve 
certain constants: $\eps_0$, $\delta_0$ and $C_i$, $i=1,\dots,4$. It is important for our application to the cascade estimates that these constants can be chosen independently of the sequence $(k_j)_{j\in\mathbb{Z}}$.  Indeed, the proofs show that dependence on this sequence occurs only in the explicit  step-size factors $r_j$ that appear in estimates like \eqref{t11za} or \eqref{t18w}.   

(2)  We will sometimes apply Propositions \ref{keyt} and \ref{t18} to a fixed pair $k$, $k-r$ instead of to a given admissible sequence $(k_j)$. 
To do this we can  imagine the pair $k-r,k$ as embedded in a strictly increasing sequence as $k-r=k_{j-1}$, $k=k_j$ with $r=r_j$ if $r>0$,  and in a strictly decreasing sequence if $r<0$.
Part (1) of this remark shows that the constants $C_1,\dots,C_4$ do not depend on the admissible sequence chosen for the embedding.    In the same way we can apply Proposition \ref{t18aa} (a refinement of Proposition \ref{t18} proved in section \ref{ptools}) to fixed pairs $k$, $k-r$. 

(3) Although admissible sequences are needed only for the one-sided cascade estimates (for example, in Proposition \ref{large}), part (2) of this remark allows us to use the above two propositions in the two-sided case as well.  

\end{rem}

The following lemma is used in defining amplication factors and in the proof of the iteration estimate. 

 \begin{lem}\label{c9a}
 Let $X_{k}=\zeta+k\frac{\beta_l}{\eps}$, $X_{k-r}=\zeta+(k-r)\frac{\beta_l}{\eps}$, where $r\in\ZZ\setminus 0$.    For $\delta\in (0,\delta_0]$ and $N_1\in \NN$ sufficiently large, 
 assume that 
 \begin{align}\label{c10}
 X_{k}\in\Gamma_{\frac{\delta}{N_1|r|}}(\beta_l), \text{ but } X_{k-r}\notin\Gamma_{\frac{\delta}{|r|}}(\beta_l).
 \end{align}
  Then 
 $|X_{k-r}|\lesssim \frac{1}{N_1}|X_k|$.
 \end{lem}

\begin{proof}
\textbf{1. }Let $\tilde X_{k-r}=t\frac{\beta_l}{\eps}$ and $\tilde X_{k}=(t+r)\frac{\beta_l}{\eps}$ be the orthogonal projections of $X_{k-r}$ and $X_{k}$ on the $\beta_l$ axis, and let $\ell$ denote the common distance of  $X_{k-r}$ and $X_{k}$ from that axis.    By \eqref{c10} we have
\begin{align}\label{c11}
\frac{\ell}{|\tilde X_{k}|}\lesssim \frac{\delta}{N_1|r|}, \;\;\;\; \frac{\ell}{|\tilde X_{k-r}|}\gtrsim \frac{\delta}{|r|},
\end{align}
and thus
\begin{align}
\frac{\delta}{|r|}\frac{|t|}{\eps}\lesssim \ell\lesssim \frac{\delta}{N_1|r|}\frac{|t+r|}{\eps}, \text{ which implies }\frac{|t+r|}{N_1}\gtrsim |t|; \text{ hence }|t|\lesssim \frac{|r|}{N_1-1}.
\end{align}

\textbf{2. } Now 
\begin{align}
\ell\sim \frac{\delta}{N_1|r|}|\tilde X_{k}|\sim \frac{\delta}{N_1|r|}\frac{|t+r|}{\eps}\lesssim \frac{\delta}{N_1|r|\eps}\left(\frac{|r|}{N_1-1}+|r|\right)\lesssim \frac{\delta}{N_1\eps}, \text{ so }
\end{align}
\begin{align}
|X_{k-r}|\lesssim \ell+\frac{|t|}{\eps}\lesssim \frac{|r|}{N_1\eps},\text{ while }|X_k|\sim |\tilde X_{k}|\sim \frac{|t+r|}{\eps}\sim \frac{|r|}{\eps},
\end{align} 
establishing the lemma.

\end{proof}

\subsection{Amplification factors}\label{afactors}
\quad In this section we define the global amplification factors $\DD(\eps,k,k-r)$ that appear in the iteration estimate \eqref{tt30}.   
Each global factor $\DD(\eps,k,k-r)$ is constructed out of microlocal factors $D(\eps,k,k-r;\beta)$ that we now proceed to define.   In this discussion 
the constants $\eps_0$, $\delta_0$, $C_i$, $i=1,\dots,4$,  are those from Propositions \ref{keyt} and \ref{t18}.   We take $(\zeta,\eps,\delta,k,r)\in \Xi\times (0,\eps_0]\times (0,\delta_0]\times \ZZ\times (\ZZ\setminus 0)$ and $\beta\in\Upsilon_0^+$.


First, for $\delta$ and $N_1$ as in Lemma \ref{c9a} we distinguish three cases for the pair $X_k=\zeta+k\frac{\beta_l}{\eps}$, $X_{k-r}=\zeta+(k-r)\frac{\beta_l}{\eps}$:
\begin{align}\label{cases}
\begin{split}
& (I) \;X_k\in\Gamma_{\frac{\delta}{N_1|r|}}(\beta),  X_{k-r}\in\Gamma_{\frac{\delta}{|r|}}(\beta)\\
& (II) \;X_k\in\Gamma_{\frac{\delta}{N_1|r|}}(\beta),  X_{k-r}\notin\Gamma_{\frac{\delta}{|r|}}(\beta)\\
& (III) \;X_k\in\Gamma_\delta(\beta)\setminus \Gamma_{\frac{\delta}{N_1|r|}}(\beta)
\end{split}
\end{align}

The source of amplification is the factor $[Br_-(\eps,k)]^{-1}$ that appears (in two places) in \eqref{a6}.  
  From Lemma \ref{t6} we have
\begin{align}\label{ta28}
\begin{split}
&\;(a) \left|[Br_-(\eps,k)]^{-1}\right|\lesssim |\Delta(\eps,k)|^{-1}\lesssim \frac{|X_k|}{\gamma}\text{ for }X_k\in\Gamma_\delta(\beta)\\
&\;(b) |\Delta(\eps,k)|^{-1}\leq C(\delta)|r|\text{ for }X_k\notin \Gamma_{\frac{\delta}{|r|}}(\beta).
\end{split}
\end{align}
Using \eqref{ta28}(b) we obtain
\begin{align}\label{tb28}
 |\Delta(\eps,k)|^{-1}\leq C(\delta)N_1|r| \text{ in case }(III).
 \end{align}
In case $(II)$ we have using Proposition \ref{keyt} and Remark \ref{uniform}:
\begin{align}\label{tc28}
\frac{1}{|\Delta(\eps,k)|}=\frac{1}{|\Delta(\eps,k-r)|}\frac{|\Delta(\eps,k-r)|}{|\Delta(\eps,k)|}\leq \begin{cases}(C(\delta)|r|)(C_1|r|)=C(\delta)r^2 \text{ if \eqref{t11za} holds}\\(C(\delta)|r|)(\frac{C_2|r|}{\eps\gamma})=\frac{C(\delta)r^2}{\eps\gamma}\text{ if \eqref{t11za} fails}\end{cases}.
\end{align}
In  case $(I)$, which by Proposition \ref{t18} can only happen when $\beta=\beta_l$, we will use the functions $E_{i,j}(\eps,k,k-r)$ to control $|\Delta(\eps,k)|^{-1}$.    For each $(i,j)\in \cO\times\cI$ we define case $(Ia)$ to be the subcase of case $(I)$ where \eqref{t18w} holds and $(Ib)$  the subcase where \eqref{t18w} fails.






\begin{defn}\label{t28}[Microlocal amplification factors]
Let $C_5\geq 1$ be a constant depending  on $C(\delta)$ as in  \eqref{tb28}, \eqref{tc28}, $N_1$ as in \eqref{cases}, and the constants $C_1,\dots,C_4$ appearing in Propositions \ref{keyt} and \ref{t18}.\footnote{The choice of $C_5$ is further clarified in the proof of Proposition \ref{tt30}.} 

For $k\in\ZZ$, $r\in \ZZ\setminus 0$, $\beta\in\Upsilon^+_0$, and $(\zeta,\eps)\in \mathrm{supp}\;\chi_b(\eps,k;\beta) \times (0,\eps_0]$, we define:

$\bullet$ $D(\eps,k,k-r;\beta)(\zeta)=C_5|r|$ in case $(III)$

$\bullet$ If $\beta\neq \beta_l$, then $D(\eps,k,k-r;\beta)(\zeta)=\begin{cases}C_5r^2\text{ in case }(II)\text{ when }\eqref{t11za}\text{ holds }\\\frac{C_5r^2}{\eps\gamma}\text{ in case }(II)\text{ when }\eqref{t11za}\text{ fails }\end{cases}$

 $\bullet$ $D(\eps,k,k-r;\beta_l)(\zeta)=C_5|r|$ in case $(II)$.
 
 \noindent If $\beta=\beta_l$ and case $(I)$ obtains, then 
 
 $\bullet$ $D(\eps,k,k-r;\beta_l)(\zeta)=\begin{cases} C_5|r|\text{ when for every }(i,j)\in \cO\times\cI, \text{ case }(Ia)\text{ holds}\\\frac{C_5|r|}{\eps\gamma}\text{ when for some }(i,j)\in \cO\times\cI, \text{ case }(Ib)\text{ holds}\end{cases}$.

\noindent For any $\beta$ if $\zeta\notin \mathrm{supp}\;\chi_b(\eps,k,\beta)$, define $D(\eps,k,k-r;\beta)(\zeta)=0$.
\end{defn}

\begin{defn}\label{t29}[Global amplification factors]
Let $\chi_{\CalB}(\zeta)$ be the characteristic function of $\cB:=\cup_{\beta\in\Upsilon^+_0}\Gamma_\delta(\beta)$.
For $k\in\ZZ$, $r\in\ZZ\setminus 0$,  and $(\zeta,\eps)\in \mathrm{supp}\;\chi_\CalB(\eps,k) \times (0,\eps_0]$, we define
\begin{align}\label{ty29}
\begin{split}
&\DD(\eps,k,k-r)(\zeta)=\begin{cases}\frac{C_5r^2}{\eps\gamma},\text{ if }D(\eps,k,k-r;\beta)(\zeta)=\frac{C_5|r|}{\eps\gamma}\text{ for some }\beta\in\Upsilon^+_0\\C_5r^2, \; otherwise\end{cases}.
\end{split}
\end{align}
\noindent If $\zeta\notin \mathrm{supp}\;\chi_\CalB(\eps,k)$  we set $\DD(\eps,k,k-r)(\zeta)=1$.

\end{defn}
Observe that 
\begin{align}\label{t30}
D(\eps,k,k-r;\beta)(\zeta)\leq \DD(\eps,k,k-r)\text{ for all }\zeta\in\Xi, \;\beta\in\Upsilon^+_0.
\end{align}

\begin{rem}\label{ty30}
The definition of $D(\eps,k,k-r;\beta)$ in case $(II)$ when $\beta\neq \beta_l$ is the only reason we need $r^2$ instead of $|r|$ in the definition of $\DD(\eps,k,k-r)$.  In any problem 
where $\Upsilon^+_0=\{\beta_l\}$, we replace $r^2$ by $|r|$ on the right in \eqref{ty29}.  This remark is used in the proof of Theorem \ref{tvv29}.

\end{rem}

\subsection{Iteration estimate}\label{iterationestimate}

\qquad In this section we prove the iteration estimate for the transformed singular problem \eqref{i9}, that is, 
\begin{align}\label{tz30}
\begin{split}
&(a) D_{x_2}V_k-\mathcal{A}(X_k)V_k=i\sum_{r\in \ZZ\setminus 0}  \ar e^{ir\frac{\omega_N(\beta_l)}{\eps}x_2}B_2^{-1}MV_{k-r}+\widehat{F^\eps_k}(x_2,\zeta)\\
&(b) BV_k=\widehat{G_k}(\zeta)\text{ on }x_2=0.
\end{split}
\end{align}
This is the transform of the singular problem \eqref{i6} in the case where 
we take 
\begin{align}\label{tz31}
\cD(\theta_{in})=d(\theta_{in})M \text{ for }d(\theta_{in})=\sum_{r\in\ZZ\setminus 0}\alpha_re^{ir\theta_{in}}.  
\end{align}
We justify the reduction to this case in Remark \ref{reduction}.

 We begin by defining  the objects that appear in the estimate.   Let us set 
\begin{align}\label{tu28z}
\mathcal{B}:=\cup_{\beta\in\Upsilon^+_0}\Gamma_\delta(\beta)\subset \Xi.
\end{align}
Using Notations \ref{Xk} we write the solution $V_k$ of the Fourier-Laplace transformed singular system \eqref{tz30} as 
\begin{align}
V_k=\sum_{\beta\in\Upsilon^+_0}\chi_b(\eps,k;\beta)V_k+\chi_g(\eps,k)V_k,
\end{align}
where $\chi_g(\zeta)$ is the characteristic function of $\mathcal{B}^c$ (complement in $\Xi$). 
Define $w_k=(w^+_k,w^-_k)$  for \emph{all} $\zeta\in\Xi$ by\footnote{It is important now to keep track of the dependence of  the various objects on $\beta\in\Upsilon^+_0$.} 
\begin{align}
V_k(x_2,\zeta)=S(\eps,k;\beta)w_k(x_2,\zeta;\beta),
\end{align}
 and observe that $w_k$ is for $\zeta\in\mathrm{supp}\;\chi_b(\eps,k;\beta)$ a solution of the diagonalized system\footnote{Recall $\zeta\in\mathrm{supp}\;\chi_b(\eps,k;\beta)$ if and only if $X_k:=\zeta+k\frac{\beta_l}{\eps}\in\Gamma_\delta(\beta)$.} 
     \begin{align}\label{a3a}
\begin{split}
&D_{x_2}w_k-\begin{pmatrix}\xi_+(\eps,k;\beta)&0\\0&\xi_-(\eps,k;\beta)\end{pmatrix}w_k=\\
&\qquad i\sum_{r\in\ZZ\setminus 0}\ar e^{ir\frac{\omega_N(\beta_l)}{\eps}x_2}S^{-1}(\eps,k;\beta)B_2^{-1}M S(\eps,k-r;\beta)w_{k-r}+S^{-1}(\eps,k;\beta)\widehat{F^\eps_k}(x_2,\zeta),\\
&BS(\eps,k;\beta)w_k= \hat G_k\text{ on }x_2=0.
\end{split}
\end{align}
 Here we have set\footnote{In the notation used in \eqref{a3b}, the frequency $\omega_N(\beta_l)$ that appears in \eqref{a3a} is $\omega_N(\zeta;\beta_l)|_{\zeta=\beta_l}.$} 
 \begin{align}\label{a3b}
 \begin{split}
 &\xi_+(\eps,k;\beta)=\mathrm{diag}\;(\omega_1(\eps,k;\beta),\dots,\omega_{N-p}(\eps,k;\beta))\\
 &\xi_-(\eps,k;\beta)=\mathrm{diag}\;(\omega_{N-p+1}(\eps,k;\beta),\dots,\omega_{N}(\eps,k;\beta)).
\end{split}
\end{align}

Solutions to the diagonalized system in the case $F=0$ are 
\begin{align}\label{a5}
w^+_k(x_2,\zeta)=\sum_{r\in\ZZ\setminus 0}\int^\infty_{x_2}e^{i\xi_+(\eps,k)(x_2-s)+ir\frac{\omega_N(\beta_l)}{\eps}s}\ar[a(\eps,k,k-r)w^+_{k-r}(s,\zeta)+b(\eps,k,k-r)w^-_{k-r}(s,\zeta)]ds,
\end{align}
\begin{align}\label{a6}
\begin{split}
&w^-_k(x_2,\zeta)=-\sum_{r\in\ZZ\setminus 0}\int^{x_2}_0e^{i\xi_-(\eps,k)(x_2-s)+ir\frac{\omega_N(\beta_l)}{\eps}s}\ar[c(\eps,k,k-r)w^+_{k-r}(s,\zeta)+d(\eps,k,k-r)w^-_{k-r}(s,\zeta)]ds-\\
&e^{i\xi_-(\eps,k)x_2}[Br_-(\eps,k)]^{-1}Br_+(\eps,k)\sum_{r\in\ZZ\setminus 0}\int^\infty_{0}e^{i\xi_+(\eps,k)(-s)+ir\frac{\omega_N(\beta_l)}{\eps}s}\ar[a(\eps,k,k-r)w^+_{k-r}(s,\zeta)+\\
&\qquad\qquad\qquad b(\eps,k,k-r)w^-_{k-r}(s,\zeta)]ds+e^{i\xi_-(\eps,k)x_2}[Br_-(\eps,k)]^{-1}\hat G_k(\zeta).
\end{split}
\end{align}
The matrices $a,b,c,d$ in \eqref{a5}, \eqref{a6} can be read off from \eqref{a3a} and we have
\begin{align}\label{a7}
|a(\eps,k,k-r)|\lesssim 1, |b(\eps,k,k-r)|\lesssim 1, |c(\eps,k,k-r)|\lesssim 1, |d(\eps,k,k-r)|\lesssim 1
\end{align}
uniformly with respect to $(\eps,k,r)$. 

We will also need the following expressions for $w^\pm_k$:
\begin{align}\label{ab7}
w^+_k(x_2,\zeta)=\sum_{r\in\ZZ\setminus 0}\int^\infty_{x_2}e^{i\xi_+(\eps,k)(x_2-s)+ir\frac{\omega_N(\beta_l)}{\eps}s}\ar M^+(\eps,k,k-r)V_{k-r}(s,\zeta)ds
\end{align}
  \begin{align}\label{ab8}
\begin{split}
&w^-_k(x_2,\zeta)=-\sum_{r\in\ZZ\setminus 0}\int^{x_2}_0e^{i\xi_-(\eps,k)(x_2-s)+ir\frac{\omega_N(\beta_l)}{\eps}s}\ar M^-(\eps,k,k-r)V_{k-r}(s,\zeta)ds-\\
&e^{i\xi_-(\eps,k)x_2}[Br_-(\eps,k)]^{-1}Br_+(\eps,k)\sum_{r\in\ZZ\setminus 0}\int^\infty_{0}e^{i\xi_+(\eps,k)(-s)+ir\frac{\omega_N(\beta_l)}{\eps}s}\ar M^+(\eps,k,k-r)V_{k-r}(s,\zeta)ds+\\
&\qquad \qquad e^{i\xi_-(\eps,k)x_2}[Br_-(\eps,k)]^{-1}\hat G_k(\zeta),
\end{split}
\end{align}
  where the matrices $M^\pm$ are defined in the obvious way and satisfy $|M^\pm(\eps,k,k-r)|\lesssim 1$.

Next define 
\begin{align}\label{tu29}
\mathcal{W}_k(x_2,\zeta;\beta)=(\tilde w^+_k(x_2,\zeta;\beta),w^-_k(x_2,\zeta;\beta),\frac{1}{\sqrt{\gamma}}\tilde w^+_k(0,\zeta;\beta),\frac{1}{\sqrt{\gamma}}w^-_k(0,\zeta;\beta)),
\end{align}
where we have set  
\begin{align}
\tilde w^+_k(x_2,\zeta;\beta)=\Delta^{-1}(\eps,k;\beta)w^+_k(x_2,\zeta;\beta).
\end{align}
For each $k$ we  define a modified $L^2$ norm of $V_k$ 
 by\footnote{In \eqref{tu30} the notation $|\cdot|_{L^2}$ means $|\cdot|_{L^2(x_2,\sigma,\eta)}$ for components that depend on $x_2$ and  $|\cdot|_{L^2(\sigma,\eta)}$ for components that do not.}
\begin{align}\label{tu30}
\|V_k\|_k=\sum_{\beta\in\Upsilon^+_0}|\chi_b(\eps,k;\beta)\CalW_k(x_2,\zeta;\beta)|_{L^2}+\left|\chi_g(\eps,k)\left(V_k(x_2,\zeta),\frac{1}{\sqrt{\gamma}}V_k(0,\zeta)\right)\right|_{L^2}, 
\end{align}
This ``partial norm"   depends on $k$ (through  $\Delta^{-1}$ and the cutoffs $\chi_b$, $\chi_g$).    It should be viewed as a piece of the ``full norm" $|(\|V_k\|)|_{\ell^2(k)}$ of $(V_k)_{k\in\ZZ}$ that is estimated in Proposition  \ref{bb6}.  We usually suppress the outer subscript $k$ of $\|V_k\|_k$.    If $f(\zeta)$ is any function of $\zeta$ then
\begin{align}\label{tuu30}
\|f(\zeta)V_k\|_k:=\sum_{\beta\in\Upsilon^+_0}|\chi_b(\eps,k;\beta)f(\zeta)\CalW_k(x_2,\zeta;\beta)|_{L^2}+\left|\chi_g(\eps,k)f(\zeta)\left(V_k(x_2,\zeta),\frac{1}{\sqrt{\gamma}}V_k(0,\zeta)\right)\right|_{L^2}.
\end{align}

Note that for $U^\gamma$ as in Theorem \ref{tt26}
\begin{align}\label{tu31}
|(\|V_k\|)|_{\ell^2(k)}\gtrsim |U^\gamma|_{L^2(t,x,\theta)}+\left|\frac{U^\gamma(0)}{\sqrt{\gamma}}\right|_{L^2(t,x_1,\theta)}.
\end{align}


\begin{prop}[Iteration estimate]\label{tt30}
Under assumptions \ref{assumption1}, \ref{assumption2}, \ref{assumption3} 
consider the transformed singular problem \eqref{tz30}.   
There exist positive constants $C$, $\gamma_0$ such that for $\gamma\geq \gamma_0$ the solution $V_k$ of \eqref{tt26a} satisfies for $k\in\ZZ$,
\begin{align}\label{t30c}
\|V_k\|\leq  \frac{C}{\gamma}\sum_{r\in\ZZ\setminus 0}\sum_{t\in\ZZ}\|\ar\at \DD(\eps,k,k-r)V_{k-r-t}\|+\frac{C}{\gamma^2}\left|\widehat{F_k}|X_k|\right|_{L^2}+\frac{C}{\gamma^{3/2}}\left|\widehat{G_k}|X_k|\right|_{L^2(\sigma,\eta)}.
\end{align}
Here we redefine $\alpha_0$ to be $1$.\footnote{In \eqref{t30c} $\|V_k\|=\|V_k\|_k$ and  $\|\ar\at \DD(\eps,k,k-r)V_{k-r-t}\|=\|\ar\at \DD(\eps,k,k-r)V_{k-r-t}\|_{k-r-t}$.}
\end{prop}

\begin{proof}

\textbf{1. Kreiss-type estimate. } On the support of $\chi_g(\zeta)$ the problem $(L(\partial),B)$ satisfies the uniform Lopatinski condition, so by using a singular Kreiss symmetrizer as in \cite{W1}, we obtain the Kreiss-type estimate
\begin{align}\label{t31}
\left|\chi_g(\eps,k)\left(V_k,\frac{1}{\sqrt{\gamma}}V_k(0)\right)\right|_{L^2}\leq C\left(\frac{1}{\gamma}\left(\sum_{r\in\ZZ\setminus 0}|\ar V_{k-r}|_{L^2}+|\widehat{F_k}|_{L^2}\right)+\frac{1}{\sqrt{\gamma}}|\widehat{G_k}|_{L^2(\sigma,\eta)}\right).
\end{align}
Observe that the right side of \eqref{t31} is dominated by the right side of \eqref{t30c}.

\textbf{2. Strategy. }To prove \eqref{t30c} it will suffice to show for each $\beta\in\Upsilon^+_0$ that 
\begin{align}\label{t32}
\begin{split}
&|\chi_b(\eps,k;\beta)\mathcal{W}_k|_{L^2}\leq \\
&\quad \frac{C}{\gamma}\sum_{r\in\ZZ\setminus 0}\sum_{t\in\ZZ}\left|\ar\at  D(\eps,k,k-r;\beta)\mathcal{W}_{k-r-t}\right|_{L^2}+\frac{C}{\gamma^2}\left|\widehat{F_k}|X_k|\right|_{L^2}+\frac{C}{\gamma^{3/2}}\left|\widehat{G_k}|X_k|\right|_{L^2(\sigma,\eta)},
\end{split}
\end{align}
where the $\mathcal{W}_j$ are all evaluated at $(x_2,\zeta;\beta)$.   We can then deduce \eqref{t30c} using \eqref{t30} and \eqref{t31}.

\textbf{3. }We now fix $\beta\in\Upsilon^+_0$ and proceed to prove \eqref{t32}.   
We will first treat the case $F=0$, $G=\sum_{k\in\ZZ}G_k(t,x_1)e^{ik\theta}$. 
A crude estimate  based just on applying Proposition \ref{g1z} to ``$\frac{1}{\Delta(\eps,k)}$\eqref{a5}" and  \eqref{a6} yields\footnote{In \eqref{t33} the cutoff $\chi_b(\eps,k)=\chi_b(\eps,k;\beta)$;  similarly,  we often suppress  the ``;$\beta$"  in other functions when a given $\beta\in \Upsilon^+_0$ is fixed by the context.}
 \begin{align}\label{t33}
\begin{split}
&(a) |\chi_b(\eps,k)\tilde w^+_k|_{L^2}\leq \frac{C}{\gamma}\sum_{r\in\ZZ\setminus 0}\left(\left|\frac{\chi_b(\eps,k)\ar w^+_{k-r}}{\Delta(\eps,k)}\right|_{L^2}+\left|\frac{\chi_b(\eps,k)\ar w^-_{k-r}}{\Delta(\eps,k)}\right|_{L^2}\right),\\
&(b) |\chi_b(\eps,k)w^-_k|_{L^2}\leq \frac{C}{\gamma}\sum_{r\in\ZZ\setminus 0}\left(\left|\frac{\chi_b(\eps,k)\ar w^+_{k-r}}{\Delta(\eps,k)}\right|_{L^2}+\left|\frac{\chi_b(\eps,k)\ar w^-_{k-r}}{\Delta(\eps,k)}\right|_{L^2}\right)+\frac{C}{\gamma^{3/2}}\left|\widehat{G_k}|X_k|\right|_{L^2(\sigma,\eta)}.
\end{split}
 \end{align}
Observe that for each $r$ we have
\begin{align}\label{twa33}
\left|\frac{\chi_b(\eps,k)\ar w^+_{k-r}}{\Delta(\eps,k)}\right|_{L^2}+\left|\frac{\chi_b(\eps,k)\ar w^-_{k-r}}{\Delta(\eps,k)}\right|_{L^2}\sim \left|\frac{\chi_b(\eps,k)\ar V_{k-r}}{\Delta(\eps,k)}\right|_{L^2},
\end{align}
so we can (and sometimes will) write \eqref{t33} using $V_{k-r}$ in place of $(w^+_{k-r},w^-_{k-r})$ on the right.
We proceed to improve the estimate \eqref{t33}.

\textbf{4. Decomposition of $\Gamma_\delta(\beta)$; the $\chi^3_b(\eps,k,k-r)$ pieces. } 
Considering  the three cases listed in \eqref{cases}, 
for a given $r\in\ZZ\setminus 0$ we write
\begin{align}
\chi_b(\eps,k)=\chi^1_b(\eps,k,k-r)+\chi^2_b(\eps,k,k-r)+\chi^3_b(\eps,k,k-r),
\end{align}
where $\chi^i_b(\eps,k,k-r)$, $i=1,2,3$ are respectively the characteristic functions of 
\begin{align}
\begin{split}
&\mathcal{A}_1(\eps,k,k-r):=\{\zeta\in\Xi: X_k\in\Gamma_{\frac{\delta}{N_1|r|}}(\beta), X_{k-r}\in\Gamma_{\frac{\delta}{|r|}}(\beta)\}\\
&\mathcal{A}_2(\eps,k,k-r):=\{\zeta\in\Xi: X_k\in\Gamma_{\frac{\delta}{N_1|r|}}(\beta), X_{k-r}\notin\Gamma_{\frac{\delta}{|r|}}(\beta)\}\\
&\mathcal{A}_3(\eps,k,k-r):=\{\zeta\in\Xi: X_k\in \Gamma_\delta(\beta)\setminus \Gamma_{\frac{\delta}{N_1|r|}}(\beta_l)\}.
\end{split}
\end{align}
Here $N_1\in\NN$ is chosen as in Proposition \ref{c9a} and $\delta\in (0,\delta_0]$.    By Proposition \ref{t18}(1) $\mathcal{A}_1(\eps,k,k-r)$ is empty when $\beta\neq \beta_l$.

Since $|\Delta(\eps,k)|^{-1}\leq C(\delta)N_1 |r|$ on $\mathrm{supp}\;\chi^3_b(\eps,k,k-r)$, we obtain:
\begin{align}\label{tw34}
\frac{C}{\gamma}\left|\frac{\chi^3_b(\eps,k,k-r)w^\pm_{k-r}}{\Delta(\eps,k)}\right|_{L^2}\leq \frac{C(\delta)N_1|r|}{{\gamma}}|w^\pm_{k-r}|_{L^2}\leq \frac{C}{{\gamma}}|D(\eps,k,k-r;\beta)w^\pm_{k-r}|_{L^2}.
\end{align}



\textbf{5. The $\chi^2_b(\eps,k,k-r)$ pieces when $\beta\neq \beta_l$. }
Consider now the piece of the second term on the right of \eqref{t33}(a) given by 
$\frac{C}{\gamma}\left|\frac{\chi^2_b(\eps,k,k-r)\ar w^-_{k-r}}{\Delta(\eps,k)}\right|_{L^2}$, 
 which arises from the following part of $\frac{\chi^2_b(\eps,k,k-r)}{\Delta(\eps,k)}$\eqref{a5}:
\begin{align}\label{t42}
\begin{split}
&A_2:=\frac{\chi^2_b(\eps,k,k-r)}{\Delta(\eps,k)}\int^\infty_{x_2}e^{i\xi_+(\eps,k)(x_2-s)+i\frac{\omega_N(\beta_l)}{\eps}s} \ar b(\eps,k,k-r)w^-_{k-r}(s,\zeta)ds,
\end{split}
\end{align}
Using \eqref{tc28} and Definition \ref{t28} we obtain
\begin{align}
 \left|\frac{\chi^2_b(\eps,k,k-r)}{\Delta(\eps,k)}\right|\lesssim D(\eps,k,k-r;\beta).
\end{align}
This yields
\begin{align}\label{t44}
|A_2|_{L^2}\leq \frac{C}{\gamma} |\ar D(\eps,k,k-r;\beta)w^-_{k-r}|_{L^2}.
\end{align}
The $\chi^2_b(\eps,k,k-r)$ pieces of the other terms on the right in \eqref{t33} are estimated in the same way.
 

\textbf{6. The $\chi^2_b(\eps,k,k-r)$ pieces when $\beta = \beta_l$. }For $i\in\cO$ we define the $N\times N$ matrix
\begin{align}\label{h5}
E_i(\eps,k,k-r)=\left(\omega_i(X_k)-r\omega_N(\beta_l/\eps)\right)I_N-\cA(X_{k-r}).
\end{align}
Lemma \ref{c9a} implies that on the support of $\chi^2_b(\eps,k,k-r;\beta_l))$, the first two terms on the right of \eqref{h5} dominate the third.  More precisely,  we have
\begin{align}\label{h4}
|X_{k-r}|=\left|X_k-\frac{r\beta_l}{\eps}\right|\leq \frac{1}{N_1}|X_k|,
\end{align}
so, after enlarging $N_1$ if necessary, we have\footnote{Use \eqref{h4} and \eqref{t6a}(e) for the last $\sim$ of \eqref{h6}.}
\begin{align}\label{h6}
|\cA(X_{k-r})|\lesssim |X_{k-r}|\leq \frac{1}{N_1}|X_k|\sim \frac{1}{N_1} |\omega_i(X_k)-\omega_N(r\beta_l/\eps)|.
\end{align}
This implies that the $N\times N$ matrix $E_i$ is invertible with 
\begin{align}\label{h6a}
|E_i^{-1}|\lesssim \frac{1}{|X_k|},  \text{ and hence }\left|\frac{1}{\Delta(\eps,k)}E_i^{-1}(\eps,k,k-r)\right|\lesssim \frac{1}{\gamma}, 
\end{align}
 since  $|\Delta(\eps,k)^{-1}|\lesssim |X_k|/\gamma$ on $\mathrm{supp}\;\chi^2_b(\eps,k,k-r)$.


For a given $r$ the main contribution to the term (in parentheses) on the right in \eqref{t33}(b) arises from the part of \eqref{ab8} given by\footnote{The ``smaller" contribution from the integrals $\int^{x_2}_0$ in \eqref{ab8} (or \eqref{a6}) is easily estimated by direct application of Proposition \ref{g1z}, so we ignore it here.}
  \begin{align}\label{h7}
 \begin{split}
 &A:=e^{i\xi_-(\eps,k)x_2} [Br_-(\eps,k)]^{-1}Br_+(\eps,k)\int^\infty_0 e^{i\xi_+(\eps,k)(-s)+i\frac{r\omega_N(\beta_l)}{\eps}s} \ar M^+(\eps,k,k-r)V_{k-r}(s,\zeta)ds.
\end{split}
\end{align}
Recalling the definitions  of $\xi_\pm$ \eqref{a3b}  and ignoring some factors in the integrand that are $O(1)$ and independent of $s$, we see that  the $q-$component of $A$ is a sum of terms of the form
\begin{align}\label{h8}
A_{q,i}:=e^{i\omega_q(\eps,k)x_2}\frac{1}{\Delta(\eps,k)}\int^\infty_0 e^{i\omega_i(\eps,k)(-s)+i\frac{r\omega_N(\beta_l)}{\eps}s} \ar M^+_i(\eps,k,k-r)V_{k-r}(s,\zeta)ds
\end{align}
 where $q\in\cI$,  $i\in \mathcal{O}$, and $M^+_i$ is the $i-$th row of the $(N-p)\times N$ matrix $M^+$. 

We  now improve the estimate of $\chi^2_{b}(\eps,k,k-r)A_{q,i}$ by an integration by parts.  
Setting 
\begin{align}\label{h8a}
V^*_{k-r}(x_2,\zeta):=e^{-i\cA(X_{k-r})x_2}V_{k-r},
\end{align}
we may rewrite $\chi^2_{b}(\eps,k,k-r)A_{q,i}$ as 
\begin{align}\label{h9}
e^{i\omega_q(\eps,k)x_2}\frac{M^+_i(\eps,k,k-r)}{\Delta(\eps,k)}\int^\infty_0 e^{-iE_{i}(\eps,k,k-r)s} \chi^2_{b}(\eps,k,k-r) \ar V^*_{k-r}(s,\zeta)ds,  
\end{align}
where $E_{i}(\eps,k,k-r)$ is given by \eqref{h5}.   From \eqref{tz30}$_{k-r}$ we obtain
\begin{align}\label{h9a}
\partial_{x_2}V^*_{k-r}=e^{-i\cA(X_{k-r})x_2}ih_{k-r},
\end{align}
where $h_{k-r}$ is the right side of \eqref{tz30}$_{k-r}(a)$.   Using \eqref{h9a},  the equation
\begin{align}
\frac{1}{-i}E_i^{-1}\frac{d}{ds}e^{-iE_is}=e^{-iE_is}, 
\end{align}
 and  integrating by parts in \eqref{h9} gives
\begin{align}\label{h10}
\begin{split}
&e^{i\omega_q(\eps,k)x_2} \frac{M^+_i(\eps,k,k-r)}{i\Delta(\eps,k)}\chi^2_{b}(\eps,k,k-r) E^{-1}_{i}(\eps,k,k-r)\cdot\\
&\qquad\qquad\left[\ar V_{k-r}(0,\zeta)+ \int^\infty_0 e^{-iE_{i}(\eps,k,k-r)s} e^{-i\cA(X_{k-r})s}\ar ih_{k-r}(s,\zeta)ds\right].
\end{split}
\end{align}
 We now obtain from \eqref{h6a} and the definition of $\cD(\eps,k,k-r;\beta_l)$  (in case $(II)$ of \eqref{cases}):
\begin{align}\label{h11}
\begin{split}
&|\chi^2_{b}(\eps,k,k-r)A_{q,i}|_{L^2}\lesssim \\
&\frac{C}{\gamma^{3/2}}\left[|\ar D(\eps,k,k-r;\beta_l) V_{k-r}(0,\zeta)|_{L^2(\sigma,\eta)}+\frac{C}{\sqrt{\gamma}}|\ar D(\eps,k,k-r;\beta_l) h_{k-r}|_{L^2}  \right]\leq \\
&\quad \frac{C}{\gamma^{3/2}}\left[|\ar D(\eps,k,k-r;\beta_l)V_{k-r}(0,\zeta)|_{L^2(\sigma,\eta)}+\frac{C}{\sqrt{\gamma}}\sum_{t\in\ZZ\setminus 0}|\ar\at D(\eps,k,k-r;\beta_l)V_{k-r-t}|_{L^2}  \right].
\end{split}
\end{align}


The estimate \eqref{h11} shows  that  the piece of the term on the right of \eqref{t33}(b) given by 
$\frac{C}{\gamma}\left|\frac{\chi^2_b(\eps,k,k-r)\ar V_{k-r}}{\Delta(\eps,k)}\right|_{L^2}$ can be \emph{replaced} by (not dominated by) 
\begin{align}\label{h12}
\frac{C}{\gamma}\sum_{t\in\ZZ}\left|\ar\at D(\eps,k,k-r;\beta_l)\CalW_{k-r-t}\right|_{L^2}.
\end{align}

The piece of the term on the right of \eqref{t33}(a) given by 
$\frac{C}{\gamma}\left|\frac{\chi^2_b(\eps,k,k-r)\ar V_{k-r}}{\Delta(\eps,k)}\right|_{L^2}$ 
 arises from the following part of $``\frac{\chi^2_b(\eps,k,k-r)}{\Delta(\eps,k)}$\eqref{ab7}":
\begin{align}\label{h13}
\begin{split}
&\frac{\chi^2_b(\eps,k,k-r)}{\Delta(\eps,k)}\int^\infty_{x_2}e^{i\xi_+(\eps,k)(x_2-s)+i\frac{r\omega_N(\beta_l)}{\eps}s} \ar M^+(\eps,k,k-r)V_{k-r}(s,\zeta)ds,
\end{split}
\end{align}
The replacement, determined by essentially the same argument as given above, is  again \eqref{h12}.

\textbf{7. Improving the $\left|\frac{\chi^1_b(\eps,k,k-r)\ar w^-_{k-r}}{\Delta(\eps,k)}\right|_{L^2}$ pieces of \eqref{t33}. }
Recall that $\chi^1_b(\eps,k,k-r;\beta)$ is nonzero only for $\beta=\beta_l$, so we assume that in this step and  the next. 
For a given $r$ the main contribution to the second term (in parentheses) on the right in \eqref{t33}(b) arises from the part of \eqref{a6} given by 
 \begin{align}\label{h14}
 \begin{split}
 &A:=e^{i\xi_-(\eps,k)x_2} [Br_-(\eps,k)]^{-1}Br_+(\eps,k)\int^\infty_0 e^{i\xi_+(\eps,k)(-s)+i\frac{r\omega_N(\beta_l)}{\eps}s} \ar b(\eps,k,k-r)w^-_{k-r}(s,\zeta)]ds.
\end{split}
\end{align}
 Ignoring some factors in the integrand that are $O(1)$ and independent of $s$, we see that the $p-$component of $A$ is a sum of terms of the form
 \begin{align}\label{h15}
A_{p,i,j}:=e^{i\omega_p(\eps,k)x_2}\frac{1}{\Delta(\eps,k)}\int^\infty_0 e^{i\omega_i(\eps,k)(-s)+i\frac{r\omega_N(\beta_l)}{\eps}s} \ar w^-_{k-r,j}(s,\zeta)ds
\end{align}
where $p\in \cI$, $i\in\cO$, and $w^-_{k-r,j}$, $j\in\mathcal{I}$ denotes a component of $w^-_{k-r}$.

 For any $(i,j)\in\cO\times \cI$ recall  the condition \eqref{t18w} involving the function 
$E_{i,j}(\eps,k,k-r):=\omega_i(\eps,k)-r\omega_N(\beta_l/\eps)-\omega_j(\eps,k-r)$:
\begin{align}\label{h16}
|E_{i,j}(\eps,k,k-r)|\geq C_3\frac{|X_k|}{|r|} \text{ or }|E_{i,j}(\eps,k,k-r)|\geq C_3|X_{k-r}|\text{ on } \mathrm{supp}\;\chi^1_b(\eps,k,k-r).
\end{align}
To estimate $\chi^1_b(\eps,k,k-r)A_{p,i,j}$ we first decompose\footnote{The cutoffs on the right in \eqref{h16a} have additional dependence on $(i,j)$, which we suppress in the notation.}
\begin{align}\label{h16a}
\chi^1_b(\eps,k,k-r)=\chi^1_{G}(\eps,k,k-r)+\chi^1_{B}(\eps,k,k-r)
\end{align}
into ``good" and ``bad" pieces supported respectively where 
the condition \eqref{h16} holds,  does not hold.


We  proceed to improve the estimate of $\chi^1_{G}(\eps,k,k-r)A_{p,i,j}$.  
We  decompose $\chi^1_G(\eps,k,k-r)$,
\begin{align}
\chi^1_G(\eps,k,k-r)=\chi^1_{G,I}(\eps,k,k-r)+\chi^1_{G,II}(\eps,k,k-r)+\chi^1_{G,III}(\eps,k,k-r),
\end{align}
into pieces where respectively, (I) the first alternative in \eqref{h16} holds,  (II) the first alternative fails and $|X_{k-r}|\geq \frac{|X_k|}{N_2}$, 
 (III) the first alternative fails and $|X_{k-r}|< \frac{|X_k|}{N_2}$;  $N_2\in\NN$ is chosen large enough below.

In case (I) we have
\begin{align}\label{h17}
\left|\frac{1}{\Delta(\eps,k)E_{i,j}}\right|\leq \frac{|X_k|}{\gamma}\frac{|r|}{C_3|X_k|}\sim \frac{|r|}{\gamma},
\end{align}
so an integration by parts just like that in step \textbf{6} gives\footnote{In this integration by parts $w^{*,-}_{k-r,j}(x_2,\zeta):=e^{-i\omega_j(\eps,k-r)x_2}w^-_{k-r,j}$  replaces $V^*_{k-r}$ as in \eqref{h8a}, and $h^-_{k-r,j}$, the $j$-component of the right side of \eqref{a3a}$_{k-r}$, replaces $h_{k-r}$ as in \eqref{h9a}.}
\begin{align}\label{h19}
\begin{split}
&|\chi^1_{G,I}(\eps,k,k-r)A_{p,i,j}|_{L^2}\lesssim \\
&\quad \frac{C}{\gamma^{3/2}}\left[|\ar D(\eps,k,k-r;\beta_l)w^-_{k-r,j}(0,\zeta)|_{L^2(\sigma,\eta)}+\frac{C}{\sqrt{\gamma}}\sum_{t\in\ZZ\setminus 0}|\ar\at D(\eps,k,k-r;\beta_l)w_{k-r-t}|_{L^2}  \right].
\end{split}
\end{align}

In case (II)  we find
\begin{align}\label{h20}
\left|\frac{1}{\Delta(\eps,k)E_{i,j}}\right|\leq \frac{|X_k|}{\gamma}\frac{1}{C_3|X_{k-r}|}\sim \frac{N_2}{\gamma},
\end{align}
so a similar integration by parts yields an estimate just like \eqref{h19} for $|\chi^1_{G,II}(\eps,k,k-r)A_{p,i,j}|_{L^2}$.

In case (III) the first two terms in the expression for $E_{i,j}$ are dominant, so we can use the argument of step \text{6} (recall \eqref{h6a}) to show 
\begin{align}\label{h21}
|E_{i,j}^{-1}|\leq \frac{C(N_2)}{|X_k|},  \text{ and hence }\left|\frac{1}{\Delta(\eps,k)E_{i,j}}\right|\leq\frac{C(N_2)}{\gamma}, 
\end{align}
provided $N_2$ is large enough. Thus, we get the estimate \eqref{h19} for $|\chi^1_{G,III}(\eps,k,k-r)A_{p,i,j}|_{L^2}$,

To estimate $\chi^1_{B}(\eps,k,k-r)A_{p,i,j}$ we do a direct estimate using \eqref{h15}, and the fact that\footnote{Use  Lemma \ref{t6} here.}
\begin{align}\label{t39}
|\Delta(\eps,k|^{-1}\lesssim \frac{|X_k|}{\gamma}\leq \frac{C_4|r|}{\eps\gamma}\text{ on }\mathrm{supp}\;\chi^1_{B}(\eps,k,k-r), \;\text {for }C_4\text{ as in }\eqref{t18z}.
\end{align}
Thus, we obtain
\begin{align}\label{t40}
|\chi^1_{B}(\eps,k,k-r)A_{p,i,j}|_{L^2}\leq \frac{C}{\gamma} \frac{C_4|r|}{\eps\gamma}|w^-_{k-r}|_{L^2}.
\end{align}


These estimates show  that  the piece of the second term on the right of \eqref{t33}(b) given by 
$\frac{C}{\gamma}\left|\frac{\chi^1_b(\eps,k,k-r)\ar w^-_{k-r}}{\Delta(\eps,k)}\right|_{L^2}$ can be \emph{replaced} by (not dominated by) \eqref{h12}. 

The piece of the second term on the right of \eqref{t33}(a) given by 
$\frac{C}{\gamma}\left|\frac{\chi^1_b(\eps,k,k-r)\ar w^-_{k-r}}{\Delta(\eps,k)}\right|_{L^2}$ 
 arises from the following part of $``\frac{\chi^1_b(\eps,k,k-r)}{\Delta(\eps,k)}$\eqref{a5}":
\begin{align}\label{h22}
\begin{split}
&\frac{\chi^1_b(\eps,k,k-r)}{\Delta(\eps,k)}\int^\infty_{x_2}e^{i\xi_+(\eps,k)(x_2-s)+i\frac{\omega_N(\beta_l)}{\eps}s} \ar b(\eps,k,k-r)w^-_{k-r}(s,\zeta)ds,
\end{split}
\end{align}
The replacement, determined by essentially the same argument as given above, is  again \eqref{h12}. 

\begin{rem}\label{gc}
Proposition \ref{goodcase} below  implies that  $\chi^1_B(\eps,k,k-r)$ can be nonzero \emph{only} in the cases $i\in\cO$, $j\in\cI\setminus\{N\}$.   Equivalently, the function $E_{i,N}$ satisfies condition \eqref{h16} for all $i\in\cO$.   This observation is used in the proof of 
Proposition \ref{large}. 

\end{rem}

\textbf{8. Improving the $\left|\frac{\chi^1_b(\eps,k,k-r)\ar w^+_{k-r}}{\Delta(\eps,k)}\right|_{L^2}$ pieces of \eqref{t33}. }
To treat the pieces involving $w^+_{k-r}$ on the right in \eqref{t33}, we first decompose $\chi^1_b(\eps,k,k-r)$,
\begin{align}
\chi^1_b(\eps,k,k-r)=\chi^1_{b,\II}(\eps,k,k-r)+\chi^1_{b,\II\II}(\eps,k,k-r),
\end{align}
into pieces where respectively, ($\II$) $|X_{k-r}|\geq \frac{|X_k|}{N_3}$, 
 ($\II\II$) $|X_{k-r}|< \frac{|X_k|}{N_3}$, for $N_3\in\NN$ to be chosen large below.
  In case ($\II$) we insert $\frac{\Delta(\eps,k-r)}{\Delta(\eps,k-r)}$ and use  (recall \eqref{t6a}(b))
 \begin{align}
 \left|\frac{\Delta(\eps,k-r)}{\Delta(\eps,k)}\right|\lesssim \frac{|X_k|}{|X_{k-r}|}\leq N_3\text{ on }\mathrm{supp}\;\chi^1_b(\eps,k,k-r)
 \end{align}
 to obtain
 \begin{align}
 \frac{C}{\gamma}\left|\frac{\chi^1_{b,\II}(\eps,k)\ar w^+_{k-r}}{\Delta(\eps,k)}\right|_{L^2}\leq \frac{C}{\gamma}|\ar D(\eps,k,k-r;\beta_l)\tilde w^+_{k-r}|_{L^2}. 
\end{align}

In case ($\II\II$) consider the contribution to the term involving $w^+_{k-r}$ on the right in \eqref{t33}(a) given by
\begin{align}\label{h23}
\begin{split}
&\frac{\chi^1_{b,\II\II}(\eps,k,k-r)}{\Delta(\eps,k)}\int^\infty_{x_2}e^{i\xi_+(\eps,k)(x_2-s)+i\frac{r\omega_N(\beta_l)}{\eps}s} \ar a(\eps,k,k-r)w^+_{k-r}(s,\zeta)ds.
\end{split}
\end{align}
If we define the $|\cO|\times |\cO|$ matrix
\begin{align}\label{h23a}
F_{+,+}(\eps,k,k-r)=\xi_+(\eps,k)-r\omega_N(\beta_l/\eps)I_{|\cO|}-\xi_+(\eps,k-r),  
\end{align}
we see by an argument parallel to  \eqref{h5}-\eqref{h6a} that the first two terms on the right of  \eqref{h23a} are dominant  for large $N_1$, so
\begin{align}\label{h24}
\left|\frac{1}{\Delta(\eps,k)}F_{+,+}^{-1}(\eps,k,k-r)\right|\lesssim \frac{1}{\gamma} \text{ on }\mathrm{supp}\;\chi^1_{b,\II\II}(\eps,k,k-r).
\end{align}
Thus, an integration by parts similar to that in step \textbf{6} shows that    the piece of the term on the right of \eqref{t33}(a) given by 
$\frac{C}{\gamma}\left|\frac{\chi^1_{b,\II\II}(\eps,k,k-r)\ar w^+_{k-r}}{\Delta(\eps,k)}\right|_{L^2}$ can be \emph{replaced} by (not dominated by) 
\begin{align}\label{h25}
\frac{C}{\gamma}\sum_{t\in\ZZ}\left|\ar\at D(\eps,k,k-r;\beta_l)\CalW_{k-r-t}\right|_{L^2}.
\end{align}

Finally, we must consider the  contribution in case ($\II\II$) to the term involving $w^+_{k-r}$ on the right in \eqref{t33}(b)  given by
\begin{align}
e^{i\xi_-(\eps,k)x_2}\chi^1_{b,\II\II}(\eps,k,k-r)[Br_-(\eps,k)]^{-1}Br_+(\eps,k)\int^\infty_{0}e^{i\xi_+(\eps,k)(-s)+ir\frac{\omega_N(\beta_l)}{\eps}s}\ar a(\eps,k,k-r)w^+_{k-r}(s,\zeta)ds.
\end{align}
The replacement, determined by essentially the same argument involving $F_{+,+}$, is  again \eqref{h25}.

\textbf{9. }Using Proposition \ref{g1z} and the formulas \eqref{a5}, \eqref{a6}, it is easy to see that $\frac{1}{\sqrt{\gamma}}|\tilde w^+_k(0,\zeta;\beta)|_{L^2(\sigma,\eta)}$ and  $\frac{1}{\sqrt{\gamma}}|w^-_k(0,\zeta;\beta))|_{L^2(\sigma,\eta)}$ satisfy, respectively, the same estimates as $|\tilde w^+_k(x_2,\zeta;\beta)|_{L^2}$, $|w^-_k(x_2,\zeta,\beta)|_{L^2}$.  
Combining the estimates of steps \textbf{4}-\textbf{8} establishes for all $\beta\in\Upsilon^+_0$ the estimate \eqref{t32} for the case $F=0$.


\textbf{10. }If $F(t,x,\theta)=\sum_{k\in\ZZ}F_k(t,x)e^{ik\theta}$,   the only modification to \eqref{a5}, \eqref{a6} is to add 
\begin{align}\label{t50}
-i\int^\infty_{x_2}e^{i\xi_+(\eps,k)(x_2-s)}\ell_+(\eps,k)\widehat{F_k}(s,\tau,\eta)ds,
\end{align}
to the right side of \eqref{a5}, and to add the terms
\begin{align}\label{t51}
\begin{split}
&i\int^{x_2}_0e^{i\xi_-(\eps,k)(x_2-s)}\ell_-(\eps,k) \widehat{F_k}(s,\tau,\eta)ds+\\
&\qquad ie^{i\xi_-(\eps,k)x_2}[Br_-(\eps,k)]^{-1}Br_+(\eps,k)\int^\infty_0e^{i\xi_+(\eps,k)(-s)}\ell_+(\eps,k) \widehat{F_k}(s,\tau,\eta)ds
\end{split}
\end{align}
to the right side of \eqref{a6}.   Using step \textbf{9} and applying Proposition \ref{g1z} again, we obtain the estimate  \eqref{t30c}.  

\end{proof}

\begin{rem}\label{t50a}In the case where the oscillatory coefficient in \eqref{tz30} has only positive spectrum, if one takes  $F(t,x,\theta)=\sum^\infty_{k=1}F_k(t,x)e^{ik\theta}$
and $G(t,x_1,\theta)=\sum^\infty_{k=1}G_k(t,x_1)e^{ik\theta}$, then the solution to \eqref{tz30} satisfies $V_k=0$ for $k<1$.    The iteration estimate \eqref{t30c} then reduces to
\begin{align}\label{tt30c}
\|V_k\|\leq  \frac{C}{\gamma}\sum^{k-1}_{r=1}\sum_{t=0}^{k-r-1}\|\ar\at \DD(\eps,k,k-r)V_{k-r-t}\|+\frac{C}{\gamma^2}\left|\widehat{F_k}|X_k|\right|_{L^2}+\frac{C}{\gamma^{3/2}}\left|\widehat{G_k}|X_k|\right|_{L^2(\sigma,\eta)}.
\end{align}
The obvious analogue of the estimate \eqref{tt30c} holds in the case where $F(t,x,\theta)=\sum^\infty_{k=N^*}F_k(t,x)e^{ik\theta}$
and $G(t,x_1,\theta)=\sum^\infty_{k=N^*}G_k(t,x_1)e^{ik\theta}$ for any $N^*\in\ZZ$.   Moreover,  the proof of Proposition \ref{tt30} shows that the constants $C$, $\gamma_0$ appearing there are independent of $N^*$.  This remark is used in the cascade estimates of section \ref{2s}. 
\end{rem}

\subsection{Control of amplification factors}\label{ptools}

\begin{proof}[\textbf{Proof of Proposition \ref{keyt}}]  We carry out the proof for a strictly increasing sequence $(k_j)$; the decreasing case is similar.   

\textbf{1. }Recall that $X_k=\zeta+\frac{k\beta_l}{\eps}$.  Since $\beta, \beta_l\in \Upsilon^+_0$, either the counterclockwise angle $\alpha_{cc}$ from $\beta_l$ to $\beta$ is between $0$ and $\frac{\pi}{2}$, or the clockwise angle $\alpha_c$ from $\beta_l$ to $\beta$ is between $0$ and $\frac{\pi}{2}$. By symmetry of $\Gamma_\delta:=\Gamma_\delta(\beta)$ about the line $\mathcal{L(\beta)}:=\{t\beta:t\in \mathbb{R}\}$, it will suffice to consider the case where the angle $\alpha:=\alpha_{cc}\in [0,\frac{\pi}{2}]$. 

Let us define the positive ``vertical" direction to be the $\beta$ direction, and the positive ``horizontal" direction to be the direction of a vector obtained by rotating $\beta$ by $\frac{\pi}{2}$ clockwise.

\textbf{2. }For each $X_k\in \Xi$ there is a point $\tilde X_k$ closest to it on $\mathcal{L}(\beta)$.  We have 
\begin{align}\label{t11a}
\begin{split}
& X_{k_j}=X_{k_{j-1}}+\frac{r_j\beta_l}{\eps},\quad \tilde X_{k_j}=\tilde X_{k_{j-1}}+\frac{(r_j\cos\alpha)\beta}{\eps} \text{ for all }j. \\
\end{split}
\end{align}
For example if $\beta \perp \beta_l$, we have $\tilde X_{k_j}=\tilde X_{k_{j-1}}$. 
In general the vertical displacement in passing from $X_{k_j}$ to $X_{k_{j-1}}$ is $-\frac{r_j\cos \alpha}{\eps}$, while the horizontal displacement is $-\frac{r_j\sin \alpha}{\eps}$ (recall $|\beta_l|=1$).\footnote{More precisely, this is the horizontal displacement of the projection of $X_{k_j}$ into the $(\sigma,\eta)$ plane.}

Moreover,  for opening angle  $\delta>0$ small we have
\begin{align}\label{t11ya}
|X_k-\tilde X_k|\leq 2\delta |\tilde X_k|\Rightarrow (1-2\delta)|\tilde X_k|\leq |X_k|\leq (1+2\delta)|\tilde X_k|\text{ for }X_k\in\Gamma_\delta.
\end{align}



\textbf{3. }Suppose $X_{k_j}\in \Gamma_\delta$, $X_{k_{j-1}}\in\Gamma_\delta$. Then by \eqref{t6a}(b) we have for some $c_0>0$
\begin{align}\label{t12}
\left|\frac{\Delta(\eps,k_{j-1})}{\Delta(\eps,k_j)}\right|\leq c_0 \frac{|X_{k_j}|}{|X_{k_{j-1}}|}.
\end{align}
Moreover, if $|\tilde X_{k_{j-1}}|\geq \frac{1}{8\eps}$, we have by \eqref{t11a} 
\begin{align}\label{t12a}
\frac{|\tilde X_{k_j}|}{|\tilde X_{k_{j-1}}|}\leq 1+\frac{r_j\cos \alpha}{1/8}\leq 9r_j.
\end{align}
In this case by \eqref{t11ya} it follows that for $\delta>0$ small enough we have
\begin{align}\label{t13}
\left|\frac{\Delta(\eps,k_{j-1})}{\Delta(\eps,k_j)}\right|\leq c_1r_j,\text{ where }c_1=c_1(c_0,9).
\end{align}

\textbf{4. }For $N_0\geq 2$ to be chosen, suppose $X_{k_j}\notin \Gamma_{\delta/N_0}$.\footnote{We allow the freedom to take $N_0$ ``large enough" and independent of $(\zeta,\eps,j)$;  the argument below shows we could in fact take $N_0=4$.   }   Then by \eqref{t6a}(c) we have $|\Delta(\eps,k_j)|^{-1}\leq c_2=c_2(\delta,N_0)$.   Together with \eqref{t6a}(a) this implies
\begin{align}\label{t14}
\left|\frac{\Delta(\eps,k_{j-1})}{\Delta(\eps,k_j)}\right|\leq c_3=c_3(\delta,N_0).
\end{align}
Thus, the estimate \eqref{t14}  can  fail only if $X_{k_j}\in\Gamma_{\delta/N_0}$.   Now let
\begin{align}\label{t15a}
C_1=\max\{c_3(\delta,N_0), c_1(c_0,9)\} \;(\text{ recall }\eqref{t13}).
\end{align}
We will say that for a given $(\zeta,\eps)$, the transition $X_{k_j}\to X_{k_{j-1}}$ is a ``bad transition" for a particular $j$ if and only if the estimate \eqref{t11za} fails for that $j$ for $C_1$ as in \eqref{t15a}.
We will  show that for a given $(\zeta,\eps)$ and $N_0$ chosen large enough, a bad transition can occur for only \emph{one} choice of $j$ (denoted $m_1(\zeta,\eps;\beta)$), and that for $j=m_1(\zeta,\eps;\beta)$ we have \eqref{t11zb}.   
 Observe that for a fixed $(\zeta,\eps)$ the points $X_{k_j}$ as $j$ varies all lie on a line $L(\zeta,\eps,\beta_l)$ parallel to $\beta_l$, and that the transition $X_{k_j}\to X_{k_{j-1}}$ can be bad only if $X_{k_j}\in\Gamma_{\delta/N_0}$.


\textbf{5. Bad transition where $X_{k_{j-1}}\notin\Gamma_{\delta}$.} Suppose the transition $X_{k_j} \to X_{k_{j-1}}$ is bad for a given $j=m$, and that $X_{k_{m-1}}\notin\Gamma_\delta$. 
Since $X_{k_m}\in\Gamma_{\delta/N_0}$ and $X_{k_{m-1}}\notin\Gamma_\delta$, 
in view of \eqref{t11a} this can happen only if
\begin{align}\label{t15aa}
|X_{k_m}|\leq \frac{c_4(\delta,\beta)r_m}{\eps},
\end{align}
where $c_4(\delta,\beta)$ is independent of $(\zeta,\eps,j)$.\footnote{To see this consider first the case $r_m=1$, then rescale.}  From this and \eqref{t6a}(a),(b) we obtain the estimate
\begin{align}\label{t15}
\left|\frac{\Delta(\eps,k_{m-1})}{\Delta(\eps,k_m)}\right|\leq \frac{c_5(\delta,\beta)r_m}{\eps\gamma}.
\end{align}

To see that the transition $X_{k_j}\to X_{k_{j-1}}$ is not bad if $j\neq m$,  suppose first that $\alpha\geq \delta$.   Since $X_{k_{m-1}}\notin\Gamma_\delta$, we have $X_{k_{m-l}}\notin \Gamma_\delta$ for all $l\geq 1$, so there can be no bad transitions for $j<m$. 
The transition $X_{k_{m+1}}\to X_{k_m}$ can be bad only if $X_{k_{m+1}}\in \Gamma_{\delta/N_0}$ and $|\tilde X_{k_m}|\leq \frac{1}{8\eps}$ (recall step \textbf{3}). But the horizontal displacement in passing from $X_{k_m}$ to $X_{k_{m+1}}$ is $\frac{r_{m+1}\sin\alpha}{\eps}\geq \frac{r_{m+1}\sin\delta}{\eps}\sim\frac{r_{m+1}\delta}{\eps}$, so if $N_0\geq 4$ one cannot have $X_{k_{m+1}}\in\Gamma_{\delta/N_0}$ and $|\tilde X_{k_m}|\leq \frac{1}{8\eps}$.\footnote{The width of $\Gamma_{\delta/N_0}$ at a vertical height of $\frac{1}{8\eps}$ is $\sim 2\cdot\frac{1}{8\eps}\cdot \frac{\delta}{N_0}$; at a vertical height of $\frac{9r_{m+1}}{8\eps}$ it is $\sim 2\cdot\frac{9r_{m+1}}{8\eps}\cdot \frac{\delta}{N_0}<\frac{2r_{m+1}\delta}{3\eps}$ for $N_0\geq 4$.}  Similarly, there are no bad  transitions $X_{k_j}$ to $X_{k_{j-1}}$ for $j>m+1$ when $\alpha\geq \delta$. 

Now suppose $\alpha<\delta$.  A transition starting from  $X_{k_{m-l}}$ for $l>1$ can be bad only if $X_{k_{m-l}}\in\Gamma_{\delta/N_0}$, and  since $X_{k_{m-1}}\notin\Gamma_\delta$, that can happen for some $l_0>1$ only  if $\alpha<\delta/N_0$.   But then $X_{k_{m-l}}\in \Gamma^-_{\delta/N_0}$ for all $l\geq l_0$ and $|X_{k_{m-l-1}}|> |X_{k_{m-l}}|$, so \eqref{t12} shows the transition $X_{k_{m-l}}\to X_{k_{m-l-1}}$ cannot be bad.
Next suppose the transition $X_{k_{m+1}}\to X_{k_m}$ is bad when $\alpha<\delta$.  The vertical and horizontal displacements are $-\frac{r_{m+1}\cos\alpha}{\eps}\sim-\frac{r_{m+1}}{\eps}$ and $-\frac{r_{m+1}\sin\alpha}{\eps}$, respectively.   Since both  $X_{k_{m+1}}, X_{k_m}\in \Gamma_{\delta/N_0}$, it follows by step \textbf{3} that $|\tilde X_{k_m}|<\frac{1}{8\eps}$. 
But then a vertical displacement of $\sim \frac{r_{m+1}}{\eps}$ starting from $X_{k_m}\in\Gamma_{\delta/N_0}$ would land $X_{k_{m-1}}\in\Gamma_\delta$.  Step \textbf{3} implies there are no bad  transitions $X_{k_j}\to X_{k_{j-1}}$ for $j>m+1$ when $\alpha <\delta$.

\textbf{6. Bad transition where $X_{k_{j-1}}\in\Gamma_{\delta}$.}  Suppose the transition $X_{k_j} \to X_{k_{j-1}}$ is bad for a given $j=m$, and that $X_{k_{m-1}}\in\Gamma_\delta$. 
Since $X_{k_m}\in\Gamma_{\delta/N_0}$ and $X_{k_{m-1}}\in\Gamma_\delta$, by step \textbf{3} this can happen only if $|\tilde X_{k_{m-1}}|\leq \frac{1}{8\eps}$, so by \eqref{t11ya} and \eqref{t12} we have
\begin{align}\label{t16}
\left|\frac{\Delta(\eps,k_{m-1})}{\Delta(\eps,k_m)}\right|\leq \frac{c_6(c_0)r_m}{\eps\gamma}
\end{align}
 since $|\tilde X_{k_m}|\leq \frac{9r_m}{8\eps}$.

 If $X_{k_{m-1}}\notin\Gamma_{\delta/N_0}$, the transition $X_{k_{m-1}}\to X_{k_{m-2}}$ is not bad, and in fact no transition $X_{k_{m-l}}\to X_{k_{m-l-1}}$ is bad for $l>1$.\footnote{Consider the cases $\alpha\geq \delta$, $\alpha<\delta$.}  The transition $X_{k_{m+1}}\to X_{k_m}$ can be bad only if $X_{k_{m+1}}\in\Gamma_{\delta/N_0}$.  But then, as we saw above, $\alpha<\delta$ and the vertical displacement for $\tilde X_{k_{m-1}}\to \tilde X_{k_m}$ is close to $\frac{r_{m}}{\eps}$.  By \eqref{t12a} we must have
\begin{align}
\frac{|\tilde X_{k_{m+1}}|}{|\tilde X_{k_m}|}\leq 9r_{m+1}, \text{ since }|\tilde X_{k_m}|\geq \frac{1}{8\eps},
\end{align}
so the transition $X_{k_{m+1}}\to X_{k_m}$ cannot be bad.  Similarly, no transition $X_{k_{m+l}}\to X_{k_{m+l-1}}$ can be bad for $l>1$.

If $X_{k_{m-1}}\in \Gamma_{\delta/N_0}$, then $\alpha < \delta$ and  the vertical displacement for $X_{k_{m-1}}\to X_{k_{m-2}}$ is close to $\frac{r_{m-1}}{\eps}$,  so $X_{k_{m-2}}\in \Gamma_\delta$ (since $|\tilde X_{k_{m-1}}|\leq \frac{1}{8\eps}$).   Moreover, $|\tilde X_{k_{m-2}}|\geq |\tilde X_{k_{m-1}}|$, so by \eqref{t12}, \eqref{t12a} the transition $X_{k_{m-1}}\to X_{k_{m-2}}$ cannot be bad.  Similarly, no transition $X_{k_{m-l}}\to X_{k_{m-l-1}}$ is bad for $l>1$, and no transition $X_{k_{m+l}}\to X_{k_{m+l-1}}$ is bad for $l\geq 1$.

\textbf{7. } We now take $C_2:=\max\{c_5(\delta,\beta),c_6(c_0)\}$ in \eqref{t11zb} to finish  the proof. 


\end{proof}





In preparation for the proof of Proposition \ref{t18}, we recall that for an admissible sequence $(k_p)_{p\in\ZZ}$ (Definition \ref{admissible}),    
for $(\zeta,\eps)\in\Xi\times (0,\eps_0]$ we define the function of $\zeta$:
\begin{align}\label{t17a}
\begin{split}
&E_{i,j}(\eps,k_p,k_{p-1}):=\omega_i(\eps,k_p;\beta_l)-\frac{r_p\omega_N(\beta_l)}{\eps}-\omega_j(\eps,k_{p-1};\beta_l), \text{ where }i\in\mathcal{O}, \;j\in\mathcal{I}.\\
\end{split}
\end{align}
Here we have $\omega_i(\eps,k;\beta_l)(\zeta)=\omega_i(\zeta;\beta_l)|_{\zeta=X_k}$, and, moreover, $\omega_N(\beta_l)=\omega_N(\zeta;\beta_l)|_{\zeta=\beta_l}$.



\begin{proof}[\textbf{Proof of Proposition \ref{t18}}]
We will carry out the proof for a strictly increasing sequence $(k_p)$; the decreasing case is similar.  We fix $(\eps,\zeta)$ throughout the proof and fix $p\in \ZZ$ in steps \textbf{1-5}. 

\textbf{1. }  For any given $\beta\in \Upsilon^+_0\setminus \{\beta_l\}$, let $\alpha\in (0,\frac{\pi}{2}]$ be the angle it makes with $\beta_l$, and let $\alpha_1>0$ be the smallest such angle.
Let $\tilde \Gamma(p)$ be the portion of $\Gamma_{\frac{\delta}{r_p}}(\beta)$ bounded ``above" by the plane orthogonal to $\beta$ containing $X_{k_p}$ and ``below" by the  plane orthogonal to $\beta$ containing $X_{k_{p-1}}$.    Also, let $\tilde X_{k_{p-1}}=t(p)\frac{\beta}{\eps}$ be the orthogonal projection of $X_{k_{p-1}}$ on $\beta$.  Then the maximum width of 
 $\tilde{\Gamma}(p)$ is $\lesssim 2\frac{\delta}{r_p}(|t(p)|+r_p)/\eps$, while the ``horizontal" displacement in passing from $X_{k_p}$ to $X_{k_{p-1}}$ is $\frac{r_p\sin\alpha}{\eps} \geq \frac{r_p\sin\alpha_1}{\eps}$.     Thus, $X_{k_p}$ and $X_{k_{p-1}}$ cannot both lie in $\Gamma_{\frac{\delta}{r_p}}(\beta)$ if $\delta=\delta(\alpha_1)$ is small enough. 

\textbf{2. }Now let $\beta=\beta_l$ and, with $\omega_i(\zeta)=\omega_i(\zeta;\beta_l)$, write $E_{i,j}$ as in \eqref{t17a} as 
\begin{align}\label{t18a}
\begin{split}
&E_{i,j}(\eps,k_p,k_{p-1})=\omega_i(X_{k_p})-\frac{r_p\omega_N(\beta_l)}{\eps}-\omega_j(X_{k_{p-1}})\\
&\tilde E_{i,j}(\eps,k_p,k_{p-1}):=\omega_i(\tilde X_{k_p})-\frac{r_p\omega_N(\beta_l)}{\eps}-\omega_j(\tilde X_{k_{p-1}}),
\end{split}
\end{align}
for $\tilde X_{k_p}$ as in \eqref{t11a}.        We also have
\begin{align}\label{t18y}
\begin{split}
&\tilde X_{k_{p-1}}=t\frac{\beta_l}{\eps}\text{ and }\tilde X_{k_p}=\tilde X_{k_{p-1}}+\frac{r_p\beta_l}{\eps}=(t+r_p)\frac{\beta_l}{\eps} \text{ for some }t=t(p)\in \mathbb{R},\\
&X_{k_p}=X_{k_{p-1}}+\frac{r_p\beta_l}{\eps}.
\end{split}
\end{align}

\textbf{3. } Since  $\omega_j(\zeta)$ is positively homogeneous of degree one, it follows from the definition of $\omega_j(\zeta)$ on $\Gamma_\delta(\beta_l)$ that in fact
$$
\omega_j(s\beta_l)=s\omega_j(\beta_l) \text{ for all }s\in\mathbb{R}.
$$
We compute
\begin{align}\label{t19}
\begin{split}
&\tilde E_{i,j}(\eps,k_p,k_{p-1})=\omega_i\left((t+r_p)\frac{\beta_l}{\eps}\right)-\frac{r_p\omega_N(\beta_l)}{\eps}-\omega_j\left(t\frac{\beta_l}{\eps}\right)=\\
&\qquad \frac{t}{\eps}(\omega_i(\beta_l)-\omega_j(\beta))+\frac{r_p\omega_i(\beta_l)-r_p\omega_N(\beta_l)}{\eps}.
\end{split}
\end{align}
Since $\omega_i(\beta_l)-\omega_j(\beta_l)\neq 0$, we see that there exists a $t_p\in \mathbb{R}$ such that the right side of \eqref{t19} vanishes at $t=t_p$.  Namely, 
\begin{align}\label{t19a}
t_p=r_p\; \Omega_{i,j}, \text{ where }\Omega_{i,j}:=\frac{\;\omega_i(\beta_l)-\omega_N(\beta_l)}{\omega_j(\beta_l)-\omega_i(\beta_l)}.     
\end{align}

Writing $t=(t-t_p)+t_p$ we obtain
\begin{align}\label{t20}
\tilde E_{i,j}(\eps,k_p,k_{p-1})=\frac{(t-t_p)}{\eps}(\omega_i(\beta_l)-\omega_j(\beta_l)):=\frac{(t-t_p)}{\eps}C(\beta_l).
\end{align}

\textbf{4. }Fix $\lambda>0$ and suppose $|t-t_p|\geq \lambda$.
\; Then \eqref{t20}  implies
\begin{align}\label{t21}
|\tilde E_{i,j}(\eps,k_p,k_{p-1})|\geq \lambda \frac{C(\beta_l)}{\eps}.
\end{align}
If the  condition $|t|\geq 2|t_p|+2$ holds, then we have
\begin{align}\label{t21a}
|\tilde E_{i,j}(\eps,k_p,k_{p-1})|\geq  \frac{|t|}{2}\frac{C(\beta_l)}{\eps}.
\end{align}

\textbf{5. }Using \eqref{t18a} and \eqref{t11ya} we obtain
\begin{align}\label{t21b}
\begin{split}
&|E_{i,j}(\eps,k_p,k_{p-1})-\tilde E_{i,j}(\eps,k_p,k_{p-1})|\leq |\omega_i(X_{k_p})-\omega_i(\tilde X_{k_p})|+|\omega_j(X_{k_{p-1}})-\omega_j(\tilde X_{k_{p-1}})|\\
&\qquad\qquad\qquad \lesssim |X_{k_p}-\tilde X_{k_p}|+|X_{k_{p-1}}-\tilde X_{k_{p-1}}|,
\end{split}
\end{align}
and thus
\begin{align}\label{t22}
|E_{i,j}(\eps,k_p,k_{p-1})|\geq |\tilde E_{i,j}(\eps,k_p,k_{p-1})|-C(|X_{k_p}-\tilde X_{k_p}|+|X_{k_{p-1}}-\tilde X_{k_{p-1}}|).
\end{align}
 Since $X_{k_p}, X_{k_{p-1}}\in \Gamma_{\frac{\delta}{r_p}}$ we have
 \begin{align}\label{t22a}
 |X_{k_p}-\tilde X_{k_p}|\lesssim \frac{\delta}{r_p}|\tilde X_{k_p}|    \text{ and } |X_{k_{p-1}}-\tilde X_{k_{p-1}}|\lesssim \frac{\delta}{r_p}|\tilde X_{k_{p-1}}|.
\end{align}
If $|\tilde X_{k_{p-1}}|\leq \frac{2|t_p|+2}{\eps}$, we deduce from \eqref{t21}, \eqref{t22},\eqref{t22a} that for $\delta=\delta(\lambda)$  small and $|t-t_p|\geq \lambda$:
\begin{align}\label{t23}
|E_{i,j}(\eps,k_p,k_{p-1})|\geq \frac{\lambda}{2}\frac{C(\beta_l)}{\eps}\geq C_3(\lambda) \frac{|X_{k_p}|}{r_p}.
\end{align}
If $|\tilde X_{k_{p-1}}|\geq \frac{2|t_p|+2}{\eps}$, then $|t|\geq 2|t_p|+2$, and we deduce  from \eqref{t21a}, \eqref{t22}, \eqref{t22a} that for $\delta$  small:
\begin{align}\label{t24}
|E_{i,j}(\eps,k_p,k_{p-1})|\geq \frac{|t|C(\beta_l)}{3\eps}\geq C_3(\lambda)  |X_{k_{p-1}}|, 
\end{align}
after reducing $C_3(\lambda)$ if necessary.

\textbf{6. }Finally, 
for any given $(\zeta,\eps)$ and $\delta$, $C_3$ as in step \textbf{5}, we define $M_{i,j}(\zeta,\eps,\delta;C_3)$  by
\begin{align}\label{t25}
M_{i,j}(\zeta,\eps,\delta;C_3):=\{p\in\mathbb{Z}:X_{k_p}\in \Gamma_{\frac{\delta}{r_p}}(\beta_l), X_{k_{p-1}}\in \Gamma_{\frac{\delta}{r_p}}(\beta_l), \text{ and both \eqref{t23}, \eqref{t24} fail.} \}
\end{align}
The above estimates show that 
\begin{align}\label{t25a}
M_{i,j}(\zeta,\eps,\delta;C_3)\subset \cM_{i,j}(\zeta,\eps,\delta):=\{p\in\mathbb{Z}:X_{k_p}\in \Gamma_{\frac{\delta}{r_p}}(\beta_l), X_{k_{p-1}}\in \Gamma_{\frac{\delta}{r_p}}(\beta_l), |t-t_p|<\lambda\}.
\end{align}
 Since $|t_p|\lesssim r_p$, the estimate \eqref{t18z} now 
follows directly from the definition of $\cM_{i,j}$ and  \eqref{t18y}.

\end{proof}


The next Proposition is used in the proof of part (b) of Theorem \ref{tt26}.

\begin{prop}\label{t18aa}
Suppose  $i\in\mathcal{O}$, $j\in\mathcal{I}\setminus \{N\}$ and assume that 
\begin{align}\label{tt1a}
\Omega_{i,j}:= \frac{\omega_i(\beta_l)-\omega_N(\beta_l)}{\omega_j(\beta_l)-\omega_i(\beta_l)}\in (-1,0).
\end{align}
There exist positive constants $\eps_0$, $\delta_0$ and positive constants $C_3$, $C_4$ \emph{independent of} $(\zeta,\eps,p)\in \Xi\times (0,\eps_0]\times  \ZZ$ such that the following situation holds:

Let $(k_p)$ be an admissible sequence.   For any given $(\zeta,\eps,\delta)\in\Xi\times(0,\eps_0]\times (0,\delta_0]$  there exists at most one exceptional element  $m=m^{i,j}(\zeta,\eps,\delta)\in\mathbb{Z}$,  such that if $p\neq m$ and if 
 $X_{k_p}\in\Gamma_{\frac{\delta}{|r_p|}}(\beta_l)$, $X_{k_{p-1}}\in\Gamma_\frac{\delta}{|r_p|}(\beta_l)$,  then we have either 
\begin{align}\label{t18wa}
|E_{i,j}(\eps,k_p,k_{p-1}|\geq C_3\frac{|X_{k_p}|}{|r_p|} \text{ or }|E_{i,j}(\eps,k_p,k_{p-1};\beta)|\geq C_3|X_{k_{p-1}}|.
\end{align}
Moreover, the exceptional value $m$ satisfies
\begin{align}\label{t18za}
|X_{k_{m}}|\leq \frac{C_4|r_{{m}}|}{\eps}.
\end{align}
\end{prop}


\begin{proof}  We carry out the proof for a strictly increasing sequence $(k_p)$; the decreasing case is similar. 

\textbf{1. } Consider  $\cM_{i,j}(\zeta,\eps,\delta)$ as in step \textbf{6} of the previous proof and $\delta(\lambda)$, $C_3(\lambda)$ as in step \textbf{5} there.  To prove the proposition, it is enough to make a choice of $\lambda>0$ such that for any fixed $(\zeta,\eps)$ the set
\begin{align}\label{t25b}
\cM_{i,j}(\zeta,\eps,\delta):=\{p\in\mathbb{Z}:X_{k_p}\in \Gamma_{\frac{\delta}{r_p}}(\beta_l), X_{k_{p-1}}\in \Gamma_{\frac{\delta}{r_p}}(\beta_l), |t-t_p|<\lambda\}
\end{align}
has cardinality $|\cM_{i,j}(\zeta,\eps,\delta)|\leq 1$.\footnote{Recall that for a given $(\zeta,\eps,p)$, $t=t(p)$ was defined in \eqref{t18y} by $\tilde X_{k_{p-1}}=t\frac{\beta_l}{\eps}$.  Also, from \eqref{t19a} we have $t_p:=r_p\Omega_{i,j}.$}  For $\Omega:=\Omega_{i,j}$ as in \eqref{tt1a} we take
\begin{align}\label{tt2}
\lambda=\lambda_{i,j}=\frac{1}{3}\min\{|\Omega-(-1)|, |\Omega-0|\}>0.
\end{align}

Suppose $n>m$ are elements of $\cM_{i,j}(\zeta,\eps,\delta)$.   We have 
\begin{align}
\begin{split}
&X_{k_{n-1}}=\zeta+k_{n-1}\frac{\beta_l}{\eps}=\zeta+(k_{n}-r_n)\frac{\beta_l}{\eps},  \\
&\tilde X_{k_{n-1}}=s\frac{\beta_l}{\eps}+(k_n-r_n)\frac{\beta_l}{\eps},
\end{split}
\end{align}
where $s\frac{\beta_l}{\eps}$ is the  orthogonal projection of $\zeta$ on $\beta_l$ for some $s\in\RR$.
\;Thus, we can write the ``$t-$values" determined by  $\tilde X_{k_{n-1}}$, $\tilde X_{k_{m-1}}$ as 
\begin{align}
t(n)=s+(k_n-r_n),\;\; t(m)=s+(k_m-r_m).
\end{align}
The assumption that $n,m\in \cM_{i,j}$ means
\begin{align}\label{tt3}
\begin{split}
&(a) |s+(k_n-r_n)-r_n\Omega|<\lambda \text{ and }|s+(k_m-r_m)-r_m\Omega|<\lambda, \text{ so }\\
&(b) |(k_n-k_m)-(r_n-r_m)-(r_n-r_m)\Omega|<2\lambda.
\end{split}
\end{align}

\textbf{2. }
The inequality \eqref{tt3}(b) shows that $r_n \neq r_m$, since otherwise $1<\frac{2\lambda}{k_n-k_m}$, which is untrue.  Thus, \eqref{tt3}(b) implies
\begin{align}\label{tt4}
\left|\left(\frac{k_n-k_m}{r_n-r_m}-1\right)-\Omega\right|< \frac{2\lambda}{|r_n-r_m|}\leq 2\lambda.
\end{align}
If $r_n>r_m$, the quantity $\left(\frac{k_n-k_m}{r_n-r_m}-1\right)>0$ since $\frac{k_n-k_m}{r_n}\geq 1$, and so \eqref{tt4} is impossible since $|\Omega|\geq 3\lambda$. 
If $r_n<r_m$, the quantity $\left(\frac{k_n-k_m}{r_n-r_m}-1\right)<-1$, and so \eqref{tt4} contradicts the fact that $|-1-\Omega|\geq 3\lambda$.
Thus, $|\cM_{i,j}(\zeta,\eps,\delta)|\leq 1$.

\textbf{3. }
If $\cM_{i,j}(\zeta,\eps,\delta)$  is nonempty, we denote its single element by $m=m^{i,j}(\zeta,\eps)$.    Since $|t_p|\lesssim r_p$, the estimate \eqref{t18za}
follows directly from \eqref{t25a} and \eqref{t18y}.


\end{proof}




Observe that $\Omega_{i,N}=-1$.  This case, which is always a good one, is treated in the next proposition.     This proposition is useful for counting large amplification factors (Proposition \ref{large}) and is also needed in the proof of Theorem \ref{tvv29}.

\begin{prop}\label{goodcase}
Suppose  $i\in\cO$, $j=N\in\cI$, so that $\Omega_{i,N}=-1$.   There exist positive constants $\eps_0$, $\delta_0$ and a positive constant $C_3$ independent of $(\zeta,\eps,p)\in \Xi\times (0,\eps_0]\times  \ZZ$ such that  for any given $(\zeta,\eps,\delta)\in\Xi\times(0,\eps_0]\times (0,\delta_0]$,  if 
 $X_{k}\in\Gamma_{\frac{\delta}{|r|}}(\beta_l)$ and  $X_{k-r}\in\Gamma_\frac{\delta}{|r|}(\beta_l)$,  then   
\begin{align}\label{tt5}
|E_{i,N}(\eps,k,k-r)|\geq C_3\frac{|X_{k}|}{|r|} \text{ or }|E_{i,N}(\eps,k,k-r)|\geq C_3|X_{k-r}|.
\end{align}

\end{prop}

\begin{proof}
\textbf{1. }Using notation similar to that in the proof of Proposition \ref{t18}, we define $t$ by $\tilde X_{k-r}=t\frac{\beta_l}{\eps}$. 
Fix $\lambda>0$ and suppose $|t-r\Omega_{i,N}|\geq \lambda$.  Steps \textbf{3-5}  of the proof of Proposition \ref{t18} imply \eqref{tt5} in this case.

\textbf{2. }Now suppose $|t-r\Omega_{i,N}|< \lambda$, that is,  $|t+r|<\lambda$.  Since $\tilde X_{k}=(t+r)\frac{\beta_l}{\eps}$, we obtain
\begin{align}\label{tt6}
\begin{split}
&|\tilde E_{i,N}(\eps,k,k-r)|=\left|\omega_i\left((t+r)\frac{\beta_l}{\eps}\right)-r\omega_N\left(\frac{\beta_l}{\eps}\right)-\omega_N\left( t\frac{\beta_l}{\eps}\right)\right|=\\
&\qquad \left|\omega_i\left((t+r)\frac{\beta_l}{\eps}\right)-\omega_N\left( (t+r)\frac{\beta_l}{\eps}\right)\right|\sim \left|(t+r)\frac{\beta_l}{\eps}\right|\sim |\tilde X_{k}|.
\end{split}
\end{align}
Since $|X_{k}|\sim |\tilde X_{k}|$ and 
\begin{align}
\begin{split}
&|\omega_i(X_{k})-\omega_i(\tilde X_{k})|\lesssim \frac{\delta}{|r|}|\tilde X_{k}|,\\
&\left|\omega_N\left(X_{k}-r\frac{\beta_l}{\eps}\right)-\omega_N\left(\tilde X_{k}-r\frac{\beta_l}{\eps}\right)\right|\lesssim |X_k-\tilde X_k|\lesssim \frac{\delta}{|r|}|\tilde X_{k}|,
\end{split}
\end{align}
the first alternative in \eqref{tt5} follows from \eqref{tt6} for $\delta$ small.

\end{proof}

\section{Cascade estimates}\label{2s}
  \qquad Our main concern in this section is to finish the proofs of Theorems \ref{tt26} and \ref{tvv29}.   For this we must show how the iteration estimate of Proposition \ref{tt30} can be used to prove useful energy estimates for the singular system \eqref{i6}.  The first task is to develop an efficient procedure for managing the proliferation of terms that arise when the iteration estimate is iterated in the one-sided case.

  
  \subsection{One-sided cascade estimates and proof of Theorem \ref{tt26}}

\quad As in the Introduction we first consider the transformed singular problem \eqref{i9} in the case where $F=0$ and $G_k=0$ for $k<1$.

A more efficient way to obtain the essential information in \eqref{b001}, \eqref{b002} or \eqref{b01} is to consider the following ``$\CG_j$-cascades" corresponding to the cascades \eqref{b002}, \eqref{b01}:
\begin{align}\label{b02}
\begin{split}
&(a) [(\cG_3)]\to [(\cG_2,\cG_1,\cG_1)]\to [(\cG_1)]\\
&(b) [(\cG_5)]\to [(\cG_4,\cG_3,\cG_2,\cG_1,\cG_3,\cG_2,\cG_1,\cG_2,\cG_1,\cG_1)]\to \\
&\quad [(\cG_3,\cG_2,\cG_1,\cG_2,\cG_1,\cG_1),(\cG_2,\cG_1,\cG_1), (\cG_1),(\cG_2,\cG_1,\cG_1),(\cG_1),(\cG_1)]\to\\
&\qquad [(\cG_2,\cG_1,\cG_1),(\cG_1),(\cG_1),(\cG_1),(\cG_1)]\to [(\cG_1)]
\end{split}
\end{align}
The principle is slightly different here; for example, the $j-$th stage of \eqref{b02}(b)  records only the \emph{new} $\cG_p$ that appear in the $(j+1)$-st stage of \eqref{b01}.  Consequently, the thirty-four $\cG_p$ terms that appear in \eqref{b02}(b) are the same as the thirty-four $\cG_p$ terms that appear in the last stage of \eqref{b01}.     

The ``rule" for constructing a $\cG_j$-cascade is that a term $\cG_p$ in a given stage should give rise to the terms 
\begin{align}
\cG_p\to (\cG_{p-1},\cG_{p-2},\dots,\cG_1,\cG_{p-2},\dots,\cG_1,\dots,\cG_{3},\cG_{2},\cG_{1},\cG_{2},\cG_{1},\cG_1).
\end{align}
in the next stage.  Terms with indices $j\leq 0$ are omitted, since we are assuming for now that $G_j=0$ for $j\leq 0$.   Thus, for example, 
$\cG_5\to(\cG_4,\cG_3,\cG_2,\cG_1,\cG_3,\cG_2,\cG_1,\cG_2,\cG_1,\cG_1)$, while $\cG_3\to (\cG_2,\cG_1,\cG_1)$ and  $\cG_2\to(\cG_1)$.

The parentheses in \eqref{b02} allow us to track the ``genealogy" of each $\cG_j$.   For example, we regard each of the $\cG_j$ in  the second group $(\cG_2,\cG_1,\cG_1)$ that appears in the third stage of \eqref{b02}(b) as a ``descendant" of the second $\cG_3$ that appears  in the second stage.  Similarly, the $\cG_1$ that appears in the final stage of \eqref{b02}(a) is a descendant of the $\cG_2$ in the second stage.    The number of arrows that precede the stage in which a given $\cG_j$ lies tells us the number of factors of the form $\frac{C}{\gamma}\ar\DD(\eps,p,p-r)$  that should multiply that  $\CalG_j$ in the final estimate; the particular choices of $p$ and $r$ appearing in those factors are determined by the 
genealogy of the given $\cG_j$. 
For example, the term $[(\cG_1)]$ in the final stage of \eqref{b02}(a) should have two such factors attached, and indeed this term corresponds to the first term inside the brackets in the last line of \eqref{b001}.   Keeping in mind the iteration estimate \eqref{i13},  one can easily reconstruct the complete estimate of $V_3$ starting just  from \eqref{b02}(a), and the same applies to any $V_k$.  For example, the last $(\cG_1)$ term appearing in the fourth stage of \eqref{b02}(b) corresponds to the following term on the right in the final estimate of $V_5$:
\begin{align}\label{b02a}
\left|\left(\frac{C}{\gamma}\right)^3\alpha_2\DD(\eps,5,3)\alpha_1\DD(\eps,3,2)\alpha_1\DD(\eps,2,1)\cG_1\right|_{L^2(\sigma,\eta)}.
\end{align}

\begin{rem}\label{b03}

1) In the estimate of $V_k$, terms appear that involve a product of up to $k-1$ factors of the form $\frac{C}{\gamma}\ar\DD(\eps,p,p-r)$, where $2\leq p\leq k$.  A term involving three such factors in the estimate of $V_5$ is given by \eqref{b02a}.    
Observe that the  product of $\DD(\eps,p,p-r)$ factors in \eqref{b02a} is a special case of a finite product of the form
\begin{align}\label{b04a}
\DD(\eps,k_j,k_{j-1})\DD(\eps,k_{j-1},k_{j-2})\DD(\eps,k_{j-2},k_{j-3})\cdots
\end{align}
where the $k_p$ that appear are elements of an admissible sequence $(k_j)_{j\in\ZZ}$.  
The results proved in  section \ref{ptools} will allow us to control these products by using the fact that 
they are
\emph{always} either of the form \eqref{b04a}, or can be embedded in products of that form.  For example, the product $\DD(\eps,8,7)\DD(\eps,7,5)\DD(\eps,3,2)$ does not have this form, but can be embedded in $$\DD(\eps,8,7)\DD(\eps,7,5)\DD(\eps,5,3)\DD(\eps,3,2),$$ which has the right form.  The large factors  are counted in Proposition \ref{large} below.

2)  We will see  that it is not necessary to keep track of the exact indices $(p,p-r)$ that appear in the individual factors, but only to keep track of the ``step sizes", where $r$ is the step size of the pair $(p,p-r)$, and of the various factors $\at$ contributed by the inner sum in the iteration estimate (see \eqref{b1} and \eqref{b05d}).

\end{rem}

\begin{prop}[Counting the large amplification factors]\label{large}
Let $(k_p)_{p\in\ZZ}$ be an admissible sequence and consider any finite product of the form \eqref{b04a}, where the factors $\DD(\eps,k_p,k_{p-1})(\zeta)$ are defined in Definition \ref{t29}. 
Then for any given $(\eps,\zeta)\in (0,\eps_0]\times \Xi$,   at most $$\EE=(|\Upsilon^+_0|-1)+\sum_{i\in\cO,j\in\cI\setminus \{N\}}\MM_{i,j}$$ of the factors in that product are ``large", that is, equal to $\frac{C_5r^2}{\eps\gamma}$. Here the $\MM_{i,j}$ are as in Theorem \ref{tt26}(a).    As $(\eps,\zeta)$ varies,  the particular indices $p$ for which  $\DD(\eps,k_p,k_{p-1})(\zeta)$ is large can vary.  

\end{prop}

\begin{proof}
\textbf{1. } We will refer to cases $(I)-(III)$ as in \eqref{cases} and Definition \ref{t28}.   For each $\beta\neq \beta_l$ the microlocal factor $D(\eps,k_p,k_{p-1};\beta)(\zeta)$ is large only in case $(II)$ when \eqref{t11za} fails.  Proposition \ref{keyt} shows that this can happen for at most one choice of  $p\in\ZZ$.  Thus, at most $|\Upsilon^+_0|-1$ of the factors $\DD(\eps,k_p,k_{p-1})$ can be large due to largeness of $D(\eps,k_p,k_{p-1};\beta)(\zeta)$ for some $\beta\neq \beta_l$.

\textbf{2. }A factor $D(\eps,k_p,k_{p-1};\beta_l)(\zeta)$ can be large only if 
  case $(Ib)$ holds for some $(i,j)\in\cO\times \cI$.  Step \textbf{6} of the proof of Proposition \ref{t18} shows that for a given  pair $(i,j)\in\cO\times (\cI\setminus \{N\})$, there can be at most $\MM_{i,j}$ indices $p$ for which this happens.    In addition, Proposition \ref{goodcase}  (or Remark \ref{gc}) shows that case $(Ib)$ never holds for pairs $(i,N)$, $i\in\cO$.  Thus, at most $\sum_{i\in\cO,j\in\cI\setminus \{N\}}\MM_{i,j}$  distinct factors $\DD(\eps,k_p,k_{p-1})(\zeta)$  can be large due to largeness of $D(\eps,k_p,k_{p-1};\beta_l)(\zeta)$ for this reason.

\textbf{3. }Thus,  at most $(|\Upsilon^+_0|-1)+\sum_{i\in\cO,j\in\cI\setminus\{N\}}\MM_{i,j}=\EE$ factors $\DD(\eps,k_p,k_{p-1})(\zeta)$ in the given product can be large.


\end{proof}

\subsubsection{Schematic representation of the $V_k$ estimates}

\qquad Observe that by \eqref{b00} and Definition \ref{t29} we have for each $(\zeta,\eps)$:\footnote{In \eqref{b05a} we are asserting that there is a constant $C$ such that either the first condition holds, \emph{or} the first condition fails and the second condition holds.}
\begin{align}\label{b05a}
\frac{C}{\gamma}|\ar\DD(\eps,p,p-r)(\zeta)|\lesssim  \frac{C}{\gamma}\frac{1}{r^{M}}\text{ or }\frac{C}{\gamma}|\ar\DD(\eps,p,p-r)(\zeta)|\lesssim \frac{C}{\gamma} \frac{1}{\eps\gamma} \frac{1}{r^{M}}.
\end{align}
Let us define $\CalD_r(\zeta)$ to be the function of $\zeta$:
\begin{align}\label{b05}
\CalD_r(\zeta)=\begin{cases}\frac{C}{\gamma}\frac{1}{r^{M}}, \text{ if }\frac{C}{\gamma}|\ar\DD(\eps,p,p-r)(\zeta)|\lesssim  \frac{C}{\gamma}\frac{1}{r^{M}}\\  \frac{C}{\gamma} \frac{1}{\eps\gamma}\frac{1}{r^{M}}, \text{ if not}\end{cases}.
\end{align}

 We claim that we can represent the essential aspects of the estimate \eqref{b001} of $V_3$ schematically by $V_3\leq \cG^T_3$ and, more generally, represent the estimate of $V_k$ by 
\begin{align}\label{b0}
V_k\leq \CG_k^T,
\end{align}
where the $\cG_k^T$ are defined recursively by \footnote{The superscript $T$ in $\cG^T_k$ is meant to indicate the ``tree-like object" generated by $\cG_k$.}
\begin{align}\label{b1}
\begin{split}
&\CG_1^T=\CG_1\\
&\CG_2^T=\CG_2+\CalD_1\CG_1\\
&\CG_3^T=\CG_3+\CalD_1\CG_2^T+(\CalD_1+\CalD_2)\CG_1\\
&\CG_4^T=\CG_4+\CalD_1\CG_3^T+(\CalD_1+\CalD_2)\CG^T_2+(\alpha_2\CalD_1+\CalD_2+\cD_3)\CG_1\\
&\CG_5^T=\CG_5+\CalD_1\CG^T_4+(\CalD_1+\CalD_2)\CG^T_3+(\alpha_2\CalD_1+\CalD_2+\cD_3)\CG^T_2+(\alpha_3
\cD_1+\alpha_2\cD_2+\CalD_3+\cD_4)\CG_1\\
&\CG_6^T=\CG_6+\CalD_1\CG^T_5+(\CalD_1+\CalD_2)\CG^T_4+(\alpha_2\CalD_1+\CalD_2+\cD_3)\CG^T_3+\\
&\qquad (\alpha_3
\cD_1+\alpha_2\cD_2+\CalD_3+\cD_4)\CG^T_2+(\alpha_4
\cD_1+\alpha_3\cD_2+\alpha_2\CalD_3+\cD_4+\cD_5)\CG_1\\
&\dots
\end{split}
\end{align}
The factors $\alpha_t$ that appear in \eqref{b1} come from the inner sum in the interaction estimate \eqref{t30c}.
Writing out $\CalG^T_3$ we obtain:
\begin{align}\label{bb1}
\cG^T_3=\cG_3+\cD_1\cG_2+\cD^2_1\cG_1+\cD_1\cG_1+\cD_2\cG_1.
\end{align}
The term $\cD^2_1\cG_1$, for example, ``represents" the term $\left(\frac{C}{\gamma}\right)^2|\alpha_1\DD(\eps,3,2)\alpha_1\DD(\eps,2,1)\CalG_1|$ in \eqref{b001}, while $\cD_2\cG_1$ represents the term $\frac{C}{\gamma}|\alpha_2\DD(\eps,3,1)\CalG_1|$.   We will refer to \eqref{b001} as ``the proper estimate of $V_3$" and to $V_3\leq \cG^T_3$ as ``the schematic estimate of $V_3$".   For every $k$ there is a one-to-one correspondence between the terms of the proper estimate of $V_k$ and those of the schematic estimate of $V_k$ (after the $\cD_r$ have been distributed as in \eqref{bb1}).  We explain below how to transform schematic estimates, which can be stated with great concision, into (proper) estimates.

 Consider a term on the right in the (proper) estimate of $V_k$, call it $T$,  that  consists of exactly $k_0$ factors of the form $\frac{C}{\gamma}\ar\at \DD(\eps,p,p-r)$ multiplying $\cG_l$, for some $k_0\leq k-1$.   
 Proposition \ref{large} shows that for any fixed $(\eps,\zeta)$  at most $\EE$ of those factors fail to satisfy the first possibility in \eqref{b05a}.  Thus, 
 \begin{align}\label{b05b}
 T\leq \left(\frac{1}{\eps\gamma}\right)^\EE \left(\frac{C}{\gamma}\right)^{k_0}\frac{|\alpha_{t_1}|}{r_1^M}\frac{|\alpha_{t_2}|}{r_2^M}\cdots\frac{|\alpha_{t_{k_0}}|}{r_{k_0}^M}|\cG_l|,
 \end{align}
where $r_i$ is the step size of the $i-$th factor.   The term $T$  would be represented in the schematic estimate of $V_k$ by 
\begin{align}\label{b05cc}
T=\alpha_{t_1}\cD_{r_1}\alpha_{t_2}\cD_{r_2}\cdots \alpha_{t_{k_0}}\cD_{r_{k_0}}\cG_l,
\end{align}
and we may represent \eqref{b05b} by\footnote{Unlike \eqref{b05cc} or $V_k\leq \cG^T_k$, the schematic estimate \eqref{b05c} is very close to a proper estimate.   To obtain a proper estimate we just replace $\cG_l$ on the right by $|\cG_l|_{L^2(\sigma,\eta)}$ and replace each $\alpha_{t_i}$ by $|\alpha_{t_i}|$.}
\begin{align}\label{b05c}
T\leq \left(\frac{1}{\eps\gamma}\right)^\EE  \alpha_{t_1}\cE_{r_1}\alpha_{t_2}\cE_{r_2}\cdots\alpha_{t_{k_0}}\cE_{r_{k_0}}\CG_l, \text{ where }\cE_r:=\frac{C}{\gamma}\frac{1}{r^M}.
\end{align}
Moreover, \eqref{b05c} implies the schematic estimate
\begin{align}\label{b05cd}
\cG^T_k\leq \left(\frac{1}{\eps\gamma}\right)^\EE \cH^T_k,
\end{align}
where $\cH^T_k$ is defined inductively by
\begin{align}\label{b05d}
\begin{split}
&\CH_1^T=\CG_1\\
&\CH_2^T=\CG_2+\CalE_1\CG_1\\
&\CH_3^T=\CG_3+\CalE_1\CH_2^T+(\CalE_1+\CalE_2)\CG_1\\
&\CH_4^T=\CG_4+\CalE_1\CH_3^T+(\CalE_1+\CalE_2)\CH^T_2+(\alpha_2\CalE_1+\CalE_2+\cE_3)\CG_1\\
&\CH_5^T=\CG_5+\CalE_1\CH^T_4+(\CalE_1+\CalE_2)\CH^T_3+(\alpha_2\CalE_1+\CalE_2+\cE_3)\CH^T_2+(\alpha_3
\cE_1+\alpha_2\cE_2+\CalE_3+\cE_4)\CG_1\\
&\CH_6^T=\CG_6+\CalE_1\CH^T_5+(\CalE_1+\CalE_2)\CH^T_4+(\alpha_2\CalE_1+\CalE_2+\cE_3)\CH^T_3+\\
&\quad \qquad (\alpha_3
\cE_1+\alpha_2\cE_2+\CalE_3+\cE_4)\CH^T_2+(\alpha_4
\cE_1+\alpha_3\cE_2+\alpha_2\CalE_3+\cE_4+\cE_5)\CG_1\\
&\dots
\end{split}
\end{align}
For any $k$ we have
\begin{align}\label{b6ab}
\begin{split}
&\cH^T_k=\cG_k+\cE_1\cH^T_{k-1}+(\cE_1+\cE_2)\cH^T_{k-2}+(\alpha_2\cE_1+\cE_2+\cE_3)\cH^T_{k-3}+\\
&\qquad(\alpha_3\cE_1+\alpha_2\cE_2+\cE_3+\cE_4)\cH^T_{k-4}+\dots+
(\alpha_{k-2}\cE_1+\alpha_{k-3}\cE_2+\dots++\alpha_2\cE_{k-3}+\cE_{k-2}+\cE_{k-1})\cG_1.
\end{split}
\end{align}




\begin{rem}\label{b1a}

 For $j\leq k$ the coefficient of $\CG_j$ in $\CH^T_k$ equals the coefficient of $\CG_{j+1}$ in $\CH^T_{k+1}$.   To see this look, for example, at the coefficients of $\CG_1$ on the outermost diagonal of \eqref{b05d} ending say, at row 5.\footnote{These coefficients are $1,\CalE_1,\CalE_1+\CalE_2,\alpha_2\CalE_1+\CalE_2+\cE_3,\alpha_3\CalE_1+\alpha_2\cE_2+\cE_3+\cE_4.$}   These are the same as the coefficients of $\CG_2$ or $\CH_2^T$ on the first  subdiagonal starting at row 2 and ending at row 6, and these are the same as the coefficients of $\CG_3$ or $\CH_3^T$ on the second subdiagonal starting at row 3 and ending at row 7, etc..  All this remains true if the $\alpha_i$ in \eqref{b05d} are replaced by $|\alpha_i|$.

  Letting $g_{j,p}$ denote the coefficient of $\cG_j$ in $\CH^T_p$ \emph{after} the replacement of  all $\alpha_i$ by $|\alpha_i|$ , we have, consequently, the relation
 \begin{align}\label{b6aa}
 g_{j,p}=g_{1,p-(j-1)},
 \end{align}
 which in view of \eqref{b05cd} implies the (proper) estimate
\begin{align}\label{b6}
\|V_k\|\leq \left(\frac{1}{\eps\gamma}\right)^\EE\sum^k_{j=1}|\CG_j|_{L^2(\sigma,\eta)} g_{1,k-(j-1)}.
\end{align}

\end{rem}

%

\begin{prop}[Estimate of $|(\|V_k\|)|_{\ell^2}$]\label{bb6}
Consider the transformed singular problem \eqref{tz30} under the hypotheses of Theorem \ref{tt26}, but assume $F=0$ and $G_k=0$ for $k<1$.  
In particular,  we assume  the coefficients $\alpha_r$ in \eqref{tz30} satisfy $|\alpha_r|\lesssim |r|^{-(M+2)}$ for some $M\geq 2$.   Let $\EE$ be as in Proposition \ref{large}. 
There exist   positive constants $K$, $\gamma_0$ such that for $\gamma > \gamma_0$ we have
\begin{align}\label{bb7}
|(\|V_k\|)|_{\ell^2}\leq \frac{K}{(\eps\gamma)^{\EE}}|(|\CG_k|_{L^2(\sigma,\eta)})|_{\ell^2}.
\end{align}
We can take $\gamma_0=CC_MD_M$, where $C_M=2+\sum^\infty_{i=2}|\alpha_i|$,  $D_M=\sum^\infty_{r=1}\frac{1}{r^M}$, and $C$ is as in \eqref{b05}.
\end{prop}

\begin{proof}

\textbf{1. }From \eqref{b6} and 
Young's inequality we obtain
\begin{align}\label{b6a}
|(\|V_k\|)|_{\ell^2}\leq \left(\frac{1}{\eps\gamma}\right)^\EE |(|\CG_k|_{L^2(\sigma,\eta)})|_{\ell^2}|(g_{1,k})|_{\ell^1}.
\end{align}

\textbf{2. } From \eqref{b05d} for $k\geq 2$ we clearly have 
\begin{align}
\begin{split}
&g_{1,k}=\cE_1g_{1,k-1}+(\cE_1+\cE_2)g_{1,k-2}+(|\alpha_2|\cE_1+\cE_2+\cE_3)g_{1,k-3}+\\
&\quad\qquad (|\alpha_3|\cE_1+|\alpha_2|\cE_2+\cE_3+\cE_4)g_{1,k-4}+\dots+\\
&\qquad \qquad (|\alpha_{k-2}|\cE_1+|\alpha_{k-3}|\cE_2+\dots+|\alpha_2|\cE_{k-3}+\cE_{k-2}+\cE_{k-1})g_{1,1}.
\end{split}
\end{align}

\textbf{3. } To sum the $g_{1,k}$, we  write:
\begin{align}\label{b6b}
\begin{split}
&g_{1,1}=1\\
&g_{1,2}=0\quad +\quad\cE_1g_{1,1}\\
&g_{1,3}=0\quad+\quad(\cE_1+\cE_2)g_{1,1}\;\;\qquad\;+\qquad\;\; \;\cE_1g_{1,2}\\
&g_{1,4}=0+(|\alpha_2|\cE_1+\cE_2+\cE_3)g_{1,1}\qquad+\qquad (\cE_1+\cE_2)g_{1,2}\qquad+\quad \cE_1g_{1,3}\\
&g_{1,5}=0+(|\alpha_3|\cE_1+|\alpha_2|\cE_2+\cE_3+\cE_4)g_{1,1}+(|\alpha_2|\cE_1+\cE_2+\cE_3)g_{1,2}+(\cE_1+\cE_2)g_{1,3}+\cE_1g_{1,4}\\
&\cdots
\end{split}
\end{align}
Letting
\begin{align}\label{b6c}
\begin{split}
&\EE_M:=\cE_1+(\cE_1+\cE_2)+(|\alpha_2|\cE_1+\cE_2+\cE_3)+(|\alpha_3|\cE_1+|\alpha_2|\cE_2+\cE_3+\cE_4)+\\
&\qquad \qquad(|\alpha_4|\cE_1+|\alpha_3|\cE_2+|\alpha_2|\cE_3+\cE_4+\cE_5)+\dots
\end{split}
\end{align}
and summing  \eqref{b6b} ``by columns", we obtain
\begin{align}\label{b8}
S:=\sum^\infty_{k=1}g_{1,k}=1+g_{1,1}\EE_M+g_{1,2}\EE_M+g_{1,3}\EE_M+\dots=1+\EE_MS.
\end{align}
Resumming \eqref{b6c} we obtain
\begin{align}
\EE_M=\cE_1(1+1+|\alpha_2|+|\alpha_3|+|\alpha_4|+\dots)+\cE_2(1+1+|\alpha_2|+|\alpha_3|+|\alpha_4|+\dots)+\cdots.
\end{align}
Recalling that  $\cE_r=\frac{C}{\gamma}\frac{1}{r^M}$ and setting $C_M=2+\sum^\infty_{i=2}|\alpha_i|$,  $D_M=\sum^\infty_{r=1}\frac{1}{r^M}$, we have
\begin{align}
\EE_M=\frac{C}{\gamma}C_MD_M.
\end{align}
Thus, for $\gamma>CC_MD_M$ the sum $S$ is finite and 
\begin{align}\label{b9}
S=\frac{1}{1-\EE_M}=1+\EE_M+\EE_M^2+\dots.
\end{align}

\end{proof}

We can now complete the proof of Theorem \ref{tt26} in the case where $\cD(\theta_{in})=d(\theta_{in})M$; for the general case see Remark \ref{reduction}.

\begin{proof}[\textbf{Conclusion of the proof of Theorem \ref{tt26}}]
\textbf{ Part (a).}    In  the case where  $F_k$ and $G_k$ vanish for $k<1$ we have the iteration estimate \eqref{tt30c}.   Iterating this estimate leads to a proliferation of both $F_l$ and $G_l$ terms, but the new $F_l$ terms can be managed just like the $G_l$ terms in the proof of Proposition \ref{bb6}.   In place of \eqref{bb7} we have for some $K$ and $\gamma\geq \gamma_0$:
\begin{align}\label{b10}
|(\|V_k\|)|_{\ell^2}\leq \frac{K}{(\eps\gamma)^{\EE}}\left[|(|\cF_k|_{L^2(x_2,\sigma,\eta)})|_{\ell^2}+  |(|\CG_k|_{L^2(\sigma,\eta)})|_{\ell^2}\right],
\end{align}
where $\cF_k:=\frac{\widehat F_k|X_k|}{\gamma^2}$,  $\cG_k=\frac{\widehat G_k|X_k|}{\gamma^{3/2}}$, and $\EE$ is as in Proposition \ref{large}.

The estimate \eqref{b10} clearly holds with the same proof when the forcing terms are 
$$
F^{N^*}:=\sum_{k=N^*}^\infty F_k(t,x)e^{ik\theta} \text{ and  }
G^{N^*}:=\sum_{k=N^*}^\infty G_k(t,x_1)e^{ik\theta}\text{ for }N^*\in\mathbb{Z}. 
$$
Given general periodic functions $F(t,x,\theta)\in H^1(t,x,\theta)$ and $G(t,x_1,\theta)\in H^1(t,x_1,\theta)$,   we define $F^{N^*}$ and  $G^{N^*}$ by truncation, and obtain \eqref{tt28} in the limit as $N^*\to -\infty$.  Here we have used the fact that the constants $\gamma_0$, $K$ appearing in \eqref{b10} are independent of $\eps$ and $N^*$; recall Remark \ref{t50a}.    The estimate \eqref{tt28} then follows from \eqref{tu31}.  

\textbf{ Part (b).}  Let $\cM_{i,j}(\zeta,\eps,\delta)$ be as in \eqref{tt26b}, where the $\lambda_{i,j}$ are defined as in \eqref{tt26c}.  Proposition \ref{t18aa} and its proof show that 
\begin{align}\label{b11}
|\cM_{i,j}(\zeta,\eps,\delta)|\leq 1 \text{ for all }(i,j)\in \cO\times (\cI\setminus\{N\}).
\end{align}
Thus, we can choose the numbers $\MM_{i,j}=1$, and $\EE$ has the value \eqref{tt26d}. 

\textbf{ Part (c).} Let $\cP=\{r\in\NN: |\alpha_r|\neq 0\}$; so by assumption we have $|\cP|=P$.   For a given $(\eps,\zeta)$ consider any finite product of amplification factors 
\begin{align}\label{b12}
\DD(\eps,k_{p_1},k_{p_1-1})(\zeta)\cdot\DD(k_{p_2},k_{p_2-1})(\zeta)\cdot \dots \cdot \DD(\eps,k_{p_{N^*}}, k_{p_{N^*}-1})(\zeta),\;\;N^*\in\NN
\end{align}
 that might now appear in (a term on the right side of) the estimate of some $V_k$.   Observe that for every $l\in\{1,\dots,N^*\}$  
the step size\footnote{As always it can happen that $r_p=r_q$ for some $p\neq q$  in such a product.} 
\begin{align}
r_{p_l}=k_{p_l}-k_{p_l-1}\in\cP.
\end{align}
For $(i,j)\in\cO\times (\cI\setminus \{N\})$ define the set 
\begin{align}
\cM_{i,j}(\zeta,\eps,\delta;N^*)=\left\{l\in\{1,\dots,N^*\}:X_{k_{p_l}}\in\Gamma_{\frac{\delta}{r_{p_l}}}, X_{k_{p_l-1}}\in\Gamma_{\frac{\delta}{r_{p_l}}}, |t(p_l)-r_{p_l}\Omega_{i,j}|<\frac{1}{2}\right\}, 
\end{align}
where $t(p_l)$ is given by $\tilde X_{k_{p_l-1}}=t(p_l)\frac{\beta_l}{\eps}$ (recall \eqref{t18y}).    Since  we have $|t(p_l)-t(p_m)|\geq 1$ if $l\neq m$, it follows that 
\begin{align}\label{b13}
|\cM_{i,j}(\zeta,\eps,\delta;N^*)|\leq P.
\end{align}
One can now repeat the proof of Proposition \ref{large} with $\MM_{i,j}=P$ for $i\in \cO, j\in \cI\setminus \{N\}$ to obtain the upper bound
 $\EE=P|\cO|(|\cI|-1)+(|\Upsilon^+_0|-1)$ for the number of large factors in the product \eqref{b12}.\footnote{Here it does not help to apply Proposition \ref{large} as stated, since we must now take advantage of the fact that the number of distinct possible step sizes is $\leq P$.}

\end{proof}

\subsection{An effect of resonances}
\quad Consider again the system \eqref{s1} under assumptions \ref{assumption1}, \ref{assumption2}, \ref{assumption3}.  For $N$ and $\beta_l\in\Upsilon^0_+$ as in \eqref{s1a}, 
suppose $j,N\in \cI$ with $j\neq N$ and  $i\in \cO$.    We say that the associated characteristic phases $(\phi_j,\phi_N,\phi_i)$ exhibit a \emph{resonance} if there exist $p,q\in \ZZ\setminus 0$ such that\footnote{Here, recall $\phi_j(t,x)=\beta_l\cdot (t,x_1)+\omega_j(\beta_l)x_2$.}    
\begin{align}\label{tt7}
p\phi_j+q\phi_N=(p+q)\phi_i \Leftrightarrow p\omega_j(\beta_l)+q\omega_N(\beta_l)=(p+q)\omega_i(\beta_l)\Leftrightarrow \frac{p}{q}=\frac{\omega_i(\beta_l)-\omega_N(\beta_l)}{\omega_j(\beta_l)-\omega_i(\beta_l)}=\Omega_{i,j},
\end{align}
for $\Omega_{i,j}$ as in \eqref{b000}.

For a given $C_3>0$ and a given admissible sequence $(k_p)$, we recall the definition of the  bad set $M_{i,j}(\eps,\zeta,\delta;C_3)$ from Proposition \ref{t18}.  The integer $p\in M_{i,j}(\eps,\zeta,\delta;C_3)$ if and only if both $X_{k_p}$, $X_{k_{p-1}}$ lie in $\Gamma_{\frac{\delta}{|r_p|}}(\beta_l)$ and 
\begin{align}\label{tt7a}
\quad |E_{i,j}(\eps,k_p,k_{p-1})|\geq C_3\frac{|X_{k_p}|}{|r_p|} \text{ or }|E_{i,j}(\eps,k_p,k_{p-1})|\geq C_3|X_{k_{p-1}}|
\end{align}
\emph{fails} to hold.  Propositions \ref{t18} and \ref{t18aa} showed that when $\Omega_{i,j}\in (-1,0)$, one can choose $C_3$ so that $|M_{i,j}(\eps,\zeta,\delta;C_3)|\leq 1$, and this was an essential step in the proof of Theorem \ref{tt26}(b).    The next proposition shows that for certain  resonances, there exist admissible sequences $(k_p)$ and sets of $\zeta$ of large measure for which the set $M_{i,j}(\eps,\zeta,\delta;C_3)$ is infinite no matter how small $C_3>0$ is taken.

\begin{prop}\label{tt8}
a) For    any admissible sequence $(k_p)$, let $\tilde E_{i,j}(\eps,k_p,k_{p-1})$ be given by \eqref{t18a}.  
Let  $0<\alpha<1$,  and suppose there is a resonance such that  $\Omega_{i,j}$  as in \eqref{tt7} satisfies
\begin{align}\label{tt9}
\Omega_{i,j}=\frac{p}{q}\in \QQ\cap (-\infty,-1) \text{ or }\Omega_{i,j}\in \QQ\cap (0,\infty).
\end{align} 
Then one can construct admissible sequences $(k_p)$ such that for all $\zeta\in \Xi$ with $|\zeta|\leq \eps^{\alpha-1}$    we have
\begin{align}\label{tt10}
|\tilde E_{i,j}(\eps,k_p,k_{p-1})|\leq C(\beta_l)\eps^{\alpha-1}\text{ for infinitely many }p.
\end{align}

b) For each admissible sequence $(k_p)$ constructed in part (a) and for $\eps>0$ and $\delta>0$ small enough, there are subsets of $\Xi$ of large measure ( $|\zeta|\leq \eps^{\alpha-1}$) for which it is impossible to choose a constant $C_3>0$ independent of $(\eps,\zeta,p)$ such that $M_{i,j}(\eps,\zeta,\delta;C_3)$ is finite.

c)  If  \;$\Omega_{i,j}\in \QQ\cap (-1,0)$, then for any admissible sequence $(k_p)$ and for $C_3>0$ as chosen in Proposition \ref{t18aa}, we have 
$|M_{i,j}(\eps,\zeta,\delta;C_3)|\leq 1$ for $\eps>0$ and $\delta>0$ both small enough.

\end{prop}

\begin{proof}
\textbf{1.  Part a.}  Consider the case $\Omega_{i,j}=\frac{p}{q}\in \QQ\cap (0,\infty)$; the other case of \eqref{tt9} is treated similarly.  We may take $p,q\in\NN$.  Let $n\in\NN$ and set 
\begin{align}\label{tt11}
k(n)=n(p+q), r(n):=nq, \text{ so }k(n)-r(n)=np.
\end{align}
We have $X_{k(n)-r(n)}=\zeta+(k(n)-r(n))\frac{\beta_l}{\eps}$, so we may write $\tilde X_{k(n)-r(n)}=s\frac{\beta_l}{\eps}+(k(n)-r(n))\frac{\beta_l}{\eps}$, where $s\frac{\beta_l}{\eps}$ is the orthogonal projection of $\zeta$ on $\beta_l$.   Setting $t=s+k(n)-r(n)$, we obtain as in \eqref{t20}:
\begin{align}\label{tt12}
\begin{split}
&\tilde E_{i,j}(\eps,k(n),k(n)-r(n))(\zeta)=\frac{t-r(n)\Omega_{i,j}}{\eps}\;C(\beta_l)=\\
&\qquad \qquad \frac{s+k(n)-r(n)-r(n)\Omega_{i,j}}{\eps}\;C(\beta_l)=C(\beta_l) \frac{s}{\eps}.
\end{split}
\end{align}
So if $|\zeta|\leq \eps^{\alpha-1}$, it follows that $\frac{|s|}{\eps}\leq \eps^{\alpha-1}$, and thus
\begin{align}\label{tt13}
|\tilde E_{i,j}(\eps,k(n),k(n)-r(n))(\zeta)|\leq C(\beta_l)\eps^{\alpha-1}.
\end{align}
Now choose $0<n_1<n_2<n_3<\dots$ such that 
\begin{align}\label{tt14}
k(n_1)-r(n_1)<k(n_1)<k(n_2)-r(n_2)<k(n_2)<k(n_3)-r(n_3)<k(n_3)<\dots,
\end{align}
and relabel the respective elements of  \eqref{tt14} as $k_1<k_2<k_3<k_4<\dots$.    Then \eqref{tt13} implies that if $|\zeta|\leq \eps^{\alpha-1}$, we have
\begin{align}\label{tt14a}
|\tilde E_{i,j}(\eps,k_{2m},k_{2m-1})(\zeta)|\leq C(\beta_l) \eps^{\alpha-1}\text{ for all }m\in \NN.
\end{align}

\textbf{Part b.} Let $(k_p)$ be the admissible sequence constructed at the end of step \textbf{1}, so 
\begin{align}
k_{2m}=k(n_m)=n_m(p+q)\text{ and }r_{2m}=n_mq.
\end{align}  
Suppose $|\zeta|\leq \eps^{\alpha-1}$ and 
fix \emph{any} $C_3>0$ independent of $(\eps,\zeta,m)$.  We claim that for $m\geq 2$, the index $2m\in M_{i,j}(\eps,\zeta,\delta;C_3)$ for $\delta$  small enough and $0<\eps\leq \eps_0(\delta)$, provided $\eps_0(\delta)$ is small enough.  

Observe that for a given $\delta>0$ and $\eps_0(\delta)$ small enough, the vectors $X_{k_{2m}}=\zeta+n_m(p+q)\frac{\beta_l}{\eps}$ and $X_{k_{2m-1}}=\zeta+n_mp\frac{\beta_l}{\eps}$ both lie in $\Gamma_{\frac{\delta}{r_{2m}}}(\beta_l)=\Gamma_{\frac{\delta}{n_mq}}(\beta_l)$ for $0<\eps\leq\eps_0(\delta)$, and we have \footnote{We remark that $\eps_0(\delta)$ does not depend on $r_{2m}=n_mq$, because of the factor $n_m$ multiplying $\frac{\beta_l}{\eps}$ in the expressions for $X_{k_{2m}}$ and $X_{k_{2m-1}}$ .} 
\begin{align}\label{tt15}
\frac{p+q}{q\eps}\sim \frac{|X_{k_{2m}}|}{r_{2m}}\leq |X_{k_{2m-1}}|\sim \frac{n_mp}{\eps}\text{ for }m \geq 2.
\end{align}
Let $m\geq 2$ and suppose that 
\begin{align}\label{tt15a}
|E_{i,j}(\eps,k_{2m},k_{2m-1})|\geq C_3 \frac{|X_{k_{2m}}|}{r_{2m}}.
\end{align}
From \eqref{t22} and \eqref{t22a} we have
\begin{align}\label{tt16}
\begin{split}
&|E_{i,j}(\eps,k_{2m},k_{2m-1})-\tilde E_{i,j}(\eps,k_{2m},k_{2m-1})|\lesssim \frac{\delta}{r_{2m}}|X_{k_{2m}}|+\frac{\delta}{r_{2m}}|X_{k_{2m-1}}|\lesssim \frac{\delta}{r_{2m}}|X_{k_{2m}}|.
\end{split}
\end{align}
Then \eqref{tt15a} and \eqref{tt16} imply that for $\delta$ small enough and $0<\eps\leq \eps_0(\delta)$, 
\begin{align}
|\tilde E_{i,j}(\eps,k_{2m},k_{2m-1})|\geq \frac{C_3}{2} \frac{|X_{k_{2m}}|}{r_{2m}}\gtrsim \frac{1}{\eps},
\end{align}
but this contradicts \eqref{tt14a}, and so \eqref{tt15a} fails.  From \eqref{tt15} we see then that \eqref{tt7a} fails for $p=2m$, $m=2,3,\dots$, establishing the claim.

\textbf{3. }Part (c) follows immediately from Proposition \ref{t18aa}.

\end{proof}

  \subsection{The two-sided case and proof of Theorem \ref{tvv29}}
 
 \qquad We  first state a simple general result for problems with two-sided cascades under a boundedness assumption on the factors $\DD(\eps,p,p-r)$.  
Recall the iteration estimate \eqref{t30c} for the singular transformed problem \eqref{tz30}:
\begin{align}\label{cc8}
\|V_k\|\leq \frac{C}{\gamma}\sum_{r\in\ZZ\setminus 0}\sum_{t\in\ZZ}\|\ar\alpha_t\DD(\eps,k,k-r)V_{k-r-t}\|+C|\cF_k|_{L^2(x_2,\sigma,\eta)}+C|\CalG_k|_{L^2(\sigma,\eta)}, \;k\in\ZZ,
\end{align}
where we have set $\cF_k:=\frac{\widehat F_k|X_k|}{\gamma^2}$,  $\cG_k=\frac{\widehat G_k|X_k|}{\gamma^{3/2}}$.

 \begin{prop}\label{2sided}
Assume the coefficients $\alpha_r$ in \eqref{tz30} satisfy $|\alpha_r|\lesssim |r|^{-(M+1)}$ for some $M\geq 2$.
 Suppose there exist positive constants $\eps_0$ and $C$ such that for all $(\eps,\zeta,k,r)\in (0,\eps_0]\times \Xi\times \ZZ \times (\ZZ\setminus 0)$ we have
 \begin{align}\label{bdd}
 |\DD(\eps,k,k-r)(\zeta)|\leq C|r|.
 \end{align}
 Then there exist positive constants $K$, $\gamma_0$ such that for $\gamma>\gamma_0$ we have
 \begin{align}
|(\|V_k\|)|_{\ell^2}\leq K\left[|(|\cF_k|_{L^2(x_2,\sigma,\eta)})|_{\ell^2}+  |(|\CG_k|_{L^2(\sigma,\eta)})|_{\ell^2}\right].
\end{align}

 \end{prop}
 
\begin{proof}
Letting $\beta_r:=\alpha_r C|r|$, we have
\begin{align}
\begin{split}
&\sum_{r\in\ZZ\setminus 0}\sum_{t\in\ZZ}\|\ar\alpha_t\DD(\eps,k,k-r)V_{k-r-t}\|\leq \sum_{r\in\ZZ\setminus 0}\sum_{t\in\ZZ}\|\beta_r\alpha_tV_{k-r-t}\|=\\
&\qquad \qquad \sum_s\left(\sum_{r+t=s}|\beta_r\alpha_t|\right)\|V_{k-s}\|:= \sum_s\gamma_s\|V_{k-s}\|.
\end{split}
\end{align}
Applying Young's inequality gives
\begin{align}
 \left|\left(\sum_s\gamma_s\|V_{k-s}\|\right)\right|_{\ell^2(k)}\leq |(\|V_k\|)|_{\ell^2}|(\gamma_s)|_{\ell^1}. 
 \end{align}
Since $\gamma_s=\sum_r|\beta_r||\alpha_{s-r}|$, applying Young's inequality again we obtain
\begin{align}
|(\gamma_s)|_{\ell^1}\leq |(\beta_r)|_{\ell^1}|(\alpha_t)|_{\ell^1}:=K_1.
\end{align}
Thus, the $\ell^2$ norm of the right side of \eqref{cc8} is $\lesssim \frac{K_1}{\gamma} |(\|V_k\|)|_{\ell^2}+|(|\cF_k|_{L^2(x_2,\sigma,\eta)})|_{\ell^2}+|(|\cG_k|_{L^2(\sigma,\eta)})|_{\ell^2}$, and the result follows by taking $\gamma_0$ large enough.

\end{proof}

We can now finish the proof of Theorem \ref{tvv29} for the case where $\cD(\theta_{in})=d(\theta_{in})M$; for the general case see Remark \ref{reduction}. 

\begin{proof}[\textbf{Conclusion of the proof of Theorem \ref{tvv29}}]
Since $\Upsilon^+_0=\{\beta_l\}$ now, Proposition \ref{tt30} and Remark \ref{ty30} show that the estimate \eqref{cc8} holds with the definition of $\DD(\eps,k,k-r)$ modified as in that remark (put $|r|$ in place of $r^2$ in \eqref{ty29}). Using Definition \ref{t29}, we see that a factor $\DD(\eps,k,k-r)(\zeta)$ can take the value $\frac{C_5|r|}{\eps\gamma}$ only if $D(\eps,k,k-r;\beta_l)(\zeta)=\frac{C_5|r|}{\eps\gamma}$, but Proposition \ref{goodcase} implies that this cannot happen since now $\cI=\{N\}$.  Thus, all factors occurring in the iteration estimate satisfy \eqref{bdd} with $C=C_5$.   Application of Proposition \ref{2sided} then yields the result. 
\end{proof}

\begin{rem}\label{reduction}[Reduction to the case $\cD(\theta_{in})=d(\theta_{in})M$.]
Consider first the reduction in the case of Theorem \ref{tt26}.  Writing as in \eqref{i7}
\begin{align}\label{r1}
\cD(\theta_{in})=\sum^N_{i,j=1}d_{i,j}(\theta_{in})M_{i,j}, \text{ where }d_{i,j}(\theta_{in})=\sum_{r\in\ZZ\setminus 0}\alpha^{i,j}_re^{ir\theta_{in}},
\end{align}
we see that the transform of the singular problem \eqref{i6} is just like \eqref{i9}, except that the sum on the right is replaced by 
\begin{align}\label{r2}
\sum_{i,j=1}^N\sum_{r\in\ZZ\setminus 0}  \alpha^{i,j}_r e^{ir\frac{\omega_N(\beta_l)}{\eps}x_2}B_2^{-1}M_{i,j}V_{k-r}.
\end{align}
The assumption \eqref{b00} implies 
\begin{align}
\alpha^{i,j}_r=0 \text{ for } r\leq 0,   \; |\alpha^{i,j}_r|\leq Ar^{-(M+2)}  \text{ for }r\geq 1
\end{align}
for  some $M\geq 2$ and $A>0$  (both) independent of $(i,j)$.
The solution formulas \eqref{a5}, \eqref{a6} and \eqref{ab7}, \eqref{ab8} change in the obvious way when the replacement \eqref{r2} is made. 
The proof of the iteration estimate \eqref{t30c} was based on an analysis of the individual terms (each associated to a particular choice of $r$) in the solution formulas.  That analysis can be repeated for the new solution formulas to yield an iteration estimate of exactly the same form \eqref{t30c}, but  with $C$ replaced by $N^2C$ and $\alpha_r$ redefined as 
\begin{align}
\alpha_r=0 \text{ for }r < 0,\;\;\alpha_0=1,  \;\;  \alpha_r =Ar^{-(M+2)} \text{ for }r\geq 1.
\end{align}
The cascade estimates leading to the proof of Theorem \ref{tt26} for general $\cD$ can be carried out exactly as before using the new iteration estimate.

   The reduction in the case of Theorem \ref{tvv29} is carried out in the same way.
\end{rem}






\section{Multiple amplification and optimality of the estimates}\label{multiple}

\qquad In this section we prove Theorem \ref{multiamp}; that is, we construct and rigorously justify geometric optics solutions to the  $3\times 3$, strictly hyperbolic WR problem \eqref{d1} on $\Omega_T=(-\infty,T]\times \{(x_1,x_2):x_2\geq 0\}$:
 \begin{align}\label{d1z}
 \begin{split}
 &\partial_t u+B_1\partial_{x_1}u+B_2\partial_{x_2}u+e^{i\frac{\phi_{3}}{\eps}}Mu=0\text{ in }x_2>0\\
 &Bu=\eps G(t,x_1,\frac{\phi_0}{\eps}):=\eps g_1(t,x_1) e^{i\frac{\phi_0}{\eps}}\text{ on }x_2=0\\
 &u=0 \text{ in }t<0,
 \end{split}
 \end{align}
which exhibit instantaneous double amplification.   
The main new element in the proof of Theorem \ref{multiamp} is the construction of the approximate solution, which occupies most of the rest of this section.   The proof is concluded in section \ref{last}.

\begin{nota}
In previous sections we used $\theta\in\RR$ as a place holder for $\frac{\phi_0}{\eps}$.  In this section we use $\theta=(\theta_1,\theta_2,\theta_3)\in\RR^3$ as a placeholder for $\frac{\Phi}{\eps}$, $\Phi=(\phi_1,\phi_2,\phi_3)$, and we use $\theta_0$ as a placeholder for $\frac{\phi_0}{\eps}$.

\end{nota}

\subsection{Tools for constructing approximate solutions}  \label{tools2}

\qquad We recall here some useful results and introduce some notation.  We set
\begin{align}
\begin{split}
&L(\partial)=\partial_t +B_1\partial_{x_1}+B_2\partial_{x_2},\;\;\;L(\sigma,\xi)=\sigma I+B_1\xi_1+B_2\xi_2\\
&\cL(\partial_\theta)=\sum^3_{m=1}L(d\phi_m)\partial_{\theta_m}, \;\;\;\phi_m(t,x)=\beta_l\cdot (t,x_1)+\omega_m(\beta_l)x_2.
\end{split}
\end{align}

Let $\cA(\beta_l)$ be the matrix
\begin{align}
\cA(\beta_l)=-(A_0\sigma_l+A_1\eta_l), \text{ where }A_0=B_2^{-1}, A_1=B_2^{-1}B_1.
\end{align}
The matrix 
${\mathcal A}(\beta_l)$ is diagonalizable with eigenvalues $\omega_m(\beta_l)=\omega_m$, $m=1,\dots,3$, and 
the eigenspace of $\cA(\beta_l)$ for $\omega_m$  coincides with the kernel 
of $L(d\phi_m)$.

\begin{lem}\cite{CG}
\label{lem1}
The (extended) stable subspace $\E^s(\beta_l)$ (recall Prop. \ref{s7a}) admits the decomposition
\begin{equation}
\label{decomposition1}
\E^s(\beta_l) = \oplus_{m \in {\mathcal I}} \, \text{\rm Ker } L(d\phi_m) \, ,
\end{equation}
and each vector space in the decomposition \eqref{decomposition1} is of real type (that is, it admits a basis 
of real vectors).
\end{lem}

\begin{lem}\cite{CG}
\label{lem2}
The following decompositions hold
\begin{equation}
\label{decomposition2}
\C^3 = \oplus_{m=1}^3 \, \text{\rm Ker } L( d \phi_m) 
= \oplus_{m=1}^3 \, B_2 \, \text{\rm Ker } L(d \phi_m) \, ,
\end{equation}
and each vector space in the decompositions \eqref{decomposition2} is of real type.

We let $P_m$, respectively, $Q_m$, $m=1,2,3$, denote the projectors associated with the first, respectively,  second
decomposition in \eqref{decomposition2}.  For each $m$ there holds $\text{\rm Im } L( d \phi_m) 
= \text{\rm Ker } Q_m$.
\end{lem}

\noindent Using Lemma \ref{lem2}, we may introduce the partial inverse $R_m$ of $L( d\phi_m)$, 
which is uniquely determined by the relations
\begin{equation*}
\forall \, m=1,2,3\, ,\quad R_m \, L( d\phi_m) =I-P_m \, ,\quad  L( d\phi_m)R_m=I-Q_m,\quad P_m \, R_m=0 \, ,\quad R_m \, Q_m =0 \, .
\end{equation*}

In the case of our strictly hyperbolic system \eqref{d1},  we choose for each $m$ a real vector $r_m$ that 
spans $\text{\rm Ker } L( d \phi_m)$. We also choose real row vectors $\ell_m$, that satisfy
\begin{equation*}
\forall \, m=1,\dots,3 \, ,\quad \ell_m \, L( d\phi_m) =0 \, ,
\end{equation*}
together with the normalization $\ell_m \, B_2 \, r_{m'} =\delta_{mm'}$. With this choice, the partial inverse 
$R_m$ and the projectors $P_m$, $Q_m$ are given by\footnote{To see this write $\cA(\beta_l)=\sum_{m}\omega_mP_m$, which implies $L(d\phi_m)=\sum_{k\neq m}(\omega_m-\omega_k)B_2P_k$, and observe that $R_m=\sum_{k\neq m}\frac{P_kB_2^{-1}}{\omega_m-\omega_k}$.}
\begin{equation*}
\forall \, X \in \C^3 \, ,\quad R_m \, X=\sum_{m' \neq m} 
\dfrac{\ell_{m'} \, X}{{\omega}_m -{\omega}_{m'}} \, r_{m'} \,\quad P_mX=(\ell_m B_2X)r_m,\quad Q_mX=(\ell_m X)B_2r_m.
\end{equation*}

\quad 
In the analysis of the profile equations we use projection operators $E_Q$, $E_P$ defined on $H^\infty$:=$H^\infty(\Omega_T\times \TT^3)$ and a partial inverse $R$ of $\cL(\partial_\theta)$ defined on functions in a certain subspace of $H^\infty$.\footnote{A small divisor condition is  needed to define $R$ on functions  that have infinitely many noncharacteristic modes.}    
 We have
\begin{align}\label{d3a}
E_P=E_0+\sum^3_{m=1}E_{P_m},  \;E_{P_{in}}=E_{P_2}+E_{P_3}, \;\;\;  E_{P_{out}}=E_{P_1},
\end{align}
and similarly expand $E_Q$, replacing $P_m$ by $Q_m$. 
In  \eqref{d3a} $E_{P_0}$ picks out the mean and $E_{P_m}$ picks out pure $\theta_m$ modes and projects with $P_m$.  More precisely, 
writing
\begin{align}\label{d3b}
U(t,x,\theta)=\underline U(t,x)+U^*(t,x,\theta)=\underline{U}+ \sum^3_{m=1}U^m(t,x,\theta_m)+U^{nc}(t,x,\theta_1,\theta_2,\theta_3),
\end{align}
where  each $U^m$ has pure $\theta_m$ oscillations and mean zero and $U^{nc}$ is obtained by retaining only noncharacteristic modes in the Fourier series of $U$, we have\footnote{We refer to $\underline{U}$ and the terms appearing in the Fourier series of each $U_m$, $m=1,2,3$, as \emph{characteristic} modes.}
   \begin{align}\label{d3c}
   \begin{split}
   &E_0 U=\underline U, \;\; E_{P_m}U=P_m U^m(t,x,\theta_m),\; m=1,2,3\\
   &(I-E_P)U=\sum^3_{m=1}(I-P_m)U^m+U^{nc}.
   \end{split}
   \end{align}
Thus, along with \eqref{d3b} we can decompose $U$ as 
\begin{align}\label{d3d}
U(t,x,\theta)=\underline{U}+E_{P_1}U+E_{P_{in}}U+(I-E_P)U,
\end{align}    
and we have the obvious analogue of \eqref{d3d} for $E_Q$.

We make the following small divisor assumption.
\begin{ass}\label{sd}
There exist constants $C>0$ and $a\in\RR$ such that for all $(k,l)\in\NN\times\NN$ with $k\neq l$ we have
\begin{align}
|\det L(kd\phi_2+ld\phi_3)|\geq C|(k,l)|^{a}.
\end{align}
\end{ass}

\begin{rem}
It follows from a result of \cite{JMR2} that this assumption is ``generically valid"; that is, it is satisfied for almost all pairs of $\QQ-independent$ characteristic phases $\phi_2$, $\phi_3$.
\end{rem}

Next we define a subspace of $H^\infty$ on which the operator $R$ is well-defined.

\begin{defn}\label{subspace}
\begin{align}\label{d3e}
\begin{split}
&\cH^\infty:=
\{U\in H^\infty: U^{nc}=\sum_{(k,l)\in \NN\times \NN, k\neq l}c_{k,l}(t,x)e^{i(k\theta_2+l\theta_3)}\}.
\end{split} 
\end{align}

\end{defn}

Writing $U\in \cH^\infty$ as in \eqref{d3b} with $U^{nc}$ as in \eqref{d3e}, we define
\begin{align}\label{R}
\begin{split}
&R(\underline U)=0\\
&R(U^m)=\partial_{\theta_m}^{-1}R_mU^m\\
&R(U^{nc})=\mathcal{L}(\partial_\theta)^{-1}U^{nc}.
\end{split}
\end{align}
Here $\partial_{\theta_m}^{-1}R_mU^m$ denotes the unique mean zero primitive in $\theta_m$ of $R_mU^m$ and
\begin{align}\label{R1}
\mathcal{L}(\partial_\theta)^{-1}U^{nc}=\sum_{(k,l)\in \NN\times\NN, k\neq l}L(ikd\phi_2+ild\phi_3)^{-1}c_{k,l}(t,x)e^{i(k\theta_2+l\theta_3)}.
\end{align}
It is clear that as a consequence of the small divisor  assumption, we have $R:\cH^\infty\to\cH^\infty$.


As operators on $\cH^\infty$, the operators $\cL(\partial_\theta)$, $E_P$, $E_Q$, and $R$  are  easily seen to satisfy
\begin{align}\label{d6}
\begin{split}
& a) E_Q\cL(\partial_\theta)=\cL(\partial_\theta)E_P=0,\;\;   \\
& b) R\cL(\partial_\theta)=I-E_P,\;\;\cL(\partial_\theta)R=I-E_Q \\
& c) E_PR=RE_Q=0.
\end{split}
\end{align}

\subsection{Profile equations}

\qquad We construct approximate solutions to the system \eqref{d1z} of the form
\begin{align}\label{d3z}
u^\eps_a(t,x)=\sum_{k=-1}^J\eps^k U_k\left(t,x,\frac{\Phi}{\eps}\right),
\end{align}
where the profiles $U_k$ lie in  $\cH^\infty$.  Plugging the ansatz \eqref{d3z} into the system \eqref{d1z} and setting coefficients of different powers of $\eps$ equal to zero, 
we obtain interior equations
\begin{align}\label{d4}
\begin{split}
& a) \cL(\partial_\theta)U_{-1}=0\\
& b)\cL(\partial_\theta)U_0+L(\partial)U_{-1}+e^{i\theta_3}MU_{-1}=0\\
& c)\cL(\partial_\theta)U_1+L(\partial)U_{0}+e^{i\theta_3}MU_{0}=0\\
& d)\cL(\partial_\theta)U_j+L(\partial)U_{j-1}+e^{i\theta_3}MU_{j-1}=0, j\geq 2
\end{split}
\end{align}
and boundary equations
\begin{align}\label{d5}
\begin{split}
& a) BU_{-1}=0\\
& b) BU_0=0\\
& c) BU_{1}=G(t,x_1,\theta_0)=g_1(t,x_1)e^{i\theta_0}\\
& d) BU_j=0, j\geq 2.
\end{split}
\end{align}

\begin{rem}\label{d7a}
The left and right sides of boundary equations are always evaluated at $x_2=0$, $\theta_1=\theta_2=\theta_3=\theta_0$. 
\end{rem}

  The vector space $E^s(\beta_l)$ is spanned by $\{r_2,r_3\}$.   The subspace $\ker B\cap E^s(\beta_l)$ is one-dimensional and is therefore spanned by some vector 
\begin{align}\label{d7}
e=e_2+e_3, \;e_m\in\mathrm{ span }\;r_m.
\end{align}
The vector space $BE^s(\beta_l)$ is one-dimensional and of real type; thus, we can write it as 
\begin{align}\label{d8}
BE^s(\beta_l)=\{X\in \CC^2: b\cdot X=0\},
\end{align}
where $b$ is a suitable real nonzero row vector.  We also choose a supplementary space of $\mathrm{span}\; e$ in $E^s(\beta_l)$:
\begin{align}\label{d9}
E^s(\beta_l)=\check E(\beta_l)\oplus \mathrm{span}\;e,\;\; \check E(\beta_l)=\mathrm{span}\;\check e,\text{ where }\check e=\check e_2+\check e_3, \;\check e_m\in\mathrm{ span }\;r_m.
\end{align}
Thus, we have an isomorphism
\begin{align}\label{d10}
B:\check E^s(\beta_l)\to BE^s(\beta_l).
\end{align}

   For any $k$ we write
   \begin{align}
   \begin{split}
   &E_{P_1}U_k=\sigma_k^1(t,x,\theta_1)r_1\\
   &E_{P_{in}}U_k=\sigma_k^2(t,x,\theta_2)r_2+\sigma_k^3(t,x,\theta_3)r_3,
   \end{split}
   \end{align}
    for scalar profiles $\sigma_k^m$.    On the boundary we have using \eqref{d9}
    \begin{align}\label{d10a}
    E_{P_{in}}U_k=\sigma_k^2(t,x_1,0,\theta_0)r_2+\sigma_k^3(t,x_1,0,\theta_0)r_3=a_k(t,x_1,\theta_0)e+\check a_k(t,x_1,\theta_0)\check e
    \end{align}
    for some scalar profiles $a$, $\check a$ with mean zero.   We set
    \begin{align}
    \check U_k = \check a_k \check e.  
    \end{align}
    For a given $H$ the boundary equation $BU_k=H(t,x_1,\theta_0)$ can now be rewritten
    \begin{align}\label{d13}
   \begin{split}
   & a) B\underline{U_k}=\underline H\\
   & b) B\check U_k=H^*-BE_{P_1}U_k-B[(I-E_P)U_k]^*.
   \end{split}
    \end{align}
    By \eqref{d8} a solution $\check U_k$ to \eqref{d13}(b)  exists if and only if 
    \begin{align}\label{d14}
    b\cdot \left(H^*-BE_{P_1}U_k-B[(I-E_P)U_k]^*\right) =0.
    \end{align}

    We  write down now the  equations that will be used to determine $U_{-1}$, $U_0$, and $E_{P_1}U_1$.   The higher order equations follow the pattern that will be apparent.    All profiles are required to be zero in $t<0$. 
    
    First we have  \emph{interior} equations obtained by applying $E_Q$ to 
     equations in \eqref{d4}
\begin{align}\label{d15}
\begin{split}
&a) E_Q[L(\partial)U_{-1}+e^{i\theta_3}MU_{-1}]=0\\
&b) E_Q[L(\partial)U_{0}+e^{i\theta_3}MU_{0}]=0\\
&c) E_Q[L(\partial)U_{1}+e^{i\theta_3}MU_{1}]=0,
\end{split}
\end{align}    
and equations obtained by applying $R$ to equations in \eqref{d4}
\begin{align}\label{d16}
\begin{split}
&a)  (I-E_P)U_{-1}=0\\
&b)  (I-E_P)U_0=-R[L(\partial)U_{-1}+e^{i\theta_3}MU_{-1}]\\
&c)  (I-E_P)U_1=-R[L(\partial)U_0+e^{i\theta_3}MU_{0}].
\end{split}
\end{align}
Each of the equations in \eqref{d15} gives rise to four equations, one for each of  $E_{Q_i}$, $i=0,1,2,3$.

We also have boundary equations 
\begin{align}\label{d16a}
B\underline{U_k}=0,  \;\;k=-1,0,1, 
\end{align}
\begin{align}\label{d17}
\begin{split}
&a)  b\cdot [-BE_{P_1}U_0+B[R(L(\partial)U_{-1}+e^{i\theta_3}MU_{-1})]]=0\\
&b)  b\cdot [G-BE_{P_1}U_1+B[R(L(\partial)U_{0}+e^{i\theta_3}MU_{0})]]=0\\
&c)  b\cdot [-BE_{P_1}U_2+B[R(L(\partial)U_{1}+e^{i\theta_3}MU_{1})]]=0
\end{split}
\end{align}
and 
\begin{align}\label{d18}
\begin{split}
&a) B\check U_{-1}=-BE_{P_1}U_{-1}\\
&b)  B\check U_0=-BE_{P_1}U_0+B[R(L(\partial)U_{-1}+e^{i\theta_3}MU_{-1})]\\
&c)  B\check U_1=G-BE_{P_1}U_1+B[R(L(\partial)U_{0}+e^{i\theta_3}MU_{0})]\\
&d)  B\check U_2= -BE_{P_1}U_2+B[R(L(\partial)U_{1}+e^{i\theta_3}MU_{1})].
\end{split}
\end{align}
The equations \eqref{d17}, \eqref{d18} are derived in the obvious way from \eqref{d13}, \eqref{d14} using \eqref{d16}. 
The equations \eqref{d17} are the respective solvability conditions for the equations \eqref{d18}(b)-(d). 

\begin{rem}\label{eval}

The operator $R$ acts on functions of $\theta=(\theta_1,\theta_2,\theta_3)$, so it must be understood that the restriction to $\theta_1=\theta_2=\theta_3=\theta_0$ in equations \eqref{d17}, \eqref{d18} is done \emph{after} the action of $R$.  Moreover,   for $V\in\cH^\infty$  (Definition \ref{subspace}) the mean of $RV|_{\theta_1=\theta_2=\theta_3=\theta_0}$ is always zero.

\end{rem}

We will use the following proposition, which slightly modifies a classical result \cite{L}, to simplify the terms $E_{Q_m}L(\partial)E_{P_m}U$ that arise in the analysis of the profile equations.
\begin{prop}\label{transport}
For $U\in\cH^\infty$ let $E_{P_m}U=\sigma(t,x,\theta_m)r_m$ and let  $X_{\phi_m}=\partial_{x_2}-\partial_{\tau}\omega_m(\beta_l)\partial_{t}-\partial_{\eta}\omega_m(\beta_l)\partial_{x_1}$
be the transport vector field associated to the phase $\phi_m$.\footnote{The vector field $X_{\phi_m}$ satisfies $X_{\phi_m}=(\partial_t+\underline{v}_m\cdot\partial_x) c_m$,   where 
$c_m=(\partial_{\xi_2}\lambda_{k_m})^{-1}$ and $\underline v_m$ is the group velocity defined in \eqref{phases}.  One can see this by differentiating $\tau+\lambda_{k_m}(\eta,\omega_m(\tau,\eta))=0.$}

    We have
\begin{align}
E_{Q_m}L(\partial)E_{P_m}U=Q_m(L(\partial)\sigma(t,x,\theta_m)r_m)=(X_{\phi_m}\sigma) B_2r_m
\end{align}
\end{prop}

\begin{proof}
For $\beta=(\tau,\eta)$  near $\beta_l$ let $r_m(\beta)$ satisfy
\begin{align}\label{d19}
(\tau I +B_1\eta+B_2\omega_m(\beta))r_m(\beta)=0,  \;r_m(\beta_l)=r_m.
\end{align}
Differentiating \eqref{d19} with respect to $\tau$ and $\eta$, evaluating at $\beta_l$, and applying $\ell_m$ on the left we obtain:
\begin{align}
-\partial_\tau\omega_m(\beta_l)=\ell_m r_m\text{ and }-\partial_\eta\omega_m(\beta_l)=\ell_mB_1r_m.
\end{align}
Thus, 
\begin{align}
\ell_m(\partial_t+B_1\partial_{x_1}+B_2\partial_{x_2})r_m=X_{\phi_m}.
\end{align}

\end{proof}


\subsection{Table of modes}\label{table}

\quad One of the main challenges in constructing approximate solutions that exhibit double amplification is the high degree of coupling among the equations \eqref{d15}-\eqref{d18}.  By making an assumption about the various nonzero modes that appear in the profiles, it will turn out that we are able to decouple the equations.

 For each $m$, we now list  the characteristic and noncharacteristic modes that we \emph{expect} to find in $U_m$.  
The characteristic  mode $4\theta_1$, for example, will appear if the Fourier series of 
$U_m$ can possibly contain a term $c(t,x)e^{i4\theta_1}$ for some nonzero $c$. Similarly,  the noncharacteristic mode $2\theta_3+3\theta_2$ will appear if
the Fourier series of $U_m$ can possibly contain a term $c(t,x)e^{i(2\theta_3+3\theta_2)}$ for some nonzero $c$.   If for a given $U_m$ the mode $2\theta_3+3\theta_2$ does not appear in the list, this means  a term of the form $c(t,x)e^{i(2\theta_3+3\theta_2)}$ with $c\neq 0$  cannot possibly appear in the Fourier series of $U_m$.

The list is produced by first making a reasonable guess that takes into account the boundary data in \eqref{d1}, the single resonance \eqref{d2}, the profile equations, and the nature of the exact solution.   For example, the expectation  that no modes
$k_1\theta_1+k_2\theta_2+k_3\theta_3$  with $k_i<1$ should appear  is suggested by the fact that in the exact solution to the singular system corresponding to the problem \eqref{d1}, all terms $V_k(x_2,\zeta)$  as in \eqref{tz30} with $k<1$ are zero.
 In section \ref{construction} the list will be fully justified by the actual construction of profiles  satisfying the profile equations whose nonzero Fourier modes lie in this list.\\

\textbf{Characteristic modes:}
\begin{align}\label{cm}
\begin{split}
&a) U_{-1}:  k\theta_2,k\theta_3, k\geq 2\\
&b) U_0: k\theta_2,k\theta_3, k\geq 1;  2\theta_1,4\theta_1\\
&c) U_1: k\theta_2,k\theta_3, k\geq 1;  2\theta_1,4\theta_1,6\theta_1\\
&d) U_2: k\theta_2,k\theta_3, k\geq 1;  2\theta_1,4\theta_1,6\theta_1,8\theta_1,
\end{split}
\end{align}
and the pattern continues.\\

\textbf{Noncharacteristic modes:}\footnote{The term  ``add"  used in \eqref{nm}  for  $U_m$, $m\geq 1$ means that the noncharacteristic modes of $U_m$ are obtained by taking the modes listed after $U_m$ together with all the noncharacteristic modes of $U_{m-1}$.    Thus, for example, the full set of noncharacteristic modes of $U_2$ consists of all the modes appearing in  lines a-d of \eqref{nm}. }  
\begin{align}\label{nm}
\begin{split}
&a) U_{-1}: \mathrm{none}\\
&b) U_0: \theta_3+k\theta_2, k\geq 2\\
&c) U_1: \mathrm{add}\;(I)\; 2\theta_3+k\theta_2, k\geq 3 \;\mathrm{and}\;  (II)\; 2\theta_3+\theta_2, 3\theta_3+2\theta_2\\
&d) U_2: \mathrm{add}\;(I)\;3\theta_3+k\theta_2,k\geq 4\; \mathrm{and}\; (II)\; 3\theta_3+\theta_2, 4\theta_3+2\theta_2, 4\theta_3+3\theta_2\\
&e) U_3: \mathrm{add}\;(I)\;4\theta_3+k\theta_2, k\geq 5 \;\mathrm{and}\;  (II)\;4\theta_3+\theta_2, 5\theta_3+2\theta_2, 5\theta_3+3\theta_2, 5\theta_3+4\theta_2,
\end{split}
\end{align}
and the pattern continues.  The pattern for group (I) is clear.  For each $m\geq 1$ there is a one-to-one correspondence between the added noncharacteristic modes of $U_m$ in group II and the characteristic modes $k\theta_1$ of $U_{m-1}$.  The first $m$ modes in group (II) for $U_m$ are obtained by increasing by one the coefficient on $\theta_3$ in each group (II) mode for $U_{m-1}$ (for example, $3\theta_3+\theta_2\to 4\theta_3+\theta_2$).  The $(m+1)$st (and last) mode in group (II) for $U_m$ is obtained by computing the sum of $\theta_3$ and the last $k\theta_1$ mode of $U_{m-1}$.  For example, the mode $5\theta_3+4\theta_2$ in group (II) for $U_3$  arises in the ``interaction term" $-Re^{i\theta_3}MU_2$ of $(I-E_P)U_3$ as:\footnote{The first  equation in  \eqref{nm2} represents the relation holding between the phases $\phi_m$ associated to the $\theta_m$. }
\begin{align}\label{nm2}
\theta_3+8\theta_1=\theta_3+(4\theta_3+4\theta_2)=5\theta_3+4\theta_2.
\end{align}

\subsection{Construction of the profiles}\label{construction}
\quad We begin with a remark on notation and terminology.  
\begin{rem}
Recall that any function $U(t,x,\theta)\in\cH^\infty$ can be written as in \eqref{d3b} as 
\begin{align}
U(t,x,\theta)=\underline U(t,x)+\sum^3_{m=1} U^m(t,x,\theta_m)+U^{nc}(t,x,\theta_1,\theta_2,\theta_3).
\end{align}
Writing $U^m(t,x,\theta_m)=\sum_{p\in \ZZ}c^m_p(t,x)e^{ip\theta_m}$, we refer to $c^m_pe^{ip\theta_m}$ as the ``$p$-mode" of $U^m$ and to $\{c^m_pe^{ip\theta_m}, m=1,2,3\}$ as the ``$p$-modes of $U$".    We will apply this terminology to each profile $U_k$ in the expansion \eqref{d3} of $u_a(t,x)$.  In such a case we write
\begin{align}
U^m_k(t,x,\theta_m)=\sum_{p\in \ZZ}c^m_{k,p}(t,x)e^{ip\theta_m} \text{ and }U^m_{k,p}(t,x,\theta_m)=c^m_{k,p}(t,x)e^{ip\theta_m}.
\end{align}
When we speak, for example, of the ``\;$2\theta_3+3\theta_2$-mode of $U_k$", we are referring to the term $c(t,x)e^{i(2\theta_3+3\theta_2)}$ in the Fourier series of $U^{nc}$.

\end{rem}

\textbf{1.  Assumptions. } In order to construct profiles satisfying the profile equations, we \emph{assume} that profiles in $\cH^\infty$ satisfying the profile equations to any order  exist, and we \emph{assume} that the nonzero modes of those profiles appear in the above table.   As long as we can \emph{construct} explicit profiles in $\cH^\infty$  satisfying the profile equations which have only those modes,  it won't matter at all what assumptions we made to construct them.\footnote{In particular, there is no problem of circularity.} The process of construction will verify the non-obvious fact that our two assumptions are \emph{consistent} with each other.

We are not concerned about the issue of uniqueness of profiles, because we have a procedure for showing that sufficiently high order approximate  solutions are close in a precise sense to the exact solution (section \ref{last}), and we know the exact solution is unique.   Recall that all profiles are required to be zero in $t<0$. 

\textbf{2. The mean $\underline{U_k}=0$ for all $k$. }    Taking the mean of the equations in \eqref{d15} and using \eqref{d16a}, we obtain for any $k\geq -1$ 
\begin{align}
\begin{split}
&E_{Q_0}[L(\partial)U_k+e^{i\theta_3}MU_k]=L(\partial)\underline{U_k}=0\\
&B\underline{U_k}=0,
\end{split}
\end{align}
and thus $\underline{U_k}=0$.

\textbf{3. Determination of $E_{P_1}U_{-1}$ and $(I-E_P)U_{-1}$, $\check U_{-1}$.  }From \eqref{d16} we have $(I-E_P)U_{-1}=0$, so 
\begin{align}
U_{-1}=\sigma_{-1}^1(t,x,\theta_1)r_1+\sigma_{-1}^2(t,x,\theta_2)r_2+\sigma^3_{-1}(t,x,\theta_3)r_3,
\end{align}
for some scalar profiles $\sigma_{-1}^m$.  Using assumption \eqref{cm}(a)  we see there is no resonance in the interaction term $e^{i\theta_3}MU_{-1}$, so we obtain
from \eqref{d15}, Proposition \ref{transport},  and the definition of $E_{Q_1}$
\begin{align}
E_{Q_1}[L(\partial)U_{-1}+e^{i\theta_3}MU_{-1}]=(X_{\phi_1}\sigma_{-1}^1) B_2r_1=0.
\end{align}
Thus, $E_{P_1}U_{-1}=0$, since the vector field $X_{\phi_1}$ is outgoing and $\sigma_{-1}^1=0$ in $t<0$. 

Similarly, we obtain
\begin{align}\label{e0}
\begin{split}
&E_{Q_2}[L(\partial)U_{-1}+e^{i\theta_3}MU_{-1}]=X_{\phi_2}\sigma_{-1}^2 B_2r_2=0\\
&E_{Q_3}[L(\partial)U_{-1}+e^{i\theta_3}MU_{-1}]=X_{\phi_3}\sigma_{-1}^3 B_2r_3+e^{i\theta_3}\sigma_{-1}^3Q_3(Mr_3)=(X_{\phi_3}\sigma_{-1}^3 +c_{-1}^3e^{i\theta_3}\sigma_{-1}^3)B_2r_3=0,
\end{split}
\end{align}
where the constant $c_{-1}^3=\ell_3Mr_3$.  

 It remains to determine the traces of $\sigma_{-1}^m$, $m=2,3$ on $x_2=0$.   From \eqref{d18}(a) we obtain $\check U_{-1}=0$.   Summarizing, we have so far:
 \begin{align}
 U_{-1}=\sigma_{-1}^2(t,x,\theta_2)r_2+\sigma^3_{-1}(t,x,\theta_3)r_3,\;\;\check U_{-1}=0, \;\;\sigma_{-1}^m, m=2,3 \text{ undetermined}.
\end{align}

\textbf{4. Determination of the 1-modes of $U_0$ and the 2-mode of $E_{P_1}U_0$. } The traces of $\sigma_{-1}^m, m=2,3$ are coupled to $E_{P_1}U_0$ by \eqref{d17}(a). Writing
\begin{align}
\begin{split}
&E_PU_0=\sigma_{0}^1(t,x,\theta_1)r_1+\sigma_{0}^2(t,x,\theta_2)r_2+\sigma^3_{0}(t,x,\theta_3)r_3, \text{ where }\\
&\quad \sigma_{0}^m(t,x,\theta_m)=\sum^\infty_{p=1}\sigma^m_{0,p}(t,x)e^{ip\theta_m},\; m=1,2,3,
\end{split}
\end{align}  
we now determine the 1-modes of $E_{P_{in}}U_0$ and the 2-mode of $E_{P_1}U_0$, that is, $\sigma_{0,1}^2e^{i\theta_2}$, $\sigma_{0,1}^3e^{i\theta_3}$, and $\sigma_{0,2}^1e^{i2\theta_1}$.   

Using \eqref{d7}, \eqref{d9}, and \eqref{d10a} we obtain
\begin{align}\label{e1}
\begin{split}
&\sigma^2_0(t,x_1,0,\theta_0)r_2=a_0(t,x_1,\theta_0)e_2+\check a_0(t,x_1,\theta_0)\check e_2\\
&\sigma^3_{0}(t,x_1,0,\theta_0)r_3=a_0(t,x_1,\theta_0)e_3+\check a_0(t,x_1,\theta_0)\check e_3.
\end{split}
\end{align}
From \eqref{d18}(b) we see that $\check U_0$ has no 1-mode; thus $\check a_{0,1}(t,x_1)=0$.  

Next we determine the 1-mode of $a_0(t,x_1,\theta_0)$, that is, $a_{0,1}(t,x_1)e^{i\theta_0}$.  The terms of \eqref{d17}(b) involving $E_{P_1}$ and $M$ have no 1-modes.       Thus, if we differentiate \eqref{d17}(b) with respect to $\theta_0$ and consider the 1-mode of the resulting equation,  this reduces to the 1-mode of 
\begin{align}\label{e1a}
-b\cdot\partial_{\theta_0}G=b\cdot \partial_{\theta_0} BRL(\partial)U_0. 
\end{align}
By \eqref{d16}(b) the only 1-modes of $U_0$ are in $E_{P}U_0$ and thus in $E_{P_{in}}U_0$ , so we obtain\footnote{The antiderivative in $R$ is ``undone" by $\partial_{\theta_0}$.}
\begin{align}\label{e1aa}
(X_{Lop}a_{0,1})e^{i\theta_0}=-b\cdot \partial_{\theta_0}G=-ib\cdot g_1(t,x_1)e^{i\theta_0},
\end{align}
which determines $a_{0,1}$.  Here 
\begin{align}
X_{Lop}=c_0\partial_t+c_1\partial_{x_1}, \; c_0\neq 0,\;\;c_j\in\RR
\end{align}
is a characteristic vector field of the Lopatinski determinant.   The derivation of \eqref{e1aa} from \eqref{e1a} as well as the derivations (and solvability) of other equations involving $X_{Lop}$ occurring below 
are discussed in section \ref{be}.

We now know the traces of $\sigma_{0,1}^m$, $m=2,3$. 
Using again the observation that the only 1-modes of $U_0$ are in $E_{P_{in}}U_0$,  we can  determine  $\sigma_{0,1}^m$, $m=2,3$ from the 1-modes of 
\begin{align}
\begin{split}
&E_{Q_2}L(\partial)E_{P_2}U_0=0\\
&E_{Q_3}L(\partial)E_{P_3}U_0=0.
\end{split}
\end{align}
That is, we determine  $\sigma_{0,1}^m$, $m=2,3$ by solving:
\begin{align}\label{ee1}
X_{\phi_m}\sigma_{0,1}^m=0, \;\;(\sigma^m_{0,1}|_{x_2=0}) r_m=a_{0,1}e_m.
\end{align}
By \eqref{d16}(b) $(I-P_1)U_0^1=0$, so we may determine the 2-mode of $E_{P_1}U_0$ from the 2-mode of 
\begin{align}
E_{Q_1}[L(\partial)E_{P_1}U_0+e^{i\theta_3}MU_0]=0.  
\end{align}
Taking account of the resonance, this gives
\begin{align}\label{ee2}
(X_{\phi_1}\sigma_{0,2}^1 \;e^{i2\theta_1} +c_{0}^1e^{i2\theta_1}\sigma_{0,1}^2 )B_2r_1=0, \text{ where }c_0^1=\ell_1Mr_2,
\end{align}
which yields $\sigma_{0,2}^1$.  

\textbf{5. Determination of $E_{P_1}U_0$, $U_{-1}$, and $(I-E_P)U_0$, $\check U_0$.} Parallel to \eqref{e1} we have
\begin{align}\label{e2}
\sigma_{-1}^m(t,x_1,0,\theta_0)r_m=a_{-1}(t,x_1,\theta_0)e_m, \;m=2,3.
\end{align}
Using \eqref{d17}(a), we determine the 2-mode of $a_{-1}$, $a_{-1,2}e^{i2\theta_0}$, from the 2-mode of\footnote{The right side of \eqref{e2aa} has no 1-mode, so $a_{-1,1}=0$.}
\begin{align}\label{e2aa}
\partial_{\theta_0} (b\cdot BRL(\partial)U_{-1})=\partial_{\theta_0} (b\cdot BE_{P_1}U_0), 
\end{align}
that is,
\begin{align}\label{e2a}
(X_{Lop}a_{-1,2}) e^{i2\theta_0}=\partial_{\theta_0}(b\cdot B e^{i2\theta_0}\sigma_{0,2}^1 r_1).
\end{align}
  In view of  \eqref{e2} we now know $\sigma_{-1,2}^m(t,x_1,0)e^{i2\theta_m}$, $m=2,3$, so we can determine $\sigma_{-1,2}^m(t,x)e^{i2\theta_m}$, $m=2,3$,  using the 2-modes of the equations \eqref{e0}. 

Having $\sigma_{-1,2}^2(t,x)e^{i2\theta_2}$, we can determine the $\theta_3+2\theta_2$ mode of $(I-E_P)U_0$, $c(t,x)e^{i\theta_3+2\theta_2}$,  from \eqref{d16}(b).
With this we determine the 4-mode of $E_{P_1}{U_0}$, $\sigma^1_{0,4}e^{i4\theta_1}$ using the resonance in 
\begin{align}
E_{Q_1}[L(\partial)E_{P_1}U_0+e^{i\theta_3}MU_0]=0.
\end{align}
Thus, we now have $E_{P_1}U_0$.    Note that to produce a 6-mode in $E_{P_1}U_0$ we would need $U_0$ to have a $2\theta_3+3\theta_2$ mode, but \eqref{d16}(b) shows it does not.

Knowing $E_{P_1}U_0$, we proceed to determine the higher ($k\geq 3$) modes of $a_{-1}(t,x_1,\theta_0)$ using ``$\partial_{\theta_0}$\eqref{d17}(a)":\footnote{Since $U_{-1}^2$ has a possibly nonzero 2-mode, the interaction term in \eqref{d17}(a) contributes a 3-mode, so $a_{-1}$ (and hence $U_{-1}$)  may have a nonzero 3-mode.  But then the interaction term may contribute a nonzero 4-mode, etc...}  
\begin{align}
X_{Lop}a_{-1}=\partial_{\theta_0}[b\cdot (BE_{P_1}{U_0}-BRe^{i\theta_3}MU_{-1})].
\end{align}
This equation turns out (see section \ref{be}) to have the form 
\begin{align}\label{e3b}
 X_{Lop}a_{-1}+e^{i\theta_0}m(D_{\theta_0})a_{-1}=g(t,x_1,\theta_0),
 \end{align}
 where $g$ is known  and $m(D_{\theta_0})$ is a \emph{bounded} Fourier multiplier,

Having $a_{-1}$  we can complete the determination of $U_{-1}$ using the interior equations \eqref{e0}.   We note that the nonzero modes of  $U_{-1}$ lie in the list of   section \ref{table}.   Next we use \eqref{d16}(b) to determine $(I-E_P)U_0$ from $U_{-1}$.  From this we see that $(I-E_P)U_0$ has characteristic (respectively, noncharacteristic) modes only of the forms
\begin{align}\label{e3}
k\theta_3,k\theta_2, k\geq 2 \text{ and } \theta_3+k\theta_2, k\geq 2
\end{align}
respectively.  Knowing $E_{P_1}U_0$ and $U_{-1}$, we determine $\check U_0$ from \eqref{d18}(b).  It remains to determine the higher ($k\geq 2$) modes of $E_{P_{in}}U_0$.

\textbf{6. Determination of the 1-modes of $U_1$, the 2-mode of $E_{P_1}U_1$, and  the 2-modes of $E_{P_{in}}U_0$}. Using \eqref{d16}(c) we determine the 1-modes of $(I-E_P)U_1$, that is, the 1-modes of 
$(I-P_m)U_1^m(t,x,\theta_m)$, $m=2,3$. 

Parallel to \eqref{e1} we have
\begin{align}
\sigma^m_1(t,x_1,0,\theta_0)r_m=a_1(t,x_1,\theta_0)e_m+\check a_1(t,x_1,\theta_0)\check e_m, \;m=2,3.
\end{align}
From \eqref{d18}(c) we can determine the 1-mode of $\check a_1$, that is $\check a_{1,1}(t,x_1)e^{i\theta_0}$.  We determine the 1-mode of $a_1$, that is $a_{1,1}(t,x_1)e^{i\theta_0}$, from  the 1-mode of ``$\partial_{\theta_0}$\eqref{d17}(c)".   This equation reduces to the 1-mode of 
\begin{align}\label{e3a}
\partial_{\theta_0}[b\cdot BRL(\partial)(E_{P_{in}}U_1+(I-E_P)U_1)]=0,
\end{align}
where  the 1-mode of the $(I-E_P)U_1$ term is known.   

Having the traces of $\sigma_{1,1}^m$, $m=2,3$, we now use the 1-modes of \eqref{d15}(c), which reduce to the 1-modes of
\begin{align}
E_{Q_{in}}[L(\partial)(E_{P_{in}}U_1+(I-E_P)U_1)]=0,
\end{align}
to finish the determination of $\sigma_{1,1}^m$, $m=2,3$.  We now have the 1-modes of $U_1$.

Letting $U^2_{1,1}(t,x,\theta_2)$ denote the 1-mode of $U_1^2(t,x,\theta_2)$ and taking account of the resonance, we can now determine the 2-mode of $E_{P_1}U_1$, that is, 
$\sigma^1_{1,2}(t,x)e^{i2\theta_1}$ from the 2-mode of 
\begin{align}
E_{Q_1}[L(\partial)(E_{P_1}U_1+(I-P_1)U^1_1)+e^{i\theta_3}MU^2_{1,1}]=0.
\end{align}
Note that the 2-mode of $(I-P_1)U^1_1$ is known from \eqref{d16}(c) and the result of step \textbf{4}.

Having the 2-mode of $E_{P_1}U_1$ we proceed to determine the 2-modes of $E_{P_{in}}U_0$, starting with their traces (recall \eqref{e1}).  For this we consider the 2-mode of ``$\partial_{\theta_0}$\eqref{d17}(b)", which reduces to the 2-mode of 
\begin{align}\label{e4}
\partial_{\theta_0} \{b\cdot [-BE_{P_1}U_1+B[R(L(\partial)U_0+e^{i\theta_3}M(U^2_{0,1}+U^3_{0,1}))]]\}=0.
\end{align}
We can write 
\begin{align}\label{e4a}
U_0=E_{P_{in}}U_0+E_{P_1}U_0+(I-E_P)U_0,
\end{align}
where the last two terms on the right are known.   Since $\check U_0$ is known, the 2-mode of equation \eqref{e4} reduces to a propagation equation involving $X_{Lop}$  for  the only unknown, 
the 2-mode of $a_0(t,x_1,\theta_0)$. 

We now have the traces of $\sigma_{0,2}^me^{i2\theta_m}$, $m=2,3$.  To determine these modes we consider the 2-modes of \eqref{d15}(b).  The relevant equations reduce to the 2-modes of the incoming equations
\begin{align}
\begin{split}
&E_{Q_2}[L(\partial)(E_{P_2}U_0+(I-E_P)U_0)]=0\\
&E_{Q_3}[L(\partial)(E_{P_3}U_0+(I-E_P)U_0)+e^{i\theta_3}MU^3_{0,1}]=0,
\end{split}
\end{align}
where the only unknowns are the 2-modes of $E_{P_m}U_0$, $m=2,3$.   These are now determined.

\textbf{7. Determination of $E_{P_1}U_1$, $U_0$ and $(I-E_P)U_1$, $\check U_1$.} To complete the determination of $E_{P_1}U_1$ we use the equations
\begin{align}\label{e5}
\begin{split}
&(a) (I-E_{P})U_1=-R[L(\partial)U_0+e^{i\theta_3}MU_0]\\
&(b) E_{Q_1}[L(\partial)(E_{P_1}U_1+(I-P_1)U^1_1)+e^{i\theta_3}MU_1]=0.
\end{split}
\end{align}
We get the 4-mode of $E_{P_1}U_1$ from the 4-mode of \eqref{e5}(b), after first determining the $e^{i4\theta_1}$ modes  of $(I-P_1)U^1_1$ and $e^{i\theta_3}MU_1$.   To get the 4-mode of $(I-P_1)U^1_1$ we  use equation \eqref{e5}(a), the known 4-mode of  $E_{P_1}U_0$,  and the known $\theta _3+2\theta_2$ mode of $U_0$.  To get the $e^{i4\theta_1}$ mode of $e^{i\theta_3}MU_1$, we need the $\theta_3+2\theta_2$ mode of $U_1$, and we can obtain the latter from \eqref{e5}(a) since we know both the $\theta_3+2\theta_2$ mode of $U_0$ and the $2$ mode of $U^2_0$. 

Similarly, we  get the 6-mode of $E_{P_1}U_1$ from the 6-mode of \eqref{e5}(b), after first determining the $e^{i6\theta_1}$ modes  of $(I-P_1)U^1_1$ and $e^{i\theta_3}MU_1$.    The 6-mode of $(I-P_1)U^1_1$ is zero by equation \eqref{e5}(a).     For  the $e^{i6\theta_1}$ mode of $e^{i\theta_3}MU_1$, we need the $2\theta_3+3\theta_2$ mode of $U_1$, and we can obtain the latter from \eqref{e5}(a) since we know the $\theta_3+3\theta_2$ mode of $U_0$.   Thus, we now have $E_{P_1}U_1$.\footnote{To produce an 8-mode in $E_{P_1}U_1$ we would need $U_1$ to have a $3\theta_3+4\theta_2$ mode, but \eqref{e5}(a) shows it does not.} 

We now complete the determination of $U_0$, by first determining $a_0(t,x_1,\theta_0)$ (as in \eqref{e1}). For this we use the boundary equation
\begin{align}\label{e6}
\partial_{\theta_0}\{b\cdot[G-BE_{P_1}U_1+BR(L(\partial)U_0+e^{i\theta_3}MU_0)]\}=0.
\end{align}
Since the only unknown piece of $U_0$ is $E_{P_{in}}U_0$ and $\check U_0$ is known, the only unknown in \eqref{e6} is $a_0$, which we determine by again solving an equation of the form \eqref{e3b}. 

In view of \eqref{e1} we now have the trace of $E_{P_{in}}U_0$,  so we can determine $E_{P_{in}}U_0$ by solving the incoming equations
\begin{align}
\begin{split}
&E_{Q_2}[L(\partial)(E_{P_2}U_0+(I-E_P)U_0)]=0\\
&E_{Q_3}[L(\partial)(E_{P_3}U_0+(I-E_P)U_0)+e^{i\theta_3}M(E_{P_3}U_0+(I-E_P)U_0)]=0.
\end{split}
\end{align}
 This gives  $U_0$.    We determine $(I-E_P)U_1$ and $\check U_1$ from \eqref{d16}(c) and \eqref{d18}(c), respectively.  Observe that the nonzero modes of  $U_0$ lie in the list of  section \ref{table}.    Although $U_1$ is not yet completely determined, \eqref{d16}(c) implies that the nonzero noncharacteristic modes of $U_1$ must lie in that list.

\textbf{8. Determination of the 1-modes of $U_2$, the 2-mode of $E_{P_1}U_2$, and  the 2-modes of $E_{P_{in}}U_1$.}    Except for obvious small changes, the determination of these modes is by an almost verbatim repetition of step \textbf{6. }  For example, instead of considering the 1-mode of \eqref{e3a}, one now considers the 1-mode of 
\begin{align}\label{e7}
\partial_{\theta_0}[b\cdot BRL(\partial)(E_{P_{in}}U_2+(I-E_P)U_2)]=0.
\end{align}

\textbf{9. Determination of $E_{P_1}U_2$, $U_1$ and $(I-E_P)U_2$, $\check U_2$.}    This step is a near repetition of step \textbf{7. } In place of \eqref{e5} we now use
\begin{align}\label{e8}
\begin{split}
&(a) (I-E_{P})U_2=-R[L(\partial)U_1+e^{i\theta_3}MU_1]\\
&(b) E_{Q_1}[L(\partial)(E_{P_1}U_2+(I-P_1)U^1_2)+e^{i\theta_3}MU_2]=0.
\end{split}
\end{align}
At this point we know the 2-mode of $E_{P_1}U_2$ and we need to determine its 4, 6, and 8-modes.  Just as in step \textbf{7} we used \eqref{e5}(b) and our knowledge of the 2-mode of $U_0^2$ to get the 4-mode of $E_{P_1}U_1$, in this step we use \eqref{e8}(b) and our knowledge of the 2-mode of $U^2_1$ to get the 4-mode of $E_{P_1}U_2$.   Just as in step \textbf{7} no higher modes of $U^2_0$ were needed to determine the 6-mode of $E_{P_1}U_1$, in this step no higher modes of $U^2_1$ are needed to get the 6 and 8 modes of $E_{P_1}U_2$. 

   With $E_{P_1}U_2$ in hand the construction of $U_1$ is completed by repeating the argument of step \textbf{7} to first construct $a_1(t,x_1,\theta_0)$, and then to construct $E_{P_{in}}U_1$.   One then obtains $(I-E_P)U_2$ and $\check U_2$ in the usual way.
   
 \textbf{10. Conclusion. }The inductive pattern is now clear.  For any $M$ this argument constructs profiles $U_{-1}$, $U_0$, ..., $U_M$ in $\cH^\infty$ satisfying the profile equations \eqref{d4}, \eqref{d5}.  Moreover, the nonzero modes of these profiles lie in the list of section \ref{table}.\\

 \subsection{Analysis of the boundary amplitude equations}\label{be}

 \emph{\quad}      In this section we show that for each $j\geq -1$ the boundary equation for the amplitude $a_j(t,x_1,\theta_0)$ in the construction of section \ref{construction}
 is an equation of the form
 \begin{align}\label{f00}
 X_{Lop}a_{j}+e^{i\theta_0}m_j(D_{\theta_0})a_{j}=g_j(t,x_1,\theta_0),\;\;g_j=0\text{ in }t<0,
 \end{align}
 where $g_j(t,x_1,\theta_0)=\sum_{k\geq 1}g_{j,k}(t,x_1)e^{ik\theta_0}$ is known, $m_j(D_{\theta_0})$ is a \emph{bounded} Fourier multiplier, and \footnote{In fact, $X_{Lop}$ is a characteristic vector field of the Lopatinski determinant.}
 \begin{align}\label{f0a}
 X_{Lop}=c_0\partial_t+c_1\partial_{x_1} \text{ with }c_j\in\RR  \text{ and }c_0\neq 0.
\end{align}

  The equation for $a_j$  derives from an equation of the form\footnote{As usual, the bracketed term on the left in \eqref{f0} is restricted to $\theta_2=\theta_3=\theta_0$ after the action of $R$.}
  \begin{align}\label{f0}
\partial_{\theta_0}\{ b\cdot B[R(L(\partial)U_j+e^{i\theta_3}MU_j)]\}=h_j(t,x_1,\theta_0),
 \end{align}
 where $h_j$ is known from previous steps, and $U_j\in\cH^\infty$ has mean zero:
 \begin{align}\label{f1}
 U_j=E_{P_{in}}U_j+E_{P_1}U_j+(I-E_P)U_j.
 \end{align}
 The  terms $E_{P_1}U_j+(I-E_P)U_j$ are also known from previous steps, and on $x_2=0$ we may write (recall \eqref{d10a}) 
 \begin{align}\label{f2}
    \begin{split}
    &E_{P_{in}}U_j(t,x_1,0,\theta)=\sigma_j^2(t,x_1,0,\theta_2)r_2+\sigma_j^3(t,x_1,0,\theta_3)r_3=\\
    &\qquad (a_j(t,x_1,\theta_2)e_2+\check a_j(t,x_1,\theta_2)\check e_2)+(a_j(t,x_1,\theta_3)e_3+\check a_j(t,x_1,\theta_3)\check e_3),
    \end{split}
    \end{align}
   where $\check a_j$, too, is known from previous steps.     Dropping all subscripts $j$ and using that fact that $R_mB_2P_m=0$ for $m=2,3$, we see then that \eqref{f0} reduces to an equation of the form
   \begin{align}\label{f3}
 \begin{split}
  &\partial_{\theta_0}\left[ b\cdot BR\left(L'(\partial)(a(t,x_1,\theta_2)e_2+a(t,x_1,\theta_3) e_3)+e^{i\theta_3}M(a(t,x_1,\theta_2)e_2+a(t,x_1,\theta_3) e_3)\right)\right]=\\
  &\qquad\qquad h(t,x_1,\theta_0), \;\;h=0\text{ in }t\leq 0,
\end{split}
\end{align}
where $h$ is known and $L'(\partial)=\partial_t+B_1\partial_{x_1}$.     
 
 It is shown in [CG] that 
 \begin{align}\label{f4}
   \partial_{\theta_0}[b\cdot BRL'(\partial)(a(t,x_1,\theta_2)e_2+a(t,x_1,\theta_3) e_3)]=X_{Lop}a(t,x_1,\theta_0)
   \end{align}
for $X_{Lop}$ as in \eqref{f0a}.
Clearly,
\begin{align}\label{f4a}
\partial_{\theta_0}[ b\cdot BR(e^{i\theta_3}Ma(t,x_1,\theta_3) e_3)]=\alpha_1e^{i\theta_0}a(t,x_1,\theta_0)\text{ for } \alpha_1=b\cdot BR_3Me_3,
\end{align}
so it remains to analyze the term $\partial_{\theta_0} [b\cdot BR\left(e^{i\theta_3}Ma(t,x_1,\theta_2)e_2\right)]$.
 
 Writing $a(t,x_1,\theta_0)=\sum_{k\geq 1}a_k(t,x_1)e^{ik\theta_0}$, we have (with $\beta_1=b\cdot BR_1Me_2$)
 \begin{align}\label{f4b}
 \begin{split}
& \partial_{\theta_0} [b\cdot BR\left(e^{i\theta_3}Ma(t,x_1,\theta_2)e_2\right)]=\partial_{\theta_0} [b\cdot BR(Me_2\sum_{k\geq 1}a_ke^{ik\theta_2+i\theta_3})]=\\
&\quad \beta_1e^{i2\theta_0}a_1(t,x_1)+ \partial_{\theta_0} [b\cdot B\sum_{k\geq 2}(L^{-1}(ikd\phi_2+id\phi_3)Me_2)a_ke^{i(k+1)\theta_0}]=\\
&\quad \beta_1e^{i\theta_0}(a_1(t,x_1)e^{i\theta_0})+ b\cdot B\sum_{k\geq 2}(L^{-1}(kd\phi_2+d\phi_3)Me_2)(k+1)a_{k}e^{i(k+1)\theta_0}.
\end{split}
\end{align}
 We can simplify the second term in the last line using
 \begin{lem}
  For each $k\in \{ 2,3,\dots\}$ the number $k(\omega_2-\omega_1)+(\omega_3-\omega_1)$ is nonzero, and for any $X\in\CC^3$
 \begin{align}\label{f4bb}
 L^{-1}(kd\phi_2+d\phi_3)X=c_1r_1+c_2r_2+c_3r_3, \text{ where }c_1=\frac{\ell_1 X}{k(\omega_2-\omega_1)+(\omega_3-\omega_1)}.
 \end{align}
 \end{lem}
 
 \begin{proof}
 We can write $L(d\phi_m)=\sum_{m\neq m'}(\omega_m-\omega_m')B_2P_{m'}$, so $L(d\phi_m)r_p=(\omega_m-\omega_p)B_2r_p$ and 
 \begin{align}
 L(kd\phi_2+d\phi_3)r_1=(k(\omega_2-\omega_1)+(\omega_3-\omega_1))B_2r_1.
 \end{align}
 For $k\in\{2,3,\dots\}$  it follows that $k(\omega_2-\omega_1)+(\omega_3-\omega_1)\neq 0$, since otherwise $L(kd\phi_2+d\phi_3)$ would have a nontrivial kernel.

 Thus, \eqref{f4bb} follows by computing
 \begin{align}
 \ell_1X=\ell_1 L(kd\phi_2+d\phi_3)(c_1r_1+c_2r_2+c_3r_3)=c_1 [k(\omega_2-\omega_1)+(\omega_3-\omega_1)].
 \end{align}
 \end{proof}
 
 Using the lemma and the fact that $b\cdot Br_p=0$, $p=2,3$, we obtain
 \begin{align}\label{f4c}
 \begin{split}
& b\cdot B\sum_{k\geq 2}(L^{-1}(kd\phi_2+d\phi_3)Me_2)(k+1)a_{k}e^{i(k+1)\theta_0}=\\
&\qquad (b\cdot Br_1)(\ell_1Me_2)e^{i\theta_0}\sum_{k\geq 2}\frac{(k+1)}{k(\omega_2-\omega_1)+(\omega_3-\omega_1)}a_{k}e^{ik\theta_0}.
 \end{split}
 \end{align}
The set of numbers $\{\frac{(k+1)}{k(\omega_2-\omega_1)+(\omega_3-\omega_1)}, k\in \{ 2,3,\dots\}\}$ is clearly bounded.     Thus, the equation \eqref{f3} takes the form
\begin{align}\label{f4d}
X_{Lop}a+e^{i\theta_0}m(D_{\theta_0})a=h(t,x_1,\theta_0), \; h=0\text{ in }t<0,
\end{align}
where the components $m(k)$ defining the bounded Fourier multiplier $m(D_{\theta_0})$ can be read off from \eqref{f4a}, \eqref{f4b}, and  \eqref{f4c}.   A standard argument based on a simple energy estimate yields a unique solution $a(t,x_1,\theta_0)\in H^\infty((-\infty,T]\times \RR\times\TT)$ satisfying $a=0$ in $t<0$. 

\begin{rem}
In section \ref{construction} the modes $a_{j,1}e^{i\theta_0}$, $a_{j,2}e^{i2\theta_0}$ were determined for each $a_j$ by simple equations like \eqref{e1aa}, \eqref{e2a} \emph{before} the full profile $a_j(t,x_1,\theta_0)$ was determined by an equation like \eqref{f00}.\footnote{We found $a_{-1,1}=0$ in step \textbf{5} of the construction of section \ref{construction}.}  It is easy to check that the 1 and 2-modes of the solution to \eqref{f00} agree with the previously determined modes.  
\end{rem}

\subsection{Proof of Theorem \ref{multiamp}.}\label{last}
\textbf{1. Part (a).}  The existence  and uniqueness of an exact solution $u^\eps(t,x)\in H^\infty(\Omega_T)$ to \eqref{d1}  follows  by applying  the results of \cite{C} (for WR problems without highly oscillatory coefficients) to the problem obtained from \eqref{d1}  by  \emph{fixing} any particular $\eps\in (0,\eps_0]$. 

Theorem \ref{tt26}(c) yields an exact solution $U^\eps(t,x,\theta_0)\in L^2$ to the  singular problem \eqref{i6} corresponding to \eqref{d1}.\footnote{To apply Theorem \ref{tt26}(c) we must take an extension of $g(t,x_1)$ to $t>T$.  This standard maneuver is used also in step \textbf{2} below; for more detail, see p. 586 of \cite{CGW3}, for example.}   One can prove higher regularity of $U^\eps$ (in fact, $U^\eps\in H^\infty(\Omega_T\times \TT)$) 
by differentiating the singular problem and repeating the argument of section 6 of \cite{W3}.  Since $v^\eps:=U^\eps(t,x,\theta_0)|_{\theta_0=\frac{\phi_0}{\eps}}$ is a solution of \eqref{d1} on $(-\infty,T]$, we conclude $u^\eps=v^\eps$. 

\textbf{2.  Part (b).} The construction carried out in sections \ref{tools2}-\ref{be} yields profiles $U_k(t,x,\theta_1,\theta_2,\theta_3)$ as in \eqref{d3} satisfying the profile equations \eqref{d4}, \eqref{d5} for $j=1,\dots, J$, where $J$ is as large as desired.  With $\theta=(\theta_1,\theta_2,\theta_3)$ define
\begin{align}\label{f5}
\begin{split}
&\cU(t,x,\theta)=\sum^J_{k=-1}\eps^k U_k(t,x,\theta),\\
&U^\eps_a(t,x,\theta_0)=\cU\left(t,x,\theta_0+\frac{\omega_1 x_2}{\eps},\theta_0+\frac{\omega_2 x_2}{\eps},\theta_0+\frac{\omega_3 x_2}{\eps}\right),
\end{split}
\end{align}
and observe that for any given $M>0$ one can choose $J=J(M)$ so that  $U^\eps_a$ satisfies
\begin{align}\label{f6}
\begin{split}
&D_{x_2}U^\eps_a+A_0(D_t+\frac{\sigma_l}{\eps}D_{\theta_0})U^\eps_a+A_1(D_{x_1}+\frac{\eta_l}{\eps}D_{\theta_0})U^\eps_a-ie^{i\left(\frac{\omega_3(\beta_l)}{\eps}x_2+\theta_0\right)}B_2^{-1}MU^\eps_a=\eps^MR^\eps_M(t,x,\theta_0)\\
&BU^\eps_a= \eps g(t,x_1)e^{i\theta_0}+\eps^Mr^\eps_M(t,x_1,\theta_0)\text{ on }x_2=0\\
&U_a=0\text{ in }t<0.
\end{split}
\end{align}
Here (with obvious notation) the error terms satisfy for any $\alpha$:
\begin{align}
|(\eps\partial_{x_2},\partial_{t,x_1,\theta_0})^\alpha R_M|_{L^2_T(t,x,\theta_0)}\leq C_\alpha,\;\;|(\partial_{t,x_1,\theta_0})^\alpha r_M|_{L^2_T(t,x_1,\theta_0)}\leq C_\alpha.
\end{align}

Next introduce the higher norm:
\begin{align}
|V|_{0,m}^2:=\int^\infty_0|V(t,x_1,x_2,\theta_0)|^2_{H^m(t,x_1,\theta_0)}dx_2 \text{ for }m\in\NN.
\end{align}
We showed in section 6 of \cite{W3} that the estimate \eqref{tt28a} of Theorem \ref{tt26} can be upgraded to a higher derivative estimate:    there exist constants $K$, $\gamma_0$ independent of $\eps$ such that for $\gamma\geq \gamma_0$ and $U$ as in Theorem \ref{tt26}:
\begin{align}\label{f7}
|U^\gamma|_{0,m}+|\partial_{x_2}U^\gamma|_{0,m}\leq \frac{K}{(\eps\gamma)^\EE}\left[\frac{1}{\eps\gamma^2}|\Lambda_D F^\gamma|_{0,m+1}+\frac{1}{\eps\gamma^{3/2}}\langle\Lambda_D G^\gamma\rangle_{m+1}\right],
\end{align}
where $\EE$ is now computed from \eqref{E} to be $\EE=1$.\footnote{In \eqref{E} we now have $P=1$, $|\cI|=2$, $|\cO|=1$, $|\Upsilon_0|=2$.}
For $m>\frac{3}{2}$ the left side of \eqref{f7} dominates $|U^\gamma|_{L^\infty(\Omega_T\times\TT)}$, so
the estimate \eqref{f7} can be applied directly to the error problem satisfied by $E^\eps=U^\eps-U^\eps_a$ to show that $E^\eps$ is $O(\eps^Q)$ in $L^\infty(\Omega_T\times \TT)$, provided $M$ in \eqref{f6} is large enough (see Remark \ref{opt}, (3)).   This implies the estimate \eqref{Q}.

\textbf{3. Part (c). }This part follows from an examination of steps \textbf{3-5} of the profile construction in section \ref{construction}.  Equation \eqref{e1aa} implies $a_{0,1}\neq 0$ in $t>0$ as long as $b\cdot g_1\neq 0$ in $t>0$.\footnote{When we say here  that  ``$a_{0,1} \neq 0$ in $t>0$", we mean that $a_{0,1}$ takes nonzero values for arbitrarily small $t>0$; the same applies to other functions.} Then \eqref{ee1} implies $\sigma^2_{0,1}\neq 0$ in $t>0$. Next \eqref{ee2} implies $\sigma^1_{0,2}\neq 0$ in $t>0$   as long as $c^1_0=\ell_1Mr_2\neq 0$.   Then \eqref{e2a} implies $a_{-1,2}\neq 0$ in $t>0$, and finally \eqref{e0} implies $\sigma^m_{-1,2}\neq 0$ in $t>0$  for $m=2,3$.

\begin{rem}\label{opt}
 \textbf{1. Optimality of estimates.} Estimates for which $\EE=0$ (for example, when $|\cI|=1$,  $|\Upsilon^0|=2$ as in Theorem \ref{tvv29}) are clearly optimal.    Simple examples show that amplification due to the factors $|X_k|$ is unavoidable in WR problems \cite{CG}. 
 
 The optimality of the estimate \eqref{tt28}  for the singular problem corresponding to  \eqref{d1} (with $\EE=1$ given by \eqref{E}) is confirmed by the observation that $U^\eps_a$ as in \eqref{f5} satisfies 
 \begin{align}\label{f8}
 |U^\eps_a(t,x,\theta_0)|_{L^2(\Omega_T\times\TT)}=O\left(\frac{1}{\eps}\right),
 \end{align}
  and so $U^\eps_a$ is amplified by the factor $\frac{1}{\eps^2}$ relative to the $O(\eps)$ boundary data of \eqref{d1}.   In \eqref{tt28} one factor of $\frac{1}{\eps}$ is contributed by the factor $|X_1|$ on $\widehat{\eps g_1}$, and a second factor is contributed by $\frac{1}{(\eps\gamma)^\EE}$, so the estimate ``predicts" exactly the order of amplification exhibited by $U^\eps_a$, and hence also by the exact solution.   

 \textbf{2. Triple and higher amplification}. Step \textbf{3} of the above proof shows instantaneous production of a (nonzero) incoming 2-mode $\frac{1}{\eps}\sigma^2_{-1,2}(t,x)e^{i2\theta_2}$ in the leading term, $U_{-1},$  of the approximate solution.  In order to achieve another amplification one could modify the oscillatory coefficient in \eqref{d1} to be $(e^{i\theta_3}+e^{i2\theta_3})|_{\theta_3=\phi_3/\eps}$.  The resonance $2\phi_2+2\phi_3=4\phi_1$ in the interaction of the modified coefficient with the above 2-mode would then be expected, by arguments parallel to steps \textbf{3-5} in the above proof, to produce an outgoing $\frac{1}{\eps}\sigma^1_{-1,4}e^{i4\theta_1}$ mode, and then a ``reflected" incoming $\frac{1}{\eps^2}\sigma^2_{-2,4}e^{i4\theta_2}$ mode in $U_{-2}$.   Similarly, adding the term $e^{i4\theta_3}$ to the oscillatory coefficient should result in production of an incoming $\frac{1}{\eps^3}\sigma^2_{-3,8}$ mode in $U_{-3}$, and so on.  With each added term in the oscillatory coefficient, the formula \eqref{E} shows that $\EE$ increases by one.  Thus, these higher interactions would  exhibit the optimality of \eqref{tt28} with $\EE$  given by \eqref{E} for larger values of $P$, at least in problems where $|\cO|=1$.

Finally, note that since $\omega_2+\omega_3=2\omega_1$,  we have $\Omega_{1,2}=\frac{\omega_1-\omega_3}{\omega_2-\omega_1}=1$, so the second possibility in the hypothesis \eqref{tt9} of Proposition \ref{tt8} holds here.    That proposition, taken together with this discussion of multiple amplification, indicates that when $\Omega_{i,j}$ is rational and lies in $(0,\infty)$ or $(-\infty,-1)$, there is no hope of proving an estimate like \eqref{tt28} with finite $\EE$ for problems where the spectrum of $d(\theta_{in})$ is an arbitrary infinite subset of $\NN$.

\textbf{3. $\gamma_0$ independent of $\eps$. }  In step \textbf{2} of the above proof, to conclude $|E^\eps|_{L^\infty(\Omega_T\times\TT)}=O(\eps^Q)$ we apply \eqref{f7} with $\gamma=\gamma_0$ to the problem satisfied by $E^\eps$, use
\begin{align}
 e^{-\gamma_0T}|E^\eps|_{L^\infty(\Omega_T\times \TT)}\leq |e^{-\gamma_0 t} E^\eps|_{0,m}+ |e^{-\gamma_0 t} \partial_{x_2}E^\eps|_{0,m}
\end{align}
and then multiply both sides of the resulting estimate by $e^{\gamma_0T}$. Here we see  it is crucial that $\gamma_0\sim 1$.  If one tried to use the methods of \cite{CGW3} to estimate $E^\eps$, one would have to take $\gamma_0\sim \frac{1}{\eps}$,  and the factor $e^{\gamma_0T}$ would then overwhelm the terms $\eps^MR^\eps_M$ and $\eps^M r^\eps_M$ coming from the right side of \eqref{f7}.  A similar problem arises if one tries to use the estimates of \cite{C} to estimate $|u^\eps-u^\eps_a|_{L^\infty(\Omega_T)}$ for each fixed $\eps$.

\end{rem}

\section{Discussion} \label{disc}\qquad Let us  assess what our results suggest about the prospects of proving uniform estimates (Remark \ref{unif}) for problems like \eqref{i3} obtained as linearizations of quasilinear problems.    This paper has dealt only with wavetrain (as opposed to pulse) solutions,  so we restrict our comments to such solutions.

 Whenever the oscillatory function $v$ in the coefficients of  \eqref{i3} is real-valued,  one cannot avoid two-sided cascades, since such functions have both positive and negative Fourier spectrum.   Our estimate for problem \eqref{i1} in the two-sided case, Theorem \ref{tvv29}, requires that 
$\Upsilon_0=\{\pm\beta_l\}$ and that there is only one incoming phase, $\cI=\{N\}$.    
This result offers some hope of proving a similar estimate for (linearizations of) quasilinear problems that satisfy these two conditions.

       There are at least two reasons to study one-sided cascades.    First, this is the simplest context in which to observe the interesting phenomenon of multiple amplification.   A second, related, reason is that the occurrence of multiple amplification allows us to \emph{confirm} that there are situations in which some of the global amplification factors  $\DD(\eps,k,k-r)(\zeta)$ are indeed large, that is, equal to $\frac{C_5r^2}{\eps\gamma}$, on $\zeta-$sets of positive measure.   The confirmation is provided by having explicit, multiply amplified, approximate solutions that we \emph{know} are close in the sense of Theorem \ref{multiamp} to exact solutions $u^\eps(t,x)=U^\eps(t,x,\frac{\phi_0}{\eps})$.   Since $U^\eps$ satisfies an estimate of the form \eqref{tt28}, we conclude that the exponent $\EE$ in that estimate must be $\geq 1$, and thus some factors $\DD(\eps,k,k-r)$ must be large.  This information has a direct bearing on problems with two-sided cascades, since the \emph{same} amplification factors occur in those problems.  Indeed, the results of section \ref{1s} show that the factors $\DD(\eps,k,k-r)$ are determined just by our assumptions on $(L(\partial),B)$; they are independent of the choice of the oscillatory factor $\cD(\theta_{in})$.    We saw in section \ref{2sc} that the presence of even a single large amplification factor in a problem with two-sided cascades is expected to rule out any estimate of the form \eqref{tt28} with finite $\EE$.  

Finally, recall that our examples of multiple amplification assumed the existence of a resonance for which $\Omega_{i,j}\in ((-\infty,-1)\cup (0,\infty))\cap \QQ$.   When such resonances are absent, our results leave open the possibility of proving estimates like that of Theorem \ref{tvv29} for more general problems  with two-sided cascades (including, for starters,  the problem  \eqref{i1} when $\Upsilon=\{\beta_l,-\beta_l\}$, but  $|\cI|\geq 2$).\footnote{The linearized vortex sheet problem satisfies part of the assumption \ref{tv30}, namely $|\cI|=1$, but $|\Upsilon^0|>2$.}  But to do that by the methods of this paper,  one must show that \emph{no} amplification factor $\DD(\eps,k,k-r)$ is large.

\bibliographystyle{alpha}
\bibliography{cascades}

\end{document}